\documentclass[a4paper, reqno]{amsart}
\usepackage{amsmath}
\usepackage{eucal}
\usepackage{amssymb}
\usepackage{amsthm}
\usepackage{mathrsfs}
\usepackage{color}
\usepackage{fullpage}

\makeatletter
\@addtoreset{equation}{section}
\makeatother

\renewcommand\thefigure{\thesection.\@arabic\c@figure}
\renewcommand\thetable{\thesection.\@arabic\c@table}
\newtheorem{theorem}{Theorem}[section]

\newtheorem{lemma}[theorem]{Lemma}
\newtheorem{proposition}[theorem]{Proposition}
\newtheorem{corollary}[theorem]{Corollary}

\newtheorem{remark}[theorem]{Remark} 

\newcommand{\ms}[1]{{\mathscr #1}}
\newcommand{\mc}[1]{{\mathcal #1}}
\newcommand{\mf}[1]{{\mathfrak #1}}
\newcommand{\mb}[1]{{\mathbf #1}}
\newcommand{\bb}[1]{{\mathbb #1}}
\newcommand{\X}{\bb{X}}
\newcommand{\HH}{\bb{H}}

\renewcommand{\Cap}{{\rm cap}}

\DeclareMathOperator{\diam}{diam}

\DeclareMathOperator{\dist}{dist}
\DeclareMathOperator{\Var}{Var}

\def\clap#1{\hbox to 0pt{\hss#1\hss}}

\begin{document}
\title{Universality of trap models in the ergodic time scale}

\author{M. Jara, C. Landim, A. Teixeira}

\begin{abstract}
  Consider a sequence of possibly random graphs $G_N=(V_N, E_N)$,
  $N\ge 1$, whose vertices's have i.i.d. weights $\{W^N_x : x\in V_N\}$
  with a distribution belonging to the basin of attraction of an
  $\alpha$-stable law, $0<\alpha<1$. Let $X^N_t$, $t \ge 0$, be a
  continuous time simple random walk on $G_N$ which waits a
  \emph{mean} $W^N_x$ exponential time at each vertex $x$. Under
  considerably general hypotheses, we prove that in the ergodic time
  scale this trap model converges in an appropriate topology to a
  $K$-process. We apply this result to a class of graphs which
  includes the hypercube, the $d$-dimensional torus, $d\ge 2$, random
  $d$-regular graphs and the largest component of super-critical
  Erd\"os-R\'enyi random graphs.
\end{abstract} 

\address{\noindent IMPA, Estrada Dona Castorina 110, CEP 22460 Rio de
  Janeiro, Brasil
\newline
e-mail: \rm \texttt{mjara@impa.br}}

\address{\noindent IMPA, Estrada Dona Castorina 110, CEP 22460 Rio de
  Janeiro, Brasil and CNRS UMR 6085, Universit\'e de Rouen, Avenue de
  l'Universit\'e, BP.12, Technop\^ole du Madril\-let, F76801
  Saint-\'Etienne-du-Rouvray, France.  \newline e-mail: \rm
  \texttt{landim@impa.br} }

\address{\noindent IMPA, Estrada Dona Castorina 110, CEP 22460 Rio de
  Janeiro, Brasil  \newline e-mail: \rm
  \texttt{augusto@impa.br} }


\maketitle

\section{Introduction}
\label{sec0}

Trap models were introduced to investigate aging, a nonequilibrium
phenomenon of considerable physical interest \cite{lsnb1, bou1,
  ck1}. These trap models are defined as follows: consider an
unoriented graph $G = (V,E)$ with finite degrees and a sequence of
i.i.d.\! strictly positive random variables $\{W_z : z\in V\}$ indexed
by the vertices.  Let $\{X_t : t\ge 0\}$ be a continuous-time random
walk on $V$ which waits a mean $W_z$ exponential time at site $z$, at
the end of which it jumps to one of its neighbors with uniform
probability.

The expected time spent by the random walk on a vertex $z$ is
proportional to the value of $W_z$. It is thus natural to regard the
environment $W$ as a landscape of valleys or traps with depth given by
the values of the random variables $\{W_z : z\in V\}$.  As the random
walk evolves, it explores the random landscape, finding deeper and
deeper traps, and aging appears as a consequence of the longer and
longer times the process remains at the same vertex.

Since \cite{bou1}, aging has been observed in short time scales in
many trap models \cite{bbg, bc4, bcm1, bc1, bbc}.  The investigation
of trap models in the time scale in which the deepest traps are
visited was started in \cite{fin, fm1}, where the authors examined the
asymptotic behavior of the random walk among traps in $\bb Z$ and in
the complete graph. In the first case the random walk converges to a
degenerated diffusion, while in the second it converges to the
$K$-process, a continuous-time, Markov dynamics on $\bb N$ which hits
any finite subset $A$ of $\bb N$ with uniform distribution. This
latter result was extended in \cite{fli1} to the hypercube and in
\cite{jlt1} to the $d$-dimensional torus, $d\ge 2$.

In the present paper, we exhibit simple conditions that imply the
convergence to the $K$-process in the scaling limit. Our conditions
are general enough to include the hypercube and the torus, as well as
random $d$-regular graphs and the largest component of the
super-critical Erd\"os-R\'enyi random graphs. These are good examples
to keep in mind throughout the text.

Let $\{G_N : N\ge 1\}$, $G_N = (V_N, E_N)$, be a sequence of possibly
random, finite, connected graphs defined on a probability space
$(\Omega, \mc F, \bb P)$, where $V_N$ represents the set of vertices
and $E_N$ the set of unoriented edges. Assume that the number of
vertices, $|V_N|$, converges to $+\infty$ in $\bb P$-probability.

Assume that on the same probability space $(\Omega, \mc F, \bb P)$, we
are given an i.i.d collection of random variables $\{W^N_j : j\ge
1\}$, $N \geq 1$, independent of the random graph $G_N$ and whose
common distribution belongs to the basin of attraction of an
$\alpha$-stable law, $0<\alpha <1$. Hence, for all $N\ge 1$ and $j\ge
1$,
\begin{equation}
\label{13}
\bb P[W^N_1>t] \;=\; \frac{L(t)}{t^\alpha}\;, \quad t>0\;,
\end{equation}
where $L$ is a slowly varying function at infinity. 

For each $N\ge 1$, re-enumerate in decreasing order the weights
$W^N_1, \dots , W^N_{|V_N|}$: $\hat W^N_j = W^N_{\sigma(j)}$, $1\le
j\le |V_N|$ for some permutation $\sigma$ of the set $\{1, \dots,
|V_N|\}$ and $\hat W^N_j \ge \hat W^N_{j+1}$ for $1\le j < |V_N|$. Let
$(x^N_1, \dots, x^N_{|V_N|})$ be a random enumeration of the vertices
of $G_N$ and define $W^N_{x^N_j} = \hat W^N_j$, $1\le j\le |V_N|$,
turning $G_N = (V_N, E_N, W^N)$ into a finite, connected,
vertex-weighted graph.

Consider for each $N \geq 1$, a continuous-time random walk
$\{X^N_t : t\ge 0\}$ on $V_N$, which waits a mean $W^N_x$
exponential time at site $x$, after which it jumps to one of its neighbors
with uniform probability. The generator $\mc L_N$ of this walk is given by:
\begin{equation}
\label{f05}
(\mc L_N f)(x) \;=\; \frac 1{\deg(x)} \, \frac 1{W^N_x} \, 
\sum_{y\sim x} [f(y) - f(x)]
\end{equation}
for every $f:V_N\to\bb R$, where $y\sim x$ means that $\{x,y\}$
belongs to the set of edges $E_N$ and where $\deg(x)$ stands for the
degree of $x$: $\deg(x) = \# \{y\in V_N : y\sim x\}$. 

\medskip\noindent{\bf Heuristics.}  The main results of this article
assert that, under fairly general conditions on the graph sequence
$G_N$, the random walk $X^N_t$ converges in the ergodic time scale to
a $K$-process. Let us now give an informal description of the above
statement.

Given the graph sequence $G_N$ and the associated weights $W^N_x$,
suppose that 
\begin{enumerate}
\item A small number of sites supports most of the stationary measure
  of the process $X^N_t$, see \eqref{B0},
\end{enumerate}
and that we are able to find a sequence $\ell_N$ satisfying the
following conditions:
\begin{enumerate}
\setcounter{enumi}{1}
\item the ball $B(x, \ell_N)$ around a typical point $x$ has a volume
  much smaller than $|V_N|$, see \eqref{BB0},
\item starting outside of the above ball, the random walk `mixes'
  before hitting its center $x$, see \eqref{BB1} and
\item the graphs $G_N$ are transitive (or satisfy the much weaker
  hypothesis \eqref{e:c4}).
\end{enumerate}

Under the above conditions, we are able to show that
\begin{equation}
\label{e:main}
\text{$X^N_t$ converges to a  $K$-process,}
\end{equation}
introduced in \cite{fm1, p1}, after proper scaling, see
Theorems~\ref{t:trans} and \ref{t:random}.

Still on a heuristic level, let us give a brief explanation of why the
above conditions should imply the stated convergence. Let $M_N$ be a
sequence of integers converging to $+\infty$ slowly enough for the
balls $B(x^N_j, \ell_N)$, $1\le j \le M_N$, to be disjoint. We call
the vertices $\{x_1^N,\dots , x_{M_N}^N\}$ the {\em deep traps} and
the remaining vertices $\{x_{M_N+1}^N,\dots ,x_{|V_N|}^N\}$ the {\em
  shallow traps}. The idea is to decompose the trajectory of the
random walk in excursions between the successive visits to the balls
$B(x^N_j, \ell_N)$.

Denote by $v_{\ell_N}(x^N_j)$ the escape probability from
$x^N_j$. This is the probability that the random walk $X^N_t$ starting
from $x^N_j$ attains the boundary of the ball $B(x^N_j,\ell_N)$ before
returning to $x^N_j$. The random walk $X^N_t$ starting from $x^N_j$
visits $x^N_j$ on average $v_{\ell_N}(x^N_j)^{-1}$ times before it
escapes. After escaping, it mixes and then it reaches a new deep trap
with a distribution determined by the topology of the graph. This
distribution does not depend on the last deep trap visited because the
process has mixed before reaching the next trap. In an excursion between
two deep traps, the random walk visits only shallow traps, which
should not influence the asymptotic behaviour.

Hence, if the escape probabilities and the degrees of the random graph
have a reasonable asymptotic behavior, see \eqref{e:c4}, we expect the
random walk $X^N_t$ to evolve as a Markov process on $\{1, \dots,
M_N\}$ which waits at site $j$ a mean $W^N_{x^N_{j}}
v_{\ell_N}(x^N_j)^{-1}$ exponential time, at the end of which it jumps
to a point in $\{1, \dots, M_N\}$ whose distribution does not depend
on $j$. This latter process can be easily shown to converge to the
$K$-process, proving the main result of this article.

There are several interesting examples of random graphs which are not
considered in this article. Either because the assumptions
\eqref{B0}--\eqref{e:c4} fail or because they have not been proved
yet. We leave as open problems the asymptotic behavior of a random walk
among random traps on uniform trees on $N$ vertices, on the critical
component of an Erd\"os-R\'enyi graph, on Sierpinski carpets, on the
giant component of the percolation cluster on a torus or on the
invasion percolation cluster. \smallskip

The article is organized as follows. In the next section we give a
precise statement of our main results. In the following two sections
we present some preliminary results on hitting probabilities and
holding times of a random walk among random traps. In section
\ref{sec4} we present the topology in which the convergence to the
$K$-process takes place and in Section \ref{sec6} we construct a
coupling between the random walk and a Markov process on the set $\{1,
\dots, M\}$.  This latter process can be seen as the trace of the
$K$-process on the set $\{1, \dots, M\}$ and the coupling as the main
step of the proof. In Section \ref{sec5} we show that this latter
process converges to the $K$-process. Putting together the assertions
of Sections \ref{sec4}, \ref{sec6}, \ref{sec5} we derive in Section
\ref{sec8} a result which provides sufficient conditions for the
convergence to the $K$-process of a sequence of random walks among
random traps on deterministic graphs. We adapt this result in Section
\ref{sec7} to random pseudo-transitive graphs and in Section
\ref{sec11} to graphs with asymptotically random conductances. We show
in Section \ref{sec10} that this latter class includes the
largest component of a super-critical Erd\"os-R\'enyi graphs.

\section{Notation and results}
\label{sec1}

Recall the notation introduced in the previous section up to the
subsection Heuristics.  Denote by $\nu_N$ the unique stationary
distribution of the process $\{X^N_t: t\ge 0\}$. An elementary
computation shows that $\nu_N$ is in fact reversible and given by
\begin{equation}
\label{nu}
\nu_N(x) \;=\; \frac{\deg(x) W^N_x}{Z_N}\;,\quad x\in V_N\;,
\end{equation} 
where $Z_N$ is the normalizing constant $Z_N = \sum_{y\in V_N}
\deg(y) W^N_y$.

For a fixed graph $G_N$ and a fixed environment $\mb W=\{W^N_z : z\in
V_N\}$, denote by $\mb P^N_x=\mb P^{G_N,\mb W}_x$, $x\in V_N$, the
probability on the path space $D(\bb R_+, V_N)$ induced by the Markov
process $\{X^N_t : t\ge 0\}$ starting from $x$. Expectation with
respect to $\mb P^N_x$ is represented by $\mb E^N_x$. We denote
sometimes $X^N_t$ by $X^N(t)$ to avoid small characters.

Let $\{\X^N_n : n \geq 0\}$ be the lazy embedded discrete-time chain
in $X^N_t$, i.e., the discrete-time Markov chain which jumps from $x$
to $y$ with probability $(1/2) \deg (x)^{-1}$ if $y\sim x$ and which
jumps from $x$ to $x$ with probability $(1/2)$.  Denote by $\pi_N$ the
unique stationary, in fact reversible, distribution of the skeleton
chain, given by
\begin{equation}
\label{00}
\pi_N(x) \;=\; \frac {\deg(x)}{\sum_{y \in V_N} \deg(y)}\;\cdot
\end{equation}

For a subset $B$ of $V_N$, we denote by $H_B$ the hitting time of $B$
and by $H^+_B$ the return time to $B$:
\begin{equation*}
\begin{array}{l}
H_B \;=\; \inf \big\{t\ge 0: X^N_t \in B \big\}\;, \\
H^+_B \;=\; \inf \big\{ t\ge 0: X^N_t \in B 
\text{ and } \exists s<t \text{ s.t. }  X^N_s\not\in B \big\}\;.
\end{array}
\end{equation*}
When $B$ is a singleton $\{x\}$, we denote $H_B$, $H^+_B$ by $H_x$,
$H^+_x$, respectively. We also write $\bb H_B$ (resp. $\HH^+_B$) for
the hitting time of a set $B$ (resp. return time to $B$) for the
discrete chain $\X^N_n$.

\smallskip
\noindent{\bf $K$-processes.} To describe the asymptotic behavior of
the random walk $X^N_t$, consider two sequences of positive real
numbers $\mb u = \{u_k : k \in \mathbb{N}\}$ and $\mb Z = \{Z_k : k
\in \mathbb{N}\}$ such that
\begin{equation}
\label{07}
\sum_{k \in \mathbb{N}} Z_k \, u_k \;<\; \infty \;, \quad
\sum_{k \in \mathbb{N}} u_k \;=\; \infty\;.
\end{equation}

Consider the set $\overline{\mathbb{N}} = \bb N \cup \{\infty\}$ of
non-negative integers with an extra point denoted by $\infty$.  We
endow this set with the metric induced by the isometry
$\phi:\overline{\mathbb{N}} \to \mathbb{R}$, which sends $n \in
\overline{\mathbb{N}}$ to $1/n$ and $\infty$ to $0$. This makes the
set $\overline{\mathbb{N}}$ into a compact metric space.

In Section \ref{sec5}, based on \cite{fm1}, we construct a Markov
process on $\overline{\mathbb{N}}$, called the $K$-process with
parameter $(Z_k , u_k)$ which can be informally described as
follows. Being at $k\in \bb N$, the process waits a mean $Z_k$
exponential time, at the end of which it jumps to
$\infty$. Immediately after jumping to $\infty$, the process returns
to $\bb N$. The hitting time of any finite subset $A$ of $\bb N$ is
almost surely finite. Moreover, for each fixed $n \ge 1$, the
probability that the process hits the set $\{1, \dots, n\}$ at the
state $k$ is equal to $u_k/\sum_{1\le j\le n} u_j$. In particular, the
trace of the $K$-process on the set $\{1, \dots, n\}$ is the Markov
process which waits at $k$ a mean $Z_k$ exponential time at the end of
which it jumps to $j$ with probability $u_j/\sum_{1\le i\le n} u_i$.

\smallskip
\noindent{\bf Topology.}
Between two successive sojourns in deep traps, the random walk $X^N_t$
visits in a short time interval several shallow traps. If we want to
prove the convergence of the process $X^N_t$ to a process which visits
only the deep traps, we need to consider a topology which disregard
short excursions. With this in mind, we introduce the following topology.

Fix $T>0$. For any function $f: [0,T] \to \bb R$ and any point $t \in
[0,T]$, we say that $f$ is locally constant at $t$ if $f$ is constant
in a neighborhood of $t$. Let
\begin{equation}
\label{ec2p}
\mc C(f) = \{t \in [0,T]; f \text{ is locally constant in } t\},
\end{equation}
and $\mc D(f) = \mc C(f)^c$.  Notice that the set $\mc D(f)$ is always
closed. Let $\Lambda$ denote the Lebesgue measure in $[0,T]$ and denote by
$M_0$ the space of functions which are locally constant a.e., that is 
\begin{equation}
\label{ec3p}
M_0 \;:=\; \{f: [0,T] \to \bb R; \Lambda\big(\mc D(f)\big)=0\}\;.
\end{equation}

We say that two locally constant functions $f$ and $g\in M_0$ are
equivalent, $f \sim g$, if $f(t) =g(t)$ for any $t \notin \mc D(f)
\cup \mc D(g)$. Note that if $f \sim g$ then $f=g$ almost everywhere.
We show in Lemma \ref{l1} below that $\sim$ is an equivalence relation
in $M_0$.

Let $M = M_0/\sim$ and make the space $M$ into a metric space by
introducing the distance
\begin{equation}
\label{c02p}
d_T(f,g) \;=\; \inf_{A \in \mc B} \big\{\|f-g\|_{\infty,A^c} +
\Lambda(A)\big\}\;, 
\end{equation}
where $\mc B = \mc B([0,T])$ is the set of Borel subsets of $[0,T]$,
and $\|f-g\|_{\infty,A^c}$ stands for the supremum norm of $f-g$
restricted to $A^c$. Intuitively speaking, the distance between $f$
and $g$ is small if they are close to each other, except for a set of
small measure.

We prove in Lemma \ref{l2} that $d_T$ is well defined and that it
introduces a metric in $M$. With this metric, $M$ is separable but not
complete. 

\smallskip
\noindent{\bf Main result.} Let $\bb V=\bb V_N = |V_N|$ and let
$\Psi_N: V_N \to \{1, \dots, \bb V_N\}$ be the random function defined
by $\Psi_N(x^N_j)=j$. The first main result of this article relies on
three assumptions.  We first require the sequence of invariant
measures $\nu_N$ to be almost surely tight.  Assume that for any
increasing sequence $J_N$, with $\lim_N J_N=\infty$,
\begin{equation}
\label{B0}
\tag{\bf B0}
\lim_{N\to\infty}  \bb E \big[ \nu_N(\{x^N_1, \dots, 
x^N_{\min\{J_N, \bb V_N\}}\}^c) \big] \;=\; 0\;. 
\end{equation}

Denote by $B(x,\ell)$ the ball of radius $\ell$ centered at $x\in V_N$
with respect to the graph distance $d=d_N$ in $G_N$. Fix a sequence
$\{\ell_N : N\ge 1\}$ of positive numbers, representing the radius of
balls we place around each deep trap. Let $\mf x$ be a vertex chosen
uniformly among the vertices of $V_N$. We assume that
\begin{equation}
\tag{{\bf B1}}
\label{BB0}
\lim_{N \to \infty} \bb E \Big[ \frac{|B(\mf x, 2
  \ell_N)|}{\bb V_N}\Big]\; = \;0 \;.
\end{equation}
It follows from this condition that the number of vertices $\bb V_N$
of the graph $G_N$ diverges in probability: 
\begin{equation*}
\lim_{N \to \infty} \bb P \big[ \bb V_N \ge K \big]\; = \;0 
\end{equation*}
for every $K\ge 1$.

Let $\Vert \mu - \nu \Vert_{TV}$ be the total variation distance
between two probability measures $\mu$, $\nu$ defined on $V_N$, and
let $t_{\text{mix}} = t^N_{\text{mix}}$ be the mixing time of the
discrete chain $\{\X^N_n : n\ge 0\}$, see equation (4.33) in
\cite{LPW09}.

We assume that the typical point $\mf x$ is not hit before the mixing
time if one starts the random walk at distance at least $\ell_N$ from
$\mf x$. More precisely, we suppose that there exists an increasing
sequence $L_N$, $\lim_{N\to\infty} L_N =\infty$, such that
\begin{equation}
\tag{{\bf B2}}
\label{BB1}
\lim_{N \to \infty} \bb E\Big [
\sup_{\substack{y \not \in B(\mf x, \ell_N)}} 
\mb P_y \big[\bb H_{\mf x} \leq L_Nt_{\text{mix}} \big] \, \Big] 
\;=\;0 \;.
\end{equation}

We finally introduce the notion of pseudo-transitive graphs, which
includes the classical definition of transitive graphs but also
encompasses other important examples such as random regular graphs,
discussed in Proposition~\ref{p:dreg}. 

Consider a sequence of possibly random graphs $G_N = (V_N, E_N)$. We
say that two subsets $A$, $B$ of $V_N$ with distinguished vertices
$\mf x\in A$, $\mf y\in B$, are isomorphic, $(\mf x, A)\equiv (\mf y,
B)$, if there exists a bijection $\varphi:A\to B$ with the property
that $\varphi(\mf x) = \mf y$ and that for any $a,b \in A$, $\{a,b\}$
is an edge of $G_N$ if and only if $\{\varphi(a), \varphi(b)\}$ is an
edge of $G_N$.

Let $\mf x$, $\mf y\in V_N$ be two vertices chosen independently and
uniformly in $V_N$. We say that $G_N$ is pseudo-transitive for the
sequence $\ell_N$, if
\begin{equation}
\label{pseudo}
\lim_{N\to\infty} \bb P \big[ 
(\mf x, B(\mf x,\ell_N)) \not \equiv (\mf y, B(\mf y,\ell_N)) 
\big] \;=\; 0\;.
\end{equation}
Clearly, any sequence of transitive graphs is pseudo-transitive for
any given sequence $\ell_N$.  

For $x\in V_N$, let $v_\ell (x) = v^N_{\ell_N}(x)$ be the probability
of escape from $x$:
\begin{equation*}
v_\ell (x) \;=\; \mb P^N_x \big[ \HH_{R(x,\ell)} < \HH^+_{x} \big]\;, 
\end{equation*}
and let $\{c_k : k\ge 1\}$ be the sequence defined by
\begin{equation}
\label{aa1}
c_k^{-1} \;=\; \inf \big\{t\ge 0 : \mathbb P[W^N_1 > t] 
\le k^{-1} \big\}\;,
\end{equation}
The constant $c_N^{-1}$ represents the typical size of $\max_{1\le
  k\le N} W^N_k$, so that $c_{\bb V} W^N_{x_j}$ for fixed $j$ is of
order one.

\begin{theorem}
\label{t:trans}
Fix a sequence of pseudo-transitive graphs $G_N$ with respect to a
sequence $\ell_N$. Suppose that \eqref{B0}--\eqref{BB1} hold and that
$\Psi_N(X^N_0)$ converges in probability to some $k\in\bb N$. Then,
letting $\beta_N^{-1}= c_{\bb V} v^N_{\ell_N}(x^N_1)$, we have that
\begin{equation*}
(c_{\bb V} \mb W^N, \Psi_N(X^N_{t \beta_N})) \quad\text{converges weakly
  to}\quad  (\mb w, K_t)\;,
\end{equation*}
where the sequence $\mb w =(w_1, w_2, \dots)$ is defined in
\eqref{10} and where for each fixed $\mb w$, $K_t$ is a $K$-process
with parameter $(\mb w,1)$ starting from $k$. In the convergence, we
adopted $L^1(\bb N)$ topology in the first coordinate and
$d_T$-topology in the second.
\end{theorem}

It is not difficult to show from the definition of the random sequence
$\mb w =(w_1, w_2, \dots)$ that $w_1$ has a Fr\'echet distribution.
In Section~\ref{sec7}, we apply Theorem \ref{t:trans} to the
hypercube, the $d$-dimensional torus, $d \geq 2$, and to a sequence of
random $d$-regular graphs, $d \geq 3$.  \medskip

The second main result of the article concerns graphs in which the
assumption \eqref{pseudo} of isometry of neighborhoods is replaced by
an asymptotic independence and a second moment bound. 

Assume that there exists a coupling $\mc Q_N$ between the random graph
$\{G_N : N\ge 1\}$ and a sequence of i.i.d random vectors $\{(D_k,
E_k) : k\ge 1\}$ such that for every $K\ge 1$ and $\delta>0$,
\begin{equation}
\tag{{\bf B3}}
\label{e:c4}
\begin{split}
& \lim_{N\to\infty} \mc Q_N \Big [\, 
\max_{1\le j\le K} \big \vert \, v_{\ell}(\mf x_j)^{-1} 
- E_j^{-1} \big\vert > \delta \Big] \;=\; 0 \;, \\
&\quad \lim_{N\to\infty} \mc Q_N \Big [\, 
\bigcup_{j=1}^K  \{ \deg(\mf x_j) \not = D_j \} \Big] \;=\; 0 \;, \\
&\qquad \mc Q_N \big[ D_1\ge 1 \,,\, 0< E_1 \le 1\big] \;=\;1\;,\quad 
E_{\mc Q_N} \big[(D_1/E_1)^2 \big] \;<\; \infty \; ,
\end{split}
\end{equation}
for one and therefore all $N\ge 1$, where $\ell=\ell_N$ is the radius
of the balls placed around each trap and introduced right above
\eqref{BB0}, and $\mf x_1, \dots, \mf x_K$ is a collection of distinct
vertices chosen uniformly in $V_N$.  We can now state our second main
result, which can be seen as a generalization of
Theorem~\ref{t:trans}.

\begin{theorem}
\label{t:random}
Fix a sequence of random graphs $G_N$. Suppose that
\eqref{B0}--\eqref{e:c4} hold and that $\Psi_N(X^N_0)$ converges in
probability to some $k\in\bb N$. Then, defining $\beta_N =
c_{\bb V}^{-1}$, we have that
\begin{equation*}
(c_{\bb V} \mb W^N, \Psi_N(X^N_{t \beta_N})) \quad\text{converges weakly
  to}\quad  (\mb w, K_t)\;,
\end{equation*}
where the sequence $\mb w =(w_1, w_2, \dots)$ is defined in
\eqref{10} and where for each fixed $\mb w$, $K_t$ is a $K$-process
starting from $k$ with parameter $(\mb Z,\mb u)$, where $Z_k =
w_k/E_k$ and $u_k = D_k E_k$.  In the convergence, we adopted $L^1(\bb
N)$ topology in the first coordinate and $d_T$-topology in the second.
\end{theorem}

In Section \ref{sec10}, we apply this result to the largest component
of a super-critical Erd\"os-R\'enyi random graph. We expect this
statement to be applicable in a wider context, such as random graphs
with random degree sequences, or percolation clusters on certain
graphs.

\section{Hitting probabilities}
\label{sec2}

We prove in this section general estimates on the hitting distribution
of a random walk on a finite graph.  These estimates will be useful in
the description of the trace of our trap model on the deepest traps.
Since $N$ will be kept fixed throughout the section, we omit $N$ from
the notation almost everywhere.

Recall that we denote by $d=d_N$ the graph distance on $V_N$: $d(x,y)=
m$ if there exists a sequence $x=z_0, z_1, \dots, z_m =y$ such that
$z_{i+1} \sim z_i$ for $0\le i\le m-1$, and if there do not exist
shorter sequences with this property. For $x\in V_N$ and a subset $C$
of $V_N$, denote by $d(x,C)$ the distance from $x$ to $C$: $d(x,C)
= \min_{y\in C} d(x,y)$.  For $\ell\ge 1$, denote by $B(C,\ell)$ the
vertices at distance at most $\ell$ from $C$: $B(C,\ell) = \{x\in V_N
: d(x,C) \le \ell\}$ and let $R(C,\ell) = B(C,\ell)^c$. When the set
$C$ is a singleton $\{x\}$, we write $B(x,\ell)$, $R(x,\ell)$ for
$B(\{x\},\ell)$, $R(\{x\},\ell)$, respectively.

Fix $M\ge 1$, a subset $A = \{x_1, \dots , x_M\}$ of $V_N$ and
$\ell\ge 1$. Recall from Section \ref{sec1} that we denote by $v_\ell
(x)$, $x\in A$, the escape probability from $x$, and let $p(x, A)$ be
the probability of reaching the set $A$ at $x$, when starting at
equilibrium:
\begin{equation}
\label{f01}
v_\ell (x) \;=\; \mb P_x \big[ \HH_{R(x,\ell)} < \HH^+_{x} \big]\;, 
\quad p(x, A) \;=\; \mb P_{\pi_N} \big[ \bb X^N(\HH_{A}) = x \big]\;,
\end{equation}
where $\pi_N$ is the stationary state of the discrete-time chain $\bb
X^N_n$, introduced in \eqref{00}.

\begin{lemma}
\label{s01}
Fix a subset $A =\{x_1,\dots,x_M\}$ of $V$. For any $z\not \in A$
and for any $L\ge 1$,
\begin{equation*}
\sum_{j=1}^M \big| \mb P_z[\X_{\HH_A} = x_j] - p(x_j, A)
\big| \; \leq \; 2 \, \big( 2^{-L} \;+\; \mb P_z[\HH_A <  L t_{\text{mix}}]
\big) \;.
\end{equation*}
Moreover, if there exists $\ell \ge 1$ such that $d(x_a,x_b) > 2 \ell+1$
for $a\not = b$, then for all $L\ge 1$ and for all $1\le i\le M$,
\begin{equation*}
\sum_{j\not = i} \big| \mb P_{x_i}[\X_{\HH_A} = x_j] - v_\ell (x_i) \, p(x_j, A)
\big| \;\leq\; 2 \, v_\ell (x_i) \, \max_{z \in R(A, \ell)} 
\Big\{  2^{-L} +  \mb P_z \big[ \HH_A <  L t_{\text{mix}} \big] \Big\}.
\end{equation*}
\end{lemma}

\begin{proof}
Fix a subset $A =\{x_1,\dots,x_M\}$ of $V$ and $z\not\in A$.  By
definition of the mixing time $t_{\text{mix}}$ and by the definition
of the total variation distance, 
\begin{equation*}
\begin{split}
& \sum_{j=1}^M \Big| \, \mb E_z \big[ \mb P_{\X (L t_{\text{mix}})} 
\big[ \X_{\bb H_A} = x_j \big]\, \big] \;-\; \mb P_\pi \big[ \X_{\bb H_A} = x_j
\big] \, \Big| \\
& \qquad = \; \sum_{j=1}^M \Big| \, \sum_{w\in V} 
\Big\{ \mb P_z \big[ \X (L t_{\text{mix}}) = w \big] - \pi(w) \Big\}
\, \mb P_w \big[ \X_{\HH_A} = x_j \big] \Big| \\
& \qquad \leq\; 
2 \, \big\lVert \mb P_z[\X_{Lt_{\text{mix}}} = \cdot\,] \,-\, \pi(\cdot) 
\big\rVert_{TV} \;\leq\; 2 \cdot 2^{-L}\;. 
\end{split}
\end{equation*}

To prove the first claim of the lemma, apply the Markov property to
get that
\begin{equation*}
\mb P_z [\X_{\HH_A} = x_j] \;\leq\; \mb E_z \big[ \mb P_{\X(L t_{\text{mix}})}
[\X_{\HH_A} = x_j]\, \big] \;+\; \mb P_z \big[ \X_{\HH_A} = x_j \,,\,
\HH_A \leq L t_{\text{mix}} \big]
\end{equation*}
and that
\begin{equation*}
\begin{split}
& \mb P_z\big[ \X_{\HH_A} = x_j \big] \; \geq\; 
\mb P_z \big[ \X_{\HH_A} = x_j \,,\, \HH_A > L t_{\text{mix}} \big]  \\
& \qquad = \; \mb E_z\big[ \mb P_{\X (L t_{\text{mix}})} [\X_{\HH_A}=x_j] \, \big] 
\;-\; \mb E_z \big[ \mb P_{\X (L t_{\text{mix}})} [\X_{\HH_A}=x_j]
\,,\, \HH_A \leq L t_{\text{mix}} \big] \;.\\
\end{split}
\end{equation*}
The triangular inequality together with the previous two bounds and
the estimate presented in the beginning of the proof show that
\begin{equation*}
\sum_{j=1}^M \big| \mb P_z[\X_{\HH_A} = x_j] - \mb P_\pi [X_{\HH_A} = x_j]
\big| \; \leq \; 2\big( 2^{-L} +  \mb P_z[\HH_A <  L t_{\text{mix}}] \big)\;.
\end{equation*}
This proves the first claim of the lemma.

We turn now to the proof of the second claim of the lemma.  Since
$d(x_i, A\setminus \{x_i\}) > \ell$ and $i\not = j$, the expression
inside the absolute value on the left hand side of the inequality can
be written as
\begin{equation*}
\mb P_{x_i} \big[ \X(\HH_A) = x_j \, \big|\,  \HH_{R(x_i,\ell)} 
< \HH^+_{x_i}  \big]
\, v_\ell (x_i) \;-\; v_\ell (x_i) \, p(x_j, A)\; .
\end{equation*}
The absolute value is thus bounded by
\begin{equation*}
\sum_{z\in V} \big|\, \mb P_{z} [ \X(\HH_A) = x_j ]
- p (x_j, A) \, \big|\, 
\mb P_{x_i}  \big[ \HH_{R(x_i,\ell)} < \HH^+_{x_i} \,,\,
\X(\HH_{R(x_i,\ell)}) = z \big] \; .
\end{equation*}
Since $d(x_a,x_b)>2\ell+1$, $a\not = b$, the set of vertices $z$ at
distance $\ell+1$ from $x_i$ is disjoint from $A$. Hence, by the first
part of the proof, the sum over $j\not =i$ of this expression is
bounded above by
\begin{equation*}
2 \, v_\ell (x_i) \, \max_{z\in R(A,\ell)} \Big\{ 2^{-L} \;+\;  
\mb P_z[\HH_A <  L t_{\text{mix}}] \Big\} 
\end{equation*}
for every $L\ge 1$. This proves the lemma.
\end{proof}

Denote by $\mc D(f)$ the Dirichlet form of a function $f: V \to
\mathbb{R}$:
\begin{equation*}
\mathcal{D}(f) \;=\; \frac{1}{2} \sum_{x\in V} \sum_{y\sim x}
\frac {\nu(x)} {\deg(x) \, W_x}  (f(x) - f(y))^2 \;.
\end{equation*}
For disjoint subsets $A$ and $B$ of $V$, denote by $\Cap (A,B)$ the
capacity between $A$ and $B$:
\begin{equation*}
\Cap (A,B) \;=\; \inf_f \mc D(f)\;,
\end{equation*}
where the infimum is carried over all functions $f: V \to \bb R$
such that $f (x) =1$ for $x\in A$, $f(y) = 0$, $y\in B$. Let $g: V
\to [0,1]$ be given by
\begin{equation*}
g_{A,B}(x) \;=\; \mb P_x[H_A \leq H_B]\;.
\end{equation*}
It is a known fact that
\begin{equation}
\label{f02}
\Cap (A,B)\;=\; \mathcal{D}(g_{A,B}) 
\;=\; \sum_{y \in A} \nu(y) \, W_y^{-1}
\, \mb P_y [H_B < H^+_A]\;.
\end{equation}
Note that we may replace in the above identity $H_B$, $H^+_A$ by
$\bb H_B$, ${\bb H}^+_A$, respectively.

Take a set $A \subset V$ composed of $M$ points which are far apart
and let $x$ be a point in $A$. In the next lemma, we are going to
estimate the probability $p(x,A) = \mb P_\pi[\X_{\HH_A} = x]$.  This
probability will be roughly proportional to $\deg(x)v_\ell(x)$. Let
us first introduce a normalizing constant.  For $\ell\ge 1$ and a
finite subset $A$ of $V$, let
\begin{equation*}
\Gamma_\ell (A) \;=\; \sum_{x\in A} \deg(x)v_\ell(x)\;.
\end{equation*}

\begin{lemma}
\label{s03}
Fix a subset $A =\{x_1,\dots,x_M\}$ of $V$ such that $d(x_a,x_b) > 2
\ell +1$, $a\not = b$, for some $\ell \ge 1$. Then,
\begin{equation*}
\max_{1\le i\le M} \Big| \, p(x_i, A) - \frac {\deg (x_i) \, v_\ell(x_i)}
{\Gamma_\ell (A)} \, \Big| \;\leq\;  2\, \max_{z\in R(A,\ell)} \big\{ 2^{-L} 
+  \mb P_z[\HH_A \leq  L t_{mix}] \big\}\;.
\end{equation*}
\end{lemma}

\begin{proof}
Fix $1\le i\le M$ and let $A_i = A \setminus \{x_i\}$. Since
$\mathcal{D}(g_{\{x_i\},A_i}) = \mathcal{D} (1-g_{\{x_i\},A_i})$, by
\eqref{f02}
\begin{equation}
\label{f03}
\deg(x_i)  \mb P_{x_i} [\HH_{A_i} < \HH^+_{x_i}] 
\;=\; \sum_{j\not = i} \deg(x_j) \mb P_{x_j}[\HH_{x_i} 
< \HH^+_{A_i}] \;.
\end{equation}

On the other hand, since $d(x_i, A_i) > \ell$,
\begin{equation*}
\begin{split}
\mb P_{x_i} [\HH_{A_i} < \HH^+_{x_i}] & \;=\; 
\mb E_{x_i} \Big[ \, \mb 1 \{\HH_{R(x_i,\ell)} < \HH^+_{x_i} \} 
\, \mb P_{\X(\HH _{R(x_i,\ell)})} [\HH_{A_i} < \HH_{x_i}] \, \Big] \\
& \;=\; \mb E_{x_i} \Big[ \mb 1\{ \HH_{R(x_i,\ell)} < \HH^+_{x_i}\} \,
\big (1- \mb P_{\X (\HH_{R(x_i,\ell)})} [\X_{\HH_A} = x_i] \big) \, \Big]\;.
\end{split}
\end{equation*}
Therefore,
\begin{equation*}
\begin{split}
& \mb P_{x_i}[\HH_{A_i} < \HH^+_{x_i}]
- v_\ell(x_i) \, [1- p(x_i,A)] \\
& \quad =\; \mb E_{x_i} \Big[ 
\mb 1\{ \HH_{R(x_i,\ell)} < \HH^+_{x_i}\} \,
\Big\{ p(x_i,A)  \,-\, \mb P_{\X (\HH_{R(x_i,\ell)})} [\X_{\HH_A} = x_i ] 
\Big\} \, \Big]   \;.
\end{split}
\end{equation*}
Since $d(x_a,x_b) > 2\ell+1$, we may replace in the previous
expression $\X (\HH_{R(x_i,\ell)})$ by $\X (\HH_{R(A,\ell)})$.  By the
first assertion of Lemma~\ref{s01}, the absolute value of the
difference inside braces is less than or equal to $2\, \max_{z\in
  R(A,\ell)} \{ 2^{-L} + \mb P_z[\HH_A \leq L t_{mix}] \}$ for every
$L\ge 1$. Hence,
\begin{equation}
\label{f04}
\begin{split}
& \Big| \mb P_{x_i}[\HH_{A_i} < \HH^+_{x_i}]
- v_\ell(x_i) \, [1- p(x_i,A)]\, \Big| \\
&\quad \leq \; 2\, v_\ell(x_i)\, \max_{z\in R(A,\ell)} \big\{ 2^{-L} 
+  \mb P_z[\HH_A \leq  L t_{mix}] \big\} 
\end{split}
\end{equation}
for every $L\ge 1$.

Similarly, from \eqref{f03} one obtains that
\begin{equation*}
\begin{split}
& \deg(x_i) \mb P_{x_i} [\, \HH_{A_i} < \HH^+_{x_i} ] 
\;=\; \\
&\qquad \sum_{j\not = i} \deg(x_j) \mb E_{x_j} \Big[\, \mb 1 
\{\HH_{R(x_j,\ell)} < \HH^+_A\} \, \mb P_{\X(R(x_j,\ell))} 
[\X_{\HH_A} = x_i] \, \Big]\; .
\end{split}
\end{equation*}
It follows from this identity and the previous argument that
\begin{equation*}
\begin{split}
& \Big| \deg(x_i) \mb P_{x_i}[\, \bb H_{A_i} <  {\bb H}^+_{x_i}]  -
\sum_{j\not = i} \deg(x_j) v_\ell(x_j) p(x_i,A) \Big| \\ 
& \quad \leq\; 2\, \sum_{j\not = i} \deg(x_j) v_\ell(x_j) 
\max_{z\in R(A,\ell)} \big\{ 2^{-L} +  \mb P_z [\, \bb H_A \leq  L t_{mix}] \big\}
\end{split}
\end{equation*}
for all $L\ge 1$. 

The two previous estimates yield the bound
\begin{equation*}
\begin{split}
& \Big|  \deg(x_i) v_\ell(x_i) [1- p(x_i,A)]
- \sum_{j\not = i} \deg(x_j) v_\ell(x_j) p(x_i,A) \Big| \\
& \qquad \leq\; 2\, \sum_{j=1}^M \deg(x_j) v_\ell(x_j)
\max_{z\in R(A,\ell)} \big\{ 2^{-L} +  \mb P_z[\, \bb H_A \leq  L t_{mix}]
\big\} \; . 
\end{split}
\end{equation*}
To conclude the proof of the lemma, it remains to divide both sides of
the inequality by $\Gamma_\ell(A)$.
\end{proof}

\section{Holding times of the trace process}
\label{sec:general}

We present in this section a general result on Markov chains computing
the time spent by this chain on a subset of the state space.  This
will be useful later in proving that the time spent by the walk on the
shallow traps can be disregarded.

Consider an \emph{irreducible} continuous-time Markov process $\{X_t :
t\ge 0\}$ on a \emph{finite} state space $V$. Denote by $\{W_x : x\in
V\}$ the mean of the exponential waiting times, by $\nu$ the unique
stationary probability measure, and by $\{\tau_j : j\ge 0\}$ the
sequence of jump times.

Denote by $\mb P_x$, $x\in V$, the probability measure on the path
space $D(\bb R_+, V)$ induced by the Markov process $X_t$ starting
from $x$. Expectation with respect to $\mb P_x$ is represented by $\mb
E_x$. For a probability measure $\mu$ on $V$, let $\mb P_\mu =
\sum_{x\in V} \mu(x) \mb P_x$.

Fix a set $A \subset V$ and let $U$ be a stopping time such that for
all $x\in A$, 
\begin{equation*}
\mb P_x [\tau_1\le U]=1\;, \quad \mb P_x[H_{A\setminus \{x\}}
\ge U] =1\;, \quad \mb E_x[U] <\infty \;.
\end{equation*}
$U=H_{R(A,\ell)}$ is the example to keep in mind, where $\ell$ is
chosen so that $d(x,y)> 2\ell+1$ for all $x\not=y\in A$.  Let $S_A = U
+ H_A \circ \theta_U$ be the hitting time of the set $A$ after time
$U$. Denote by $v(x)$ the probability that starting from $x$ the
stopping time $U$ occurs before the process returns to $x$: $v(x) =
\mb P_x [ U < H^+_{x}]$, which should be understood as an escape
probability.

Let $D_k$, $k\ge 0$, be the time of the $k$-th return to $A$ after
escaping: $D_0=0$, $D_1 = S_A$, $D_{k+1} = D_k + S_A \circ
\theta_{D_k}$, $k\ge 1$. Clearly, if $X_0$ belongs to $A$, $\{X_{D_k}:
k\ge 0\}$ is a discrete time Markov chain on $A$.  On the other hand,
by assumption $\mb E_x [D_1] = \mb E_x [U + H_A \circ \theta_U]$ is
finite.

\begin{lemma}
\label{t02}
The Markov chain $\{X_{D_k} : k\ge 0\}$ is irreducible. Moreover, for
every $f: V \to \bb R$,
\begin{equation*}
\lim_{k\to\infty} \frac 1k \int_0^{D_k} f(X_t) \, dt 
\;=\; \sum_{z\in A} \rho(z) \, \mb E_z \Big[ \int_0^{D_1} f(X_t) 
\, dt \Big]
\end{equation*}
$\mb P_{\nu}$-almost surely, where $\rho$ is the unique stationary
state of the discrete time chain $\{X_{D_k} : k\ge 0\}$.
\end{lemma}

\begin{proof}
We first prove the irreducibility of the chain $\{X_{D_k} : k\ge 0\}$.
Fix $x$, $y\in A$ and consider a self-avoiding path $x_0=x, \dots, x_n
= y$ such that the discrete-time Markov chain associated to the Markov
process $X_t$ jumps from $x_i$ to $x_{i+1}$, $0\le i <n$, with
positive probability. Such path exists by the irreducibility of
$X_t$. Let $x_j$ be the first state in the sequence $x_1, \dots, x_n$
which belongs to $A$. Since $\mb P_x[H_{A\setminus \{x\}} \ge U] =1$,
\begin{equation*}
\begin{split}
& \mb P_x \big[ X_{D_1} = x_j \big] \;\ge\;
\mb P_x \Big[ X_{D_1} = x_j \,,\, Z_1 =x_1, \dots, Z_j = x_j \Big] \\
&\quad \;=\; \mb P_x \Big[ X_{U + H_A \circ \theta_{U}} = x_j \,,\,
U \le H_{A\setminus \{x\}}  \,,\, Z_1 =x_1, \dots, Z_j = x_j \Big]\;, 
\end{split}
\end{equation*}
where $\{Z_n : n\ge 0\}$ is the discrete-time jump chain associated to
the process $\{X_t:t\ge 0\}$. Since $U\ge \tau_1$, on the event $\{Z_1
=x_1, \dots, Z_j = x_j\}\cap \{ U \le H_{A\setminus \{x\}}\}$, $U +
H_A \circ \theta_{U} = \tau_j$. The previous probability is thus
equal to
\begin{equation*}
\mb P_x \Big[ X_{\tau_j} = x_j \,,\, Z_1 =x_1, \dots, Z_j = x_j \Big] 
\;= \; \mb P_x \big[ Z_1 =x_1, \dots, Z_j = x_j \big]\; >\; 0\;.
\end{equation*}
Repeating this argument for the subsequent states in the sequence
$x_1, \dots, x_n$ which belong to $A$, we prove that the chain
$X_{D_k}$ is irreducible.

Fix a function $f:V\to \bb R$. Clearly,
\begin{equation*}
\frac 1k \int_0^{D_k} f(X_t) \, dt
\;=\; \frac 1k \sum_{x\in A} \sum_{j=0}^{k-1} \int_{D_j}^{D_{j+1}} 
f(X_t) \, dt \, \mb 1\{X_{D_j} =x\} \;.
\end{equation*}
For $x\in A$, let $K^x_{1} = \min \{j\ge 0 : X_{D_j}=x\}$, $K^x_{n+1}
= \min \{j> K^x_{n} : X_{D_j}=x\}$, $n\ge 1$, and let $L^x_k = \# \{ j
< k : X_{D_j}=x\}$. With this notation, we can rewrite the previous
sum as
\begin{equation*}
\frac 1k \sum_{x\in A} \sum_{n=1}^{L^x_k} \int_{D_{K^x_n}}^{D_{K^x_n +1}} 
f(X_t) \, dt \; =\;
\sum_{x\in A} \frac {L^x_k}k  \frac 1{L^x_k} 
\sum_{n=1}^{L^x_k} \int_{D_{K^x_n}}^{D_{K^x_n +1}}  f(X_t) \, dt \;.
\end{equation*}
By the irreducibility of the chain $X_{D_k}$, for each $x\in A$,
$L^x_k/k$ converges a.s. as $k\uparrow \infty$ to $\rho(x)$.
Moreover, for each $x$, the variables $\int_{[D_{K^x_n}, D_{K^x_n
    +1})} f(X_t) \, dt$, $n\ge 1$, are independent and identically
distributed.  Hence, since $L^x_k \uparrow\infty$, by the law of large
numbers, $\mb P_{\nu}$-almost surely,
\begin{equation*}
\lim_{k\to\infty}
\frac 1{L^x_k} \sum_{n=1}^{L^x_k} \int_{D_{K^x_n}}^{D_{K^x_n +1}}
f(X_t) \, dt \;=\; \mb E_{x} \Big[ \int_{0}^{D_{1}}
f(X_t) \, dt \Big]\;.
\end{equation*}
The lemma follows from the two previous convergences.
\end{proof}

\begin{proposition}
\label{t01}
The unique stationary state $\rho$ of the discrete-time Markov chain
$\{X_{D_k} : k\ge 0\}$ satisfies
\begin{equation}
\label{rho}
\rho(x) = \nu(x) \, v(x) \, W^{-1}_x \, \mb E_\rho[D_1] \;=\;
\frac{\nu(x) \, v(x) \, W^{-1}_x}{\sum_y \nu(y) \, v(y) \, W^{-1}_y}.
\end{equation}
Moreover, for every $g: V\to \bb R$,
\begin{equation}
\label{out}
\sum_{x\in A} v(x) \, \nu(x) \, W^{-1}_x \, \mb E_x \Big[ \int_0^{D_1} 
g(X_t) \,dt \Big]\;=\; \sum_{x\in V} g (x) \, \nu(x)\;. 
\end{equation}
\end{proposition}

\begin{proof}
Applying Lemma \ref{t02} to $f=1$, we obtain that $\mb P_\nu$-almost
surely 
\begin{equation}
\label{e01}
\lim_{k\to\infty} \frac {D_k}k \;=\; \lim_{k\to\infty}
\frac 1k \int_0^{D_k} dt \;=\; \mb E_\rho \big[ D_1\big]\;.
\end{equation}
By Lemma \ref{t02} with $f(y) = \mb 1\{y=x\}$, we get that $\mb
P_\nu$-almost surely
\begin{equation*}
\lim_{k\to\infty} \frac 1k \int_0^{D_k} \mb 1\{X_t =x\} \, dt 
\;=\; \rho(x) \, \mb E_x \Big[ \int_0^{D_1} \mb 1\{X_t=x\} \, dt \Big]
\end{equation*}
because starting from $y\not = x$, the process does not visit $x$
before time $D_1$. In particular, all terms on the right-hand side in
the statement of Lemma \ref{t02}, but the one $z=x$, vanish.  On the
other hand, dividing and multiplying the expression on the left-hand
side of the previous equation by $D_k$, we obtain by the ergodic
theorem and by \eqref{e01} that
\begin{equation}
\label{ED1}
\mb E_\rho \big[ D_1\big] \, \nu(x)  \;=\; \rho(x) \, 
\mb E_x \Big[ \int_0^{D_1} \mb 1\{X_t=x\} \, dt \Big]\;.
\end{equation}
The time spent at $x$ before $D_1$ is the time spent at $x$ before $U$
which is a geometric sum of independent exponential times. The success
probability of the geometric is $v(x)$ and the mean of the exponential
distributions is $W_x$. Hence, the right-hand side of the previous
formula is equal to $\rho(x) W_x/ v(x)$. This proves the first
identity in \eqref{rho}. To derive the second identity, note that $\mb
E_\rho[D_1]$ does not depend on $x$, and it is therefore only a
normalizing constant to make $\rho$ into a probability distribution.

By the ergodic theorem, for every $g:V\to \bb R$,
\begin{equation*}
\lim_{k\to\infty} \frac 1{D_k} \int_0^{D_k} g(X_t) \,
dt \;=\; \sum_{x\in V} g (x) \, \nu(x) \;.
\end{equation*}
To conclude the proof of the proposition, it remains to show that the
left hand side of this expression is equal to the left-hand side of
\eqref{out}. To this end, we will use the previous lemma.

For a function $g:V\to \bb R$, by Lemma \ref{t02} for $f=g$ and
\eqref{e01}, we get
\begin{equation*}
\lim_{k\to\infty} \frac 1{D_k} \int_0^{D_k} g(X_t) \,
dt \;=\; \lim_{k\to\infty} \frac k{D_k} \frac 1{k} \int_0^{D_k} g(X_t) \,
dt \;=\; \frac 1{\mb E_\rho \big[ D_1\big]} 
\mb E_\rho \Big[ \int_0^{D_1} g(X_t) \, dt\Big].
\end{equation*}
To conclude the proof of the proposition, it
suffices to use \eqref{rho}.
\end{proof}

\begin{corollary}
\label{corED}
We have that
\begin{equation*}
\mb E_\rho [D_1] = \frac{E_\rho \big[W_x/v(x) \big]}
{1-\nu(V \setminus A)} \;\cdot
\end{equation*}
Furthermore, for any function $g:V\to \bb R$,
\begin{equation*}
\mb E_\rho \Big[\int_0^{D_1} g(X_t)\, dt \Big] \;=\; E_\nu[g]\, 
\mb E_\rho [D_1]\;.
\end{equation*}
\end{corollary}

\begin{proof}
We can write
\begin{equation*}
\mb E_\rho [D_1] = \mb E_\rho \Big[ \int_0^{D_1} dt \Big] = 
\mb E_\rho \Big[ \int_0^{D_1} \mb 1 \{X_t \in A\} dt \Big] + 
\mb E_\rho \Big[\int_0^{D_1} \mb 1 \{X_t \not \in A\} dt \Big]\;.
\end{equation*}
By the same reasoning as below \eqref{ED1}, we conclude that the first
expectation in the sum above equals $E_\rho \big[W_x/v(x) \big]$. To
evaluate the second expectation, we use Proposition~\ref{t01} with $g
= \mb 1 \{V \setminus A\}$ to conclude that
\begin{equation*}
\mb E_\rho \Big[\int_0^{D_1} \mb 1 \{X_t \not \in A\} \, dt \Big] 
\;=\; \mb E_\rho [D_1] \, \nu(V \setminus A)\; .
\end{equation*}
Putting together the above equations, we conclude the proof of the
first assertion of the corollary. 

The second claim follows from the first identity in \eqref{rho} and
from \eqref{out}.
\end{proof}

\section{Topology}
\label{sec4}

We define in this section a topology on the space of {\em locally
constant functions}, where our main convergence will take place.

Fix $T>0$. For any non-empty set $F \subseteq
[0,T]$ and any $\delta >0$ define the ball $B(F,\delta)$ by
\begin{equation}
\label{ec1p}
B(F,\delta) \;=\; \bigcup_{t\in F} \{s \in [0,T]\,;\, |t-s| 
< \delta\} \; . 
\end{equation}

When $F$ is a singleton $\{t\}$, we simply write $B(t,\delta)$ instead
of $B(\{t\},\delta)$.  For any function $f: [0,T] \to \bb R$ and any point
$t \in [0,T]$, we say that $f$ is locally constant at $t$ if there
exists $\delta>0$ such that $f$ is constant in $B(t,\delta)$.  Define
the open set
\begin{equation}
\label{ec2}
\mc C(f) = \{t \in [0,T]\,;\, f \text{ is locally constant in } t\},
\end{equation}
and let $\mc D(f)$ be the closed set $\mc D(f) = \mc C(f)^c$. Let
$\Lambda$ denote the Lebesgue measure in $[0,T]$ and denote by $\mf M_0$ the
space of locally constant functions:
\begin{equation}
\label{ec3}
\mf M_0 \;:=\; \{f: [0,T] \to \bb R\,;\, \Lambda\big(\mc D(f)\big)=0\}\;.
\end{equation}

We say that two locally constant functions $f$ and $g\in \mf M_0$ are
equivalent, $f \sim g$, if $f(t) =g(t)$ for any $t \in \mc C(f) \cap
\mc C(g)$. Note that $f=g$ almost everywhere if $f \sim g$.

\begin{lemma}
\label{l1}
The relation $\sim$ is an equivalence relation in $\mf M_0$.
\end{lemma}
\begin{proof}
  The relation $\sim$ is clearly reflexive and symmetric. To prove
  that it is also transitive, fix three functions $f$, $g$ and $h$ in
  $\mf M_0$ such that $f \sim g$ and $g\sim h$. Fix $t\in \mc C(f)
  \cap \mc C(h)$.  Since $t$ is a point where both $f$ and $h$ are
  locally constant, $f$ and $h$ are constant in a neighborhood of
  $t$. As $f=g$ almost surely and $h=g$ almost surely, we must have
  that $f(t)=h(t)$.
\end{proof}

Let $\mf M = \mf M_0/\sim$ and make the space $\mf M$ into a metric space by
introducing the distance
\begin{equation}
\label{c02}
d_T(f,g) \;=\; \inf_{A \in \mc B ([0,T])} \big\{\|f-g\|_{\infty,A^c} +
\Lambda(A)\big\}\;, 
\end{equation}
where $\mc B([0,T])$ is the set of Borel subsets of $[0,T]$, and
$\|f-g\|_{\infty,A^c}$ stands for the supremum norm of $f-g$
restricted to $A^c$.

\begin{lemma}
\label{l2}
The distance $d_T$ is well defined and it introduces a metric in $\mf M$.
\end{lemma}

\begin{proof}
Since $\mc D(h)$ has measure zero for any $h \in \mf M_0$, in the
formula defining the distance $d_T$ we can restrict the infimum to
those sets $A$ whose complement $A^c$ is contained in $\mc C(f)\cap
\mc C(g)$.

To see that $d_T$ is well defined, note that replacing $f$ by some $f'
\sim f$ does not alter the value of $d_T(f,g)$, according to the
previous remark. Symmetry of $d_T$ is also clear.

Now suppose that $d_T(f,g)=0$. To prove that $f=g$ in $\mf M$, i.e., that
$f\sim g$, we need to show that $f(t)=g(t)$ for all $t \in \mc C(f)
\cap \mc C(g)$. Fix such point $t$ and note that $f$ and $g$ are
constant on a neighborhood $B(t)$ of $t$. Taking $\epsilon <
\Lambda(B(t))$, since $d_T(f,g)=0$, we can find a set $A_\epsilon$ in
$\mc B([0,T])$ such that $\|f-g\|_{\infty,A_\epsilon^c}
+\Lambda(A_\epsilon) \leq \epsilon$. Therefore, $B(t)\cap A_\epsilon^c
\neq \varnothing$, so that $|f(t)-g(t)| \leq \epsilon$. Since this
holds for arbitrary $\epsilon>0$, $f(t)=g(t)$, as we wanted to prove.

Finally, to prove the triangular inequality, consider three functions
$f$, $g$ and $h$ in $\mf M$. For $\epsilon >0$, let $A_\epsilon$, $B_\epsilon$
be sets in $\mc B([0,T])$ such that
\begin{equation*}
\begin{split}
& d_T(f,g) \;\geq\; \|f-g\|_{\infty,A_\epsilon^c} \;+\;
\Lambda(A_\epsilon) \;-\; \epsilon \;, \\
& \quad d_T(g,h) \;\geq\; \|g-h\|_{\infty,B_\epsilon^c} \;+\;
\Lambda(B_\epsilon) \;-\; \epsilon \;.
\end{split}
\end{equation*}
Since for any set $E$ in $\mc B([0,T])$, $\|f-h\|_{\infty,E} \le
\|f-g\|_{\infty,E} + \|g-h\|_{\infty,E}$,
\begin{equation*}
d_T(f,h) \;\leq\; \|f-h\|_{\infty,A_\epsilon^c \cap B_\epsilon^c} 
\;+\; \Lambda(A_\epsilon \cup B_\epsilon)  \;\leq\;
d_T(f,g) \;+\; d_T(g,h) \;+\; 2\epsilon \;.
\end{equation*}
Since $\epsilon$ is arbitrary, $d_T$ is a metric.
\end{proof}

The space $\mf M$ is separable with respect to the metric $d_T$, but
it is not complete. On the one hand, the set of piecewise constant
functions for which the jump points and the range are in a dense
countable set is dense in $\mf M$. On the other hand, the
function $f(t)=t$ which does not belong to $\mf M$ can be arbitrarily
approximated in the distance $d_T$ by functions in $\mf M$.

Among all elements of an equivalence class, we choose one
representative as follows. For each element $f$ of $\mf M_0$, let
$\tilde f:[0,T]\to \bb R$ be given by
\begin{equation}
\label{01}
\tilde f(t) \;=\; \frac{1}{2} \Big\{\liminf_{\substack{s \to t\\ s\in
    \mc C(f)}} f(s) + \limsup_{\substack{s \to t\\ s\in
    \mc C(f)}} f(s)\Big\}\;.
\end{equation}
When $\liminf_{s \to t , s\in \mc C(f)} f(s) = -\infty$ and
$\limsup_{s \to t , s\in \mc C(f)} f(s)= +\infty$, we set $\tilde f(t)
=0$.

Clearly, $\tilde f = f$ on $\mc C(f)$ so that $\mc D(\tilde f) \subset
\mc D(f)$, where inclusion may be strict. In particular, $\tilde f$
belongs to the equivalence class of $f$: $\tilde f \sim f$.

\begin{lemma}
\label{s10b}
We have that 
\begin{equation*}
\limsup_{\substack{s \to t\\ s\in \mc C(f)}} f(s) \;=\; 
\limsup_{\substack{s \to t\\ s\in \mc C(g)}} g(s) 
\end{equation*}
whenever $f \sim g$, with a similar identity if we replace $\limsup$
by $\liminf$. In particular, $\tilde f = \tilde g$ if $f\sim g$ and
equation \eqref{01} distinguishes a unique representative for each
equivalence class of $\mf M$.
\end{lemma}

\begin{proof}
Consider two functions $f$, $g$ in the same equivalence class of $\mf
M$. It is enough to show that
\begin{equation*}
\limsup_{\substack{s \to t\\ s\in \mc C(f)}} f(s) \;\le\; 
\limsup_{\substack{s \to t\\ s\in \mc C(g)}} g(s) \quad\text{and}\quad
\liminf_{\substack{s \to t\\ s\in \mc C(f)}} f(s) \;\ge\; 
\liminf_{\substack{s \to t\\ s\in \mc C(g)}} g(s)\;.
\end{equation*}

We prove the first inequality, the derivation of the second one being
similar. There exists a sequence $\{s_j : j\ge 1\}$ such that $s_j\in
\mc C(f)$, $\lim_j s_j = t$,
\begin{equation*}
\limsup_{\substack{s \to t\\ s\in \mc C(f)}} f(s) \;=\;
\lim_{j\to\infty} f(s_j)\;.
\end{equation*}
Since $s_j$ belongs to $\mc C(f)$, $f$ is constant in an interval
$(s_j-\epsilon, s_j+\epsilon)$ and therefore in the interval $I_j =
(s_j-\epsilon, s_j+\epsilon) \cap (s_j-(1/j), s_j+(1/j))$. Of course,
$I_j\subset \mc C(f)$.  As $\mc D(g)$ has Lebesgue measure $0$, $\mc
C(g) \cap I_j \not = \varnothing$. Take an element $s'_j$ of this
latter set. Since $I_j$ is contained in $\mc C(f)$, $s'_j$ belongs to
$\mc C(f) \cap \mc C(g)$ so that $g(s'_j) = f(s'_j)$. Moreover, since
$f$ is constant in $I_j$ and $s_j$, $s'_j$ belong to $I_j$, $f(s_j) =
f(s'_j)$. On the hand, $\lim_j s'_j = t$ because $s_j$ converges to
$t$ and $|s'_j-s_j|<(1/j)$. Hence,
\begin{equation*}
\lim_{j\to\infty} f(s_j) \;=\; \lim_{j\to\infty} g(s'_j) \;\le\;
\limsup_{\substack{s \to t\\ s\in \mc C(g)}} g(s)\;,
\end{equation*}
which proves the lemma.
\end{proof}

 From now on when considering an equivalence class in $\mf M$, we always
refer to the representative defined by \eqref{01}. For example, $\Vert
f\Vert_\infty$, $f\in\mf M$, whose value may be different for two
distinct functions in $\mf M_0$ belonging to the same equivalence
class, means in reality $\Vert \tilde f\Vert_\infty$.

In order to obtain a compactness criterion in $\mf M$, we introduce
the following modulus of continuity. For a measurable function
$f:[0,T]\to \bb R$ and $\delta>0$, let
\begin{equation*}
\omega_\delta(f) \;=\; \Lambda\big(B(\mc D(f),\delta)\big)\;.  
\end{equation*}
The modulus of continuity $\omega_\delta(f)$ converges to $0$ as
$\delta \to 0$ if and only if $f$ belongs to $\mf M_0$. We extend this
definition to the space $\mf M$. For an equivalence class $f\in \mf M$, 
let 
\begin{equation*}
\omega_\delta(f) \;=\; \Lambda\big(B(\mc D(\tilde f),\delta)\big)\;.  
\end{equation*}
Lemma \ref{s10b} ensures that the modulus of continuity is well
defined, i.e., that $\omega_\delta(f) = \omega_\delta(g)$ if $f$ and
$g$ belong to the same equivalence class, because $\tilde f = \tilde
g$.

\begin{proposition}
\label{s09}
A subset $\mc F \subseteq \mf M$ is sequentially precompact with
respect to $d_T$ if 
\begin{equation}
\label{ec1}
\sup_{f \in \mc F} \Vert  f \Vert_\infty \;<\;
\infty \quad\text{and}\quad
\lim_{\delta \to 0} \sup_{f \in \mc F} \omega_\delta(f) \;=\; 0\;.
\end{equation}
\end{proposition}

\begin{proof}
For $f \in \mc F$, define $\ell_f^\delta(t) = \dist(t, B(\mc D(\tilde
f),\delta)^c)$. Since $\ell_f^\delta$ is $1$-Lipschitz for any $f \in
\mf M$ and any $\delta >0$, the family $\{\ell_f^\delta, f \in \mc
F\}$ is equicontinuous. Fix a sequence $f_n$ in $\mc F$ and a sequence
$\{\delta_m : m\ge 1\}$ of positive numbers such that $\lim_m \delta_m
= 0$. Since $\sup_{f \in \mc F} \Vert f\Vert_\infty < \infty$,
by a standard Cantor diagonal argument, we can extract a subsequence,
still denoted by $f_n$, for which, as $n\uparrow\infty$,
$\ell_{f_n}^{\delta_m}$ converges uniformly to some function
$\ell^{\delta_m}$ for every $m$, and $\tilde f_n(t)$ converges to some
limit $F(t)$ for any rational $t$ in $[0,T]$.

Let $\epsilon_m = \limsup_{n \to \infty} \omega_{\delta_m}(f_n)$.  By
\eqref{ec1}, $\lim_m \epsilon_m = 0$.  Since $\ell_{f_n}^{\delta_m}$
converges uniformly to $\ell^{\delta_m}$ and since
$\{\ell^{\delta_m}_{f_n} \neq 0\} = B(\mc D(\tilde f_n), \delta_m)$,
\begin{equation}
\label{ec2e}
\Lambda(\ell^{\delta_m}\neq 0) \;\leq\; \limsup_{n \to \infty} 
\Lambda(\ell_{f_n}^{\delta_m} \neq 0) \;=\; \limsup_{n \to \infty} \
\omega_{\delta_m}(f_n) \;=\; \epsilon_m \;.
\end{equation}

We claim that for every $t\in [0,T]$ such that $\ell^{\delta_m}(t)=0$
for some $m\ge 1$, there exist a neighborhood $N(t)$ of $t$ and an
integer $n_0\ge 1$ for which $F$ is constant on $N(t)\cap \bb Q$ and
$\tilde f_n(t)$ is constant on $N(t)$ for $n\ge n_0$. We postpone the
proof of this claim.

As $\lim_m \epsilon_m = 0$, by \eqref{ec2e} $\lim_m
\Lambda(\ell^{\delta_m}\neq 0) = 0$. There exists therefore a
subsequence $\{m(j) : j\ge 1\}$ such that $\sum_j
\Lambda(\ell^{\delta_{m(j)}}\neq 0) <\infty$. Let $A = \cap_{k\ge
  1}\cup_{j\ge k} \{ \ell^{\delta_{m(j)}}\neq 0 \}$ so that $\Lambda
(A)=0$. If $t$ belongs to the set $A^c$, which has full measure,
$\ell^{\delta_{m(j)}} (t) = 0$ for some $j$. By the conclusions of the
previous paragraph, there exist a neighborhood $N(t)$ of $t$ and an
integer $n_0\ge 1$ for which $F$ is constant on $N(t)\cap \bb Q$ and
$\tilde f_n(t)$ is constant on $N(t)$ for $n\ge n_0$.

In view of the previous result we may define a function $\hat
F:[0,T]\to\bb R$ which vanishes on the set $A$, and which on each
element $t$ of the set $A^c$ is locally constant with value given by
the value of $F$ on a rational point close to $t$. In particular, $A^c
\subset \mc C(\hat F)$ which ensures that $\hat F$ belongs to $\mf
M_0$. Moreover, it follows from the convergence of $\tilde f_n$ to $F$
on the rationals that $\tilde f_n (t)$ converges to $\hat F(t)$. Since the
set $A$ has Lebesgue measure $0$, $f_n$ converges almost surely to
$\hat F$. Therefore, by Egoroff theorem, $f_n$ converges to $\hat F$
with respect to the metric $d_T$.

To conclude the proof of the proposition, it remains to verify the
assertion assumed in the beginning of the argument.  Fix $t\in [0,T]$
and suppose that $\ell^{\delta_m}(t) =0$ for some $m \ge 1$. In this
case, since $\ell_{f_n}^{\delta_m}(t)$ converges to
$\ell^{\delta_m}(t)=0$, $\lim_n \dist(t,B(\mc D(\tilde
f_n),\delta_m)^c) =0$. Take a point $t_n$ in the compact set $B(\mc
D(\tilde f_n),\delta_m)^c$ realizing this distance to conclude that
there exists a sequence $t_n$ converging to $t$ for which
$\ell_{f_n}^{\delta_m}(t_n)=0$.  As $\ell_{f_n}^{\delta_m}(t_n)=0$,
$\tilde f_n$ is constant in the interval
$(t_n-\delta_m,t_n+\delta_m)$.  Therefore, the functions $\tilde f_n$
are constant in a neighborhood $N(t)$ of $t$ for $n$ large
enough. Since $\tilde f_n$ converges on the rationals to $F$, we
conclude, as claimed, that $F$ is constant in $N(t) \cap \bb Q$.
\end{proof}

Another topology which can be defined in the space $\mf M$ corresponds
to the projection of the Skorohod's $M_2$ topology, which is generated
by the Hausdorff distance between the graphs of the functions. For two
equivalence classes $f$ and $g$ in $\mf M$, define the distance
$d^{(2)}_T (f,g)$ by
\begin{equation}
\label{c03}
d^{(2)}_T (f,g) \;:=\; d_H(\Gamma_{\tilde f}, \Gamma_{\tilde g})\;, 
\end{equation}
where $\tilde f$, $\tilde g$ are the representative of the equivalence
class of $f$, $g$,
\[
\Gamma_{\tilde f} = \bigcup_{t \in [0,T]} \{t\} 
\times [\liminf_{s \to t} \tilde f(s),\limsup_{s \to t} \tilde f(s)]\;,
\]
and $d_H$ is the Hausdorff distance. 

Recall the definition of the modulus of continuity $\omega_\delta(f)$
and note that $\omega_\delta(f) \geq 2\delta$ unless $f$ is
constant. Denote by $B(f;r)$, $B^{(2)}(f;r)$ the ball of center $f$
and radius $r$ with respect to the metric $d_T$, $d^{(2)}_T$,
respectively.

\begin{lemma}
\label{l3}
For any equivalence class $f \in \mf M$ and any $\delta>0$,
\[
B^{(2)}(f;\delta) \;\subseteq\; B (f;\delta + \omega_{2\delta}(f)) \;.
\]
\end{lemma}

\begin{proof}
Fix $f \in \mf M$, $\delta>0$ and $g \in B^{(2)} (f;\delta)$. By
definition of $d_T$,
\begin{equation*}
\begin{split}
& d_T(g,f) \;=\; d_T(\tilde g, \tilde f) \;\leq\; 
\lVert \tilde f - \tilde g \rVert_{\infty,
  B(\mathcal{D}(\tilde f),2\delta)^c} 
\;+\; \Lambda \big( B(\mathcal{D}(\tilde f),2\delta) \big)\\
&\qquad = \;\lVert \tilde f - \tilde g 
\rVert_{\infty, B(\mathcal{D}(\tilde f),2\delta)^c} \;+\; 
\omega_{2\delta} (\tilde f)\;.
\end{split}
\end{equation*}
In order to evaluate the first term above, fix $t \notin
B(\mathcal{D}(\tilde f),2\delta)$ so that $\tilde f$ is constant in
$B(t, 2\delta)$. In particular, $\Gamma_{\tilde f} \subset \Sigma =
[0,t - 2\delta] \times \bb R \cup [0,T] \times \{\tilde f(t)\} \cup [t
+ 2\delta, T] \times \bb R$.  Since $d^{(2)}_T(\tilde g, \tilde f) =
d^{(2)}_T(g,f) \leq \delta$, by definition of the Hausdorff distance,
\begin{equation*}
\delta \;\geq\; \text{dist} \big( (t,\tilde g(t)) \,,\, \Gamma_{\tilde f} \big)
\;\geq\; \text{dist} \big( (t,\tilde g(t)) \,,\, \Sigma  \big) 
\;=\; 2\delta \wedge |\tilde f (t) - \tilde g (t)| \;.
\end{equation*}
This implies that $|\tilde f (t) - \tilde g (t)| \leq \delta$ for
every $t \notin B(\mathcal{D}(\tilde f),2\delta)$, which finishes the
proof of the lemma.
\end{proof}

Consider a sequence $\{Y_n : 1\le n\le \infty\}$ of real-valued
stochastic processes defined on some probability space $(\Omega, \mc
F, P)$. Assume that the trajectories of each $Y_n$, $1\le n\le
\infty$, belong to $\mf M_0$ $P$-almost surely. This is the case, for
instance, of continuous-time Markov chains taking values on a
countable subset of $\bb R$.

\begin{theorem}
\label{t1} 
Fix $T>0$. If $d^{(2)}_T(Y_n,Y_\infty)$ converges to $0$ in
probability as $n \uparrow \infty$, then $d_T(Y_n,Y_\infty)$ converges
to $0$ in probability as $n \uparrow \infty$.
\end{theorem}

\begin{proof}
It is enough to show that for each $\epsilon >0$, $\lim_{n\to \infty}
P[d_T(Y_n,Y_\infty) > 2 \epsilon] =0$. Fix $\delta < \epsilon$ so that
the previous probability is bounded by $P[d_T(Y_n,Y_\infty) > \epsilon
+ \delta]$.  This latter probability is in turn less than or equal to
\begin{equation*}
P \big[ d_T(Y_n,Y_\infty) > \epsilon + \delta \,,\,
\omega_{2\delta} (\tilde Y_\infty) \le \epsilon \big]
\;+\; P \big[ 
\omega_{2\delta} (\tilde Y_\infty) > \epsilon \big]\;.
\end{equation*}
Since $Y_\infty$ has trajectories in $\mf M_0$ $P$-almost surely, the
second term vanishes as $\delta\downarrow 0$. The first one is bounded
by $P [ d_T(Y_n,Y_\infty) > \delta + \omega_{2\delta} (\tilde
Y_\infty)]$ which by the previous lemma is less than or equal to $P [
d^{(2)}_T(Y_n,Y_\infty) > \delta ]$. By assumption, this term vanishes
as $n\uparrow\infty$.
\end{proof}

Assume that in the probability space $(\Omega, \mc F, P)$ introduced
before the statement of the previous theorem is also defined a
sequence $\{X_n : 1\le n< \infty\}$ of real-valued stochastic
processes whose trajectories belong to $\mf M_0$ $P$-almost surely.

\begin{corollary}
\label{c1}
Fix $T>0$. If both $d_T(X_n,Y_n)$ and $d^{(2)}_T(Y_n,Y_\infty)$
converge to zero in probability as $n\uparrow\infty$, then
$d_T(X_n,Y_\infty)$ also converges to zero in probability as
$n\uparrow\infty$.
\end{corollary}

\section{Main result}
\label{sec6}

We prove in this section that under certain assumptions the
continuous time Markov process $X^N_t$, introduced in Section
\ref{sec1}, is close, in an appropriate time scale and with respect to
the topology introduced in Section~\ref{sec4}, to a simple random walk
$Y^N_t$ which only visits the set $A_N$ of the deepest traps and which
has identically distributed jump probabilities: $p_N(x,y) =
\rho_N(y)$, $x$, $y\in A_N$.  For such result we need, roughly
speaking, the set of deepest traps $A_N$
\begin{itemize}
\item to support most of the stationary measure $\nu$.
\item to consist of well separated points, 
\item to be unlikely to be hit in a short time,
\item to have comparable escape probabilities from different points.
\end{itemize}

The main result presented below holds in a more general context than
the one described in Section \ref{sec1}.  We suppose throughout this
section that $\{G_N : N\ge 1\}$ is a sequence of finite, connected,
vertex-weighted graphs, where $\{W^N_x : x\in V_N\}$ represents the
positive weights. The vertices of $V_N$ are enumerated in decreasing
order of weights, $V_N = \{x^N_1, \dots, x^N_{|V_N|}\}$, $W^N_{x_j}
\ge W^N_{x_{j+1}}$, $1\le j\le |V_N|-1$.

Denote by $X^N_t$ the Markov process on $V_N$ with generator given by
\eqref{f05}. We do not assume  that the depths $W^N_x$ are chosen
according to \eqref{13}, but we impose some conditions presented
below in ({\bf A0})--({\bf A3}).

We write in this section $J_N\uparrow\infty$ to represent an
increasing sequence of natural numbers $\{J_N : N\ge 1\}$ such that
$\lim_{N\to \infty} J_N = \infty$.  To keep notation simple, we
sometimes omit the dependence on $N$ of states, measures and sets.

Recall that $\nu=\nu_N$, defined in \eqref{nu}, is the stationary
measure of the random walk $X^N_t$. Assume that $\nu(B^c_N)$ vanishes
asymptotically for any sequence of subsets $B_N = \{x_1^N, \dots,
x_{J_N}^N\}\subset V_N$ such that $J_N\uparrow\infty$:
\begin{equation}
\label{A0}
\tag{\bf A0}
\lim_{N\to\infty}  \nu_N(B_N^c) \;=\; 0\;.
\end{equation}

Fix three sequences $M_N\uparrow\infty$, $\ell_N\uparrow\infty$ and
$L_N\uparrow\infty$, $M_N\le |V_N|$. The sequence $M_N$ represents the
number of deep traps selected, and $\ell_N$ a lower bound on the
minimal distance among these deepest traps.  We formulate three
assumptions on these sequences.  Let $A_N = \{x^N_1, \dots,
x^N_{M_N}\}$ be the set of the deepest traps.  We first require the
deepest traps to be well separated:
\begin{equation}
\label{A1}
\tag{\bf A1}
d(x^N_i, x^N_j) > 2\ell_N +1 \;, \quad 1\le i \not = j\le M_N 
\end{equation}
for all $N$ large enough.  This condition, which is analogous to
condition \eqref{BB0}, ensures that any path $\{x^N_i = z_0, z_1,
\dots, z_m=x^N_j\}$ from $x^N_i$ to $x^N_j$ has a state $z_k$ which
belongs to $R(A_N, \ell_N)$.

The second assumption is somehow related to \eqref{e:c4} and requires,
as explained below, the different escape probabilities $v_x$, $x \in
A_N$, to have similar order of magnitude.  For a subset $B$ of $V_N$,
let $\nu_B$ be the measure $\nu$ conditioned on $B$:
\begin{equation*}
\nu_B(x) \;=\; \frac{W^N_x \, \deg(x)}{\sum_{y\in B} W^N_y \,
  \deg(y)}\;, \quad x\in B\;.
\end{equation*}
Expectation with respect to $\nu_B$ is denoted by $E_{\nu_{B}}$.

We suppose that there exists a sequence $\{\beta_N : N\ge 1\}$ such
that for any sequence of subsets $B_N = \{x^N_1, \dots,
x^N_{J_N}\}\subset A_N$ such that $|B_N|= J_N\uparrow\infty$
\begin{equation}
\tag{\bf A2}
\label{A2}
\limsup_{N \to \infty}  E_{\nu_{B}}
\Big[\frac{W^N_x}{\beta_N\, v_\ell(x)}\Big] \;<\; \infty\;, \quad
\limsup_{N \to \infty} \frac 1{|B_N|} \, E_{\nu_{B}}
\Big[\frac{\beta_N\, v_\ell(x)}{W^N_x}\Big] \;<\; \infty\;.
\end{equation}
This hypothesis postulates essentially a law of large numbers for
$\deg(x_j)\, v_\ell(x_j)$ and a bound for the sum of $(W^N_{x_j})^2
\deg(x_j)/v_\ell(x_j)$.

In analogy with \eqref{BB1}, we will also assume that the hitting time
of $A_N$ is much smaller than the mixing time of the discrete-time
random walk on $G_N$. For $L \ge 1$ let
\begin{equation}
\label{03}
\kappa_N \;=\; \kappa (L, M_N, \ell_N) 
\;=\; \max_{x\in A_N} \max_{z \not\in B(x, \ell_N)} 
\mb P^N_z \big[ \HH_{x} <  L t_{\text{mix}}^N \big]\; .
\end{equation}
Assume that for some sequence $L_N \uparrow \infty$,
\begin{equation}
\tag{\bf A3}
\label{A3}
\lim_{N \to \infty} M_N^3 \,  \kappa_N \;=\; 0\;,
\quad \lim_{N \to \infty} M_N^2 \, 2^{-L_N} \;=\; 0  \;.
\end{equation}

\begin{remark}
\label{s20}
Consider three sequences $M_N\uparrow\infty$, $\ell_N\uparrow\infty$ and
$L_N\uparrow\infty$ satisfying {\rm ({\bf A0})--({\bf A2})} and such that
\begin{equation}
\label{06}
\lim_{N \to \infty} \kappa(L_N, M_N, \ell_N)  \;=\; 0\;.
\end{equation}
Then, there exists a sequence $M'_N\uparrow\infty$, $M_N'\le M_N$, for
which the three sequence $M_N'$, $\ell_N$, $L_N$ satisfy {\rm ({\bf
    A0})--({\bf A3})}.
\end{remark}

Indeed, it follows from \eqref{06} and the fact that
$L_N\uparrow\infty$ that there exists a se\-quence $K_N\uparrow\infty$
such that $\lim_{N\to\infty} K_N^2 \, 2^{-L_N} =0$, $\lim_{N\to\infty}
K_N^3 \, \kappa(L_N, M_N, \ell_N) =0$. Define a new sequence $M'_N$ by
${M}'_N = \min\{M_N, K_N\}$ and define $A'_N$ accordingly. Since
$A'_N\subset A_N$ and $\kappa'_N\le \kappa_N$, ({\bf A0})--({\bf A3})
hold for the sequences $M'_N$, $\ell_N$, $L_N$. \smallskip

Hence, in the applications, if one is able to prove \eqref{06}, one
can redefine the sequence $M_N$ to obtain ({\bf A3}) which is the
condition assumed in the main result of this section.  Moreover, if a
sequence $M_N$ satisfies conditions ({\bf A1}), ({\bf A2}),
\eqref{06}, then any sequence $M'_N\uparrow\infty$ which increases to
infinity with $N$ at a slower pace than $M_N$, $M'_N \le M_N$, also
satisfies these three conditions. The same observation holds for the
sequence $L_N$. Hence, in the applications, both sequences shall
increase very slowly to infinity, in a way that ({\bf A3}) is
fulfilled, and all the problem rests on the identification of a
convenient space scale $\ell_N$, large for the process to mix before
returning to a state, as required in condition \eqref{06}, but not too
large, to permit a good description of a ball of radius $\ell_N$ and a
good estimate of the escape probability $v_\ell(x)$.

Let $\rho_N$ be the probability measure on the set $A_N$ given by
\begin{equation}
\label{rho2}
\rho_N(x_j) = \frac{\deg(x_j)\, v_\ell(x_j)}
{\sum_{1\le i \leq M_N} \deg(x_i)\, v_\ell(x_i)}\;,
\end{equation}
where $v_\ell(x_j)=v^N_{\ell_N}(x_j)$ is the escape probability
introduced in \eqref{f01}. By \eqref{nu}, $\rho_N$ can also be written
as
\begin{equation}
\label{rho3}
\rho_N(x_j) = \frac{\nu(x_j) v_\ell(x_j) W^{-1}_{x_j}}{\sum_{1\le i \le M_N} 
\nu(x_i) v_\ell(x_i) W^{-1}_{x_i}}\;,
\end{equation}
which corresponds to \eqref{rho} with $U = H_{R(A_N,\ell_N)}$.

For each $N \ge 1$, consider the continuous-time Markov process
$\{Y^N_t : t \ge 0\}$ on $A_N$ defined as follows. While at $x\in A_N$
the process waits a mean $W^N_x /v_\ell(x)$ exponential time at the
end of which it jumps to $y \in A_N$ with probability $\rho_N(y)$. Note that
the jump distribution is independent of the current state and that the
process may jump to its current state since we did not impose $y$ to
be different from $x$. Moreover, the probability measure
$\nu^N(x)/\nu^N(A_N)$ is the (reversible) stationary state of the
Markov chain $\{Y^N_t : t\ge 0\}$.

We are now in a position to state the main result of this paper, from
which we will deduce Theorems~\ref{t:trans} and \ref{t:random}.

\begin{theorem}
\label{dXYf}
Suppose that conditions {\rm ({\bf A0})--({\bf A3})} are in
force. Then, for every $N\ge 1$, there exists a coupling $Q_N$ between
the stationary, continuous-time Markov chain $\{Y^N_{\beta_Nt} : t\ge
0\}$ described above, and the Markov chain $\{X_{\beta_Nt}^N : t\ge
0\}$ such that $Q_N[X_{0}^N=Y_{0}^N=y]=\rho(y)$, $y\in A_N$, and
\begin{equation*}
\lim_{N\to\infty} Q_N \big[ d_T(X^N_{\beta_N\cdot}, Y^N_{\beta_N\cdot}) >
\delta \big] \;=\; 0
\end{equation*}
for every $T \geq 0$ and $\delta>0$, where $d_T$ stands for the
distance introduced in \eqref{c02}.
\end{theorem}

\begin{corollary}
\label{s05}
Consider a sequence $\{z^N : N\ge 1\}$, $z^N\in A_N$, such that
\begin{equation*}
\liminf_{N\to\infty} \rho_N(z^N) > 0\;.
\end{equation*}
The statement of Theorem \ref{dXYf} remains in force if in the
assumptions we replace the property $Q_N[X^N_0=Y^N=y]=\rho(y)$, $y\in
A_N$, by the property $Q_N[X^N_0=Y^N=z^N]=1$.
\end{corollary}

\begin{proof}
The assertion of this corollary follows from Theorem \ref{dXYf} by
conditioning on the event $X^N_0=z^N$.
\end{proof}

Theorem \ref{dXYf} follows from Lemmas \ref{s00}, \ref{sd01} and
Proposition \ref{dXY} below. Theorem \ref{dXYf} asserts that the
process $X_{\beta_Nt}^N$ is close to the process $Y_{\beta_Nt}^N$
which jumps at rate $\beta_N v_\ell(x)/W^N_x$. If this latter
expression is not of order one, the asymptotic behavior of
$Y^N_{\beta_Nt}$ will not be meaningful and our approximation of
$X_{\beta_Nt}^N$ by $Y_{\beta_Nt}^N$ devoid of interest. Hence, in the
applications we expect
\begin{equation*}
\beta_N \;\approx\; \frac{W^N_{x_j}}{v_\ell(x_j)} \;\cdot
\end{equation*}

\begin{lemma}
\label{s00}
Assume that hypotheses {\rm ({\bf A0})--({\bf A3})} are in
force. Then, there exists a subset $B_N =\{x^N_1, \dots,
x^N_{J_N}\}\subset A_N$ such that,
\begin{eqnarray}
\label{e:c1}
&& \lim_{N \to \infty} \kappa_N \, M_N\, 
\frac{\beta_N}{E_{\rho}\Big[\frac{W^N_x}{v_\ell(x)}\Big]} \;=\; 0\; , \\
\label{e:c2}
&&\quad \lim_{N \to \infty} \nu(B_N^c) \;=\; 0\;, \\
\label{e:c3}
&&\qquad \limsup_{N \to \infty} \frac{\beta_N} 
{E_{\rho}\Big[\frac{W^N_x}{v_\ell(x)}\Big]}
\, \nu(A_N^c) \,\rho(B_N) \;=\;0\;.
\end{eqnarray}
\end{lemma}

\begin{proof}
We start proving \eqref{e:c1}. By definition of the probability
measure $\rho_N$ this expression is equal to
\begin{equation*}
\kappa_N \, M^2_N \, \frac 1{M_N} \, E_{\nu_{A}}
\Big[\frac{\beta_N\, v_\ell(x)}{W^N_x}\Big]  \;.
\end{equation*}
This term vanishes as $N\uparrow\infty$ in view of ({\bf A3})
and ({\bf A2}) with $B_N=A_N$.

By ({\bf A0}), $\nu(A_N^c)$ vanishes as $N\uparrow\infty$.  There
exists, therefore, a sequence $K_N\uparrow\infty$ such that
$\lim_{N\to\infty} K_N \nu(A_N^c) =0$. Let $B_N =\{x^N_1, \dots
x^N_{J_N}\}$, where $J_N = \min\{M_N, K_N\}$ so that $|B_N| \nu(A_N^c)
\to 0$.  The second assertion of the lemma follows from assumption
({\bf A0}) because $J_N\uparrow\infty$.  Moreover, as
\begin{equation*}
\beta_N\, E_{\rho}\Big[\frac{W^N_x}{v_\ell(x)}\Big]^{-1}
\,\rho(B_N) \;\le\; E_{\nu_{B}}
\Big[\frac{\beta_N\, v_\ell(x)}{W^N_x}\Big] \;,
\end{equation*}
by ({\bf A2}) and by definition of the set $B_N$, we have that
\begin{equation*}
\limsup_{N \to \infty} \frac{\beta_N} {E_{\rho}\Big[\frac{W^N_x}{v_\ell(x)}\Big]}
\, \nu(A_N^c) \,\rho(B_N) \;\le\; 
C_0 \, \limsup_{N \to \infty} |B_N| \, \nu(A_N^c) \;=\;0
\end{equation*}
for some finite constant. This concludes the proof of the lemma. 
\end{proof}

\begin{lemma}
\label{sd01}
Assume that conditions {\rm ({\bf A2}), ({\bf A3}),
  \eqref{e:c1}--\eqref{e:c3}} are in force. Then, there exists a
sequence $\{K_N:N\ge 1\}$ such that
\begin{eqnarray}
\label{B1}
&& \lim_{N \to \infty} K_N \, M_N \, 2^{-L_N} \; =\; 0\;, \quad
\lim_{N \to \infty} K_N \, M_N^2 \, \kappa_N \; =\; 0\;,\\
&& \quad 
%\tag{\bf B2}
\label{B2}
\lim_{N \to \infty} \frac{K_N \, \nu(V_N \setminus A_N)}
{\beta_N \, \nu(A_N)} \, E_\rho\Big[ \frac{W^N_x}{v_\ell(x)}\Big]
\;=\;0\;, \\
&&\qquad 
%\tag{\bf B3}
\label{B3}
\lim_{N \to \infty} \frac{K_N}{\beta_N} 
\, E_\rho\Big[\frac{W^N_x}{v_\ell(x)}{\mb 1}\{x \notin B_N\} \Big]
\;=\;0\;, \\
&&\qquad\quad
%\tag{\bf B4}
\label{B4}
\lim_{N \to \infty} \frac{K_N^2 \nu(V_N \setminus A_N)}
{\beta_N \nu(A_N)} \, E_\rho\Big[ \frac{W^N_x}{v_\ell(x)}\Big] 
\, \rho(B_N) \;=\;0\;, \\
%\tag{\bf B5}
\label{B5}
&&\qquad\qquad
\lim_{N \to \infty} \frac{K_N}{\beta_N} \, 
E_\rho\Big[\frac{W^N_x}{v_\ell(x)}\Big] \;=\;\infty\;.  
\end{eqnarray}
\end{lemma}

\begin{proof}
In view of ({\bf A3}), there exists a sequence $\psi_N\uparrow\infty$
such that $\psi_N \, M_N^2 \, 2^{-L_N}$, $\psi_N \, M_N^3 \, \kappa_N$
vanish as $N\uparrow\infty$. We may choose this sequence $\psi_N$ so
that the limits in {\rm \eqref{e:c1}} and {\rm \eqref{e:c2}} still
hold when multiplied by $\psi_N$, as well as the one in {\rm
  \eqref{e:c3}} when multiplied by $\psi_N^2$. Given this sequence
$\psi_N$, let
\begin{equation*}
K_N \;=\; \frac{\psi_N
  \beta_N}{E_{\rho}\Big[\frac{W^N_x}{v_\ell(x)}\Big]}
\;=\; \psi_N \, \beta_N \, E_{\nu_{A}}
\Big[\frac{v_\ell(x)}{W^N_x}\Big]\;.
\end{equation*}

Conditions \eqref{B1} follow the definition of $\psi_N$ and from
({\bf A2}), while condition \eqref{B2} follows from \eqref{e:c2} since
$B_N\subseteq A_N$. To verify \eqref{B3}, it is enough to remember that
$\rho(x) W^N_x v_\ell(x)^{-1} = \nu(x)$ and to recall
\eqref{e:c2}. Condition \eqref{B4} follows from assumptions
\eqref{e:c3}, \eqref{e:c2} and the definition of $K_N$. Finally,
condition \eqref{B5} requires $\psi_N$ to diverge.
\end{proof}

\begin{proposition}
\label{dXY}
Suppose that conditions {\rm ({\bf A1}), ({\bf A2}),
  \eqref{B1}--\eqref{B5}} are in force. Then, for every $N\ge 1$,
there exists a coupling $Q_N$ between the stationary, con\-tin\-u\-ous-time
Markov chain $\{Y^N_{\beta_Nt} : t\ge 0\}$ on $A_N$ with mean $W^N_x/
\beta_N v_\ell(x)$ exponential times and uniform jump probabilities
$p_N(x,y)=\rho_N(y)$, $x$, $y\in A_N$, and the Markov chain
$\{X_{\beta_Nt}^N : t\ge 0\}$ such that
$Q_N[X_{0}^N=Y_{0}^N]=\rho(y)$, $y\in A_N$, and
\begin{equation*}
\lim_{N\to\infty} Q_N \big[ d_T(X^N_{\beta_N\cdot}, Y^N_{\beta_N\cdot}) >
\delta \big] \;=\; 0
\end{equation*}
for every $T \geq 0$ and $\delta>0$, where $d_T$ stands for the
distance introduced in \eqref{c02}.
\end{proposition}

\begin{proof}
Recall the definition of the sequence of stopping times $\{D_k : k\ge
0\}$ introduced in Section~\ref{sec:general} with $U =
H_{R(A_N,\ell_N)}$. Since by ({\bf A1}) $R(A_N,\ell_N) \not =
\varnothing$ and since the state space is finite and irreducible, $\mb
E_x[U] <\infty$ for all $x\in A$. It also follows from
assumption ({\bf A1}) that $\mb P_x[H_{A\setminus \{x\}} \ge U]
=1$ for all $x \in A$. Therefore, by Lemma~\ref{t02} and
Proposition~\ref{t01}, the discrete-time Markov chain $X^N_{D_k}$
is irreducible and its unique stationary state is the measure $\rho$
defined in \eqref{rho2}.

We start the construction of the measure $Q_N$ by coupling the
discrete skeleton of the chain $Y^N_t$ with the chain $X^N_{D_k}$, and
by coupling the waiting times of the chain $Y^N_t$ with the times
spent by $X^N_t$ at each site of $A_N$. It follows from
Lemma~\ref{s03}, which presents an estimate of the distance between
the measure $\rho$ and the measure $p(\,\cdot\,,A)$, from
Lemma~\ref{s01} and from the strong Markov property at time
$H_{R(A,\ell)}$ that
\begin{equation}
\label{c04}
\sup_{y \in A} \big\Vert \mb P_y[X^N_{D_1} = \cdot] - \rho(\cdot) 
\big\Vert_{TV} \;\leq\; (M_N+1)\, (2^{-L_N}  \,+\, M_N \, \kappa_N) 
\;=:\; a_N\;.
\end{equation}

Let $\sigma_0=0$ and denote by $\{\sigma_i : i \geq 1\}$ the jump
times of the chain $Y^N_t$, including among these jumps the ones to
the same site. We couple the initial state $X^N_{0}$ and $Y^N_{0}$ so
that $Q_N[X^N_{0} = Y^N_{0}]=1$, $Q_N[X^N_{0} = x]=\rho(x)$, $x\in A$.
As $Y^N_{\sigma_1}$ is distributed according to $\rho$, by \eqref{c04}
we can couple $X^N_{D_1}$ and $Y^N_{\sigma_1}$ in a way that they
coincide with probability at least $1-a_N$.  Moreover,
conditioned on $X^N_{D_i} = x$, the number of visits of $X^N_t$ to the
point $x$ between times $D_i$ and $D_{i+1}$ is a geometric random
variable with success probability $v_\ell(x)$, so that
\begin{equation*}
\int_{[D_i, D_{i+1})} \mb 1 \{X^N_t = x\} \, dt
\end{equation*}
is an exponential random variable with expectation $W_{x}/
v_\ell(x)$. This is also the distribution of the time that $Y^N_t$
spends in $x$. Proceeding by induction and using the strong Markov
property at times $D_i$ (for $X^N_t$) and $\sigma_i$ (for $Y^N_t$), we
obtain a coupling $Q_N$ between $X^N_t$ and $Y^N_t$ such that
\begin{equation*}
Q_N \Big[
\begin{array}{c}
X^N_{D_i} = Y^N_{\sigma_i} \,,\, \int_{D_i}^{D_{i+1}} 
\mb 1 \{X^N_t = X^N_{D_i}\} \, dt  = \sigma_{i+1} - \sigma_i \\
\text{ for every $0\le i \le K_N$}
\end{array} \Big] \;\geq\; 1 - K_N\, a_N \; ,
\end{equation*} 
where $K_N$ is the sequence introduced in Lemma \ref{sd01}.  Denote
the event appearing in the previous formula by $\mathcal G$.  By
\eqref{B1},
\begin{equation}
\label{c10}
\lim_{N\to\infty} Q_N[\,\mc G^c \,]\;=\;0\;.
\end{equation}

We claim that the coupling $Q_N$ defined above satisfies the statement
of the theorem.  To estimate the distance between the processes
$X^N_t$ and $Y^N_t$, we introduce a third process $\bar X^N_t$ close
to $X^N_t$ in the distance $d_T$.  Following \cite{bl2}, consider the
process $\bar X^N_t$ defined by
\begin{equation}
\label{e:belowD}
\bar X^N_t = X^N( \sup \{s \leq t : X^N_s \in A_N\})\;.
\end{equation}
The (non-Markovian) process $\bar X^N_t$ indicates the last site in
$A_N$ visited by $X^N_s$ before time $t$. We adopt for $\bar X^N_t$
the same convention agreed for the process $Y^N_t$ and consider that
the process $\bar X^N_t$ jumped from $y\in A_N$ to $y$ at time $t'$ if
the process $X^N_t$ being at $y$ at time $s<t'$, reached $R(A_N,\ell)$
and then returned to $y$ at time $t'$ before hitting another site
$z\in A_N\setminus \{y\}$. With this convention, the jump times of the
process $\bar X^N_t$ are exactly the stopping times $\{D_i: i \geq
1\}$.

We assert that for every $T>0$ and $\delta>0$,
\begin{equation}
\label{c05}
\lim_{N\to\infty}  \mb{P}_\rho 
\big[ d_{T} (\bar X^N_{\beta_N\cdot}, X^N_{\beta_N\cdot}) > \delta \big] \;=\;0\;.
\end{equation}
Fix $T>0$ and $\delta>0$. By definition of the process $\bar{X}^N$,
\[
d_{T} (\bar X^N_{\beta_N\cdot}, X^N_{\beta_N\cdot}) \;\le\; 
\frac{1}{\beta_N} \int_0^{\beta_N T} {\mb 1}\{X^N_t \notin A_N\} \,dt\;.
\]
Therefore, 
\begin{equation*}
\mb{P}_\rho \big[d_{T} (\bar X^N_{\beta_N\cdot}, X^N_{\beta_N\cdot}) 
> \delta \big] \,\leq\,  \frac{1}{\beta_N \delta} \, 
\mb E_\rho \Big[ \int_0^{D_{K_N}} 
\mb 1\{ X^N_{t} \not\in A_N\} \, dt \Big]
\,+\, \mb P_\rho [  D_{K_N} \leq \beta_N T]\, .
\end{equation*}
Let us define
\[
\Delta_N := \int_0^{D_{K_N}} {\mb 1} \{X^N_t \notin A_N\} dt
\;=\; D_{K_N} - \sigma_{K_N} \;.
\]
This quantity will appear a couple of times in the computations below.
By \eqref{e:B5}, $\mb P_\rho [ D_{K_N} \leq \beta_N T]$ vanishes as
$N\uparrow\infty$ because $\sigma_{K_N} \le D_{K_N}$. On the other
hand, by definition of the process $\bar X^N_t$ and by stationarity,
\begin{equation*}
\frac{1}{\beta_N \delta} \mb E_\rho [\Delta_N] \; =\; 
\frac{K_N}{\beta_N \delta} \, \mb E_\rho \Big[ \int_0^{D_1} \mb 1\{
X^N_{t} \not \in A_N\} \, dt \Big]\;. 
\end{equation*}
By Corollary~\ref{corED}, the previous expression equals
\begin{equation}
\label{e:boundingmu}
\frac{K_N}{\beta_N \delta} \, \nu_N (V_N \setminus A_N) \,
\mb E_\rho [D_1] \;=\;   \frac{K_N \nu_N (V_N \setminus A_N)}
{\beta_N \nu_N (A_N) \delta} \,
E_\rho \Big[\frac{W^N_x}{v_\ell(x)}\Big]\;\cdot
\end{equation}
By \eqref{B2}, this expression vanishes as
$N\uparrow\infty$. This proves \eqref{c05}.

Now we turn into the estimation of the distance between $\bar X^N_t$
and $Y^N_t$.  On the event $\mc G$, the first $K_N$ jumps of the
processes $\bar X^N_t$ and $Y^N_t$ are the same, and the process $Y^N_t$ is
always ``ahead of'' $\bar X^N_t$ in the sense that $\bar X^N_t$ spends
more time at each site than $Y^N_t$. We need to show that the delay
between $\bar X^N_t$ and $Y^N_t$ is small. Let $B_N =
\{x_1^N,\dots,x_{M'_N}^N\} \subseteq A_N$ be the set appearing in
conditions \eqref{B3} and \eqref{B4}, and let $\mf N_N$ be the number
of times the process $Y^N$ visits $B_N$ before $\sigma_{K_N}$:
\[
\mf N_N \;:=\; \# \{j < K_N : Y_{\sigma_j}^N \in B_N \}\;.
\]
Denote by $\mc G_1$ the event $\mc G \cap \{ \sigma_{K_N} \geq \beta_N
T\}$. Since we have that $d_{T} (\bar X^N_{\beta_N\cdot},
Y^N_{\beta_N\cdot}) = \beta_N^{-1} d_{\beta_N T}(\bar
X^N_\cdot,Y^N_\cdot)$, on the set $\mc G_1$, $d_{T} (\bar
X^N_{\beta_N\cdot}, Y^N_{\beta_N\cdot}) \leq \beta_N^{-1}
d_{\sigma_{K_N}}(\bar X^N_\cdot,Y^N_\cdot)$. Therefore, on the set
$\mc G_1$,
\begin{equation*}
\begin{split}
& d_{T} (\bar X^N_{\beta_N\cdot}, Y^N_{\beta_N\cdot}) \;\le\;
\frac 1{\beta_N} \sum_{j=1}^{K_N} \int_{\sigma_{j-1}}^{\sigma_j} 
{\mb 1}\{Y^N_t \not = \bar X^N_t\} \, dt \\
&\quad \le\; \frac 1{\beta_N} \int_0^{\sigma_{K_N}}
{\mb 1}\{Y^N_{t} \notin B_N\} \, dt \;+\;
\frac 1{\beta_N} \sum_{j=1}^{K_N} \int_{\sigma_{j-1}}^{\sigma_j} 
{\mb 1}\{Y^N_{t} \in B_N, Y^N_t \not = \bar X^N_t\} \, dt\;.
\end{split}
\end{equation*}
We claim that each integral in the second term of the previous sum is
bounded by $\Delta_N$. Indeed, the total delay of the process $\bar
X^N_t$ with respect to the process $Y^N_t$ in the interval
$[0,\sigma_{K_N}]$ is $D_{K_N}-\sigma_{K_N} = \Delta_N$.  On the other
hand, either the length of time interval $[\sigma_{j-1}, \sigma_j]$ is
bounded by $\Delta_N$, in which case the claim is trivial, or the
length is greater than $\Delta_N$. In this latter situation, since the
total delay between $Y$ and $\bar X$ in the interval
$[0,\sigma_{K_N}]$ is $\Delta_N$, $D_{j-1} - \sigma_{j-1} \le
\Delta_N$ for $1\le j\le K_N$. Hence, in the interval
$[\sigma_{j-1}+\Delta_N, \sigma_{j})$ we have that $\bar
X_t=Y_t$. This proves our assertion. In conclusion, if one recalls the
definition of $\mf N_N$, on the set $\mc G_1$,
\begin{equation*}
d_{T} (\bar X^N_{\beta_N\cdot}, Y^N_{\beta_N\cdot}) \;\le\;
\frac 1{\beta_N} \int_0^{\sigma_{K_N}}
{\mb 1}\{Y^N_{t} \notin B_N\} \, dt \;+\; \frac 1{\beta_N}
\Delta_N \mf N_{N}\;.
\end{equation*}
In conclusion,
\begin{equation*}
\begin{split}
& Q_N \big[ d_{T} (\bar X^N_{\beta_N\cdot}, Y^N_{\beta_N\cdot}) > \delta
\big] \\
&\quad \leq\; Q_N[\mc G_1^c] \;+\;  \frac{2}{\beta_N \delta} 
Q_N \Big[\int_0^{\sigma_{K_N}} {\mb 1}\{Y^N_t \notin B_N\} dt\Big] 
\;+\; Q_N \big[ \Delta_N \mf N_{N} > (1/2) \delta \beta_N \big] \;.
\end{split}
\end{equation*}
The first term vanishes as $N\uparrow\infty$ by \eqref{c10} and
\eqref{e:B5}. By Tchebyshev and Cauchy-Schwarz inequalities, $P[ZW
>\delta] = P[\sqrt{ZW} >\sqrt{\delta}] \leq (\delta^{-1}
E[Z]E[W])^{1/2}$ for any pair of nonnegative random variables $Z$,
$W$. Therefore, the sum of the second and third terms is bounded by
\begin{equation*}
\frac{2K_N}{\beta_N \delta} E_\rho\Big[\frac{W^N_x}{v_\ell(x)}
{\mb 1}\{x \notin B_N\}\Big] \;+\;
\sqrt{\frac{2 \, Q_N [\Delta_N] \, Q_N [\mf N_N]} 
{\delta\, \beta_N}} 
\end{equation*}
Since $Q_N [\mf N_N] = K_N \rho(B_N)$, by \eqref{e:boundingmu} this
expression is less than or equal to
\begin{equation*}
\frac{2K_N}{\beta_N \delta} E_\rho\Big[\frac{W^N_x}{v_\ell(x)}
{\mb 1}\{x \notin B_N\}\Big] \;+\;
\sqrt{\frac{2K_N^2}{\beta_N \delta} \, \frac{\nu(V^N\setminus A_N)}{\nu(A_N)} 
\, E_\rho\Big[\frac{W^N_x}{v_\ell(x)}\Big] \, \rho(B_N)}\;.
\end{equation*}
By assumptions \eqref{B3} and \eqref{B4}, this expression vanishes as
$N \uparrow \infty$.

To conclude the proof of the theorem it remains to show that 
\begin{equation}
\label{e:B5}
\lim_{N \to \infty} Q_N \big[\sigma_{K_N} \leq \beta_N T\big] \;=\;0\;.
\end{equation}
For any random variable $Z$ and any $T \geq 0$ such that $E[Z]
\geq 2T$, by Tchebycheff inequality we have that
\[
P[Z < T ] \;\leq\; \frac{4 \Var(Z)}{E[\,Z\,]^2}\;\cdot
\]
Note that
\[
Q_N[\sigma_{K_N}] \;=\; K_N \, E_\rho\Big[
\frac{W^N_x}{v_\ell(x)}\Big]\;, 
\quad \Var_{Q_N}(\sigma_{K_N}) \;\le\; 2\, K_N  \,
E_\rho\Big[ \Big(\frac{W^N_x}{v_\ell(x)}\Big)^2\Big]\;,
\]
and that, by assumption \eqref{B5}, $K_N E_\rho[W^N_x/v_\ell(x)] \geq
2\beta_N T$ for $N$ sufficiently large. By the previous
elementary inequality,
\[
Q_N \big[ \sigma_{K_N} \leq \beta_N T \big] \;\leq\; 
\frac{8\,  E_\rho\Big[ \Big(\frac{W^N_x}{v_\ell(x)}\Big)^2\Big]}
{K_N E_\rho\Big[ \frac{W^N_x}{v_\ell(x)}\Big]^2} 
\;\leq\; \frac{8\,\beta_N}
{K_N E_\rho\Big[ \frac{W^N_x}{v_\ell(x)}\Big]} \,
\frac{E_\rho\Big[ \Big(\frac{W^N_x}{v_\ell(x)}\Big)^2\Big]}
{\beta_N E_\rho\Big[ \frac{W^N_x}{v_\ell(x)}\Big]} \;\cdot
\]
By assumption \eqref{B5}, the first term of this expression vanishes
as $N\uparrow\infty$. The second one is equal to
\begin{equation*}
E_{\nu_{A}} \Big[\frac{W^N_x}{\beta_N v_\ell(x)}\Big]\;.
\end{equation*}
By ({\bf A2}) this expression is bounded uniformly in $N$.  This
concludes the proof of \eqref{e:B5} and the one of Proposition
\ref{dXY}.
\end{proof}

In the previous proposition, instead of starting from the stationary
measure $\rho_N$, we may also start from any state $z^N$, provided its
asymptotic $\rho_N$-measure does not vanish with $N$, as in the
hypothesis of Corollary \ref{s05}.

The following remark will be important when proving Theorem \ref{t:random}
\begin{remark}
\label{s12}
Assumption {\rm ({\bf A0})} has only been used in Lemma \ref{s00} to
prove the existence of a sequence of subsets $B_N$ satisfying
\eqref{e:c2}, \eqref{e:c3}. In particular, Theorem \ref{dXYf} remains
in force if hypothesis {\rm ({\bf A0})} is replaced by the existence
of a sequence $I_N \le M_N$, $I_N\uparrow\infty$, for which $B_N 
=\{x^N_1, \dots x^N_{I_N}\}$ satisfies \eqref{e:c2} and
such that
\begin{equation}
\label{14}
\lim_{N\to\infty} |B_N|\, \nu_N(A^c_N) \;=\;0\;.
\end{equation}
\end{remark}

\section{$K$-processes}
\label{sec5}

We introduce in this section $K$-processes, a class of strong Markov
processes on $\overline{\mathbb{N}} = \bb N \cup\{\infty\}$ with one
fictitious state. We refer to \cite{fm1} for historical remarks and to
\cite{p1} for a detailed presentation and the proofs omitted here. The
main result of this section presents sufficient conditions for the
convergence of a sequence of finite-state Markov processes to a
$K$-process.

Throughout this section we fix two sequences of positive real numbers
$\{u_k : k \in \mathbb{N}\}$ and $\{Z_k : k \in \mathbb{N}\}$.  The
first sequence represents the `entrance measure' and the second one
the `hopping times' of the $K$-process.  The only assumption we make
over these sequences is that
\begin{equation}
\label{e:uZ}
\sum_{k \in \mathbb{N}} Z_k u_k < \infty.
\end{equation}
However, the process will be more interesting in the case
\begin{equation}
\label{e:sumu}
\sum_{k \in \mathbb{N}} u_k = \infty\;.
\end{equation}
If this sum is finite, the $K$-process associated to the sequences
$u_k$ and $Z_k$ corresponds to a Markov process on $\bb N$ with no
fictitious state.

Consider the set $\overline{\mathbb{N}}$ of non-negative integers with
an extra point denoted by $\infty$.  We endow this set with the metric
induced by the isometry $\phi:\overline{\mathbb{N}} \to \mathbb{R}$
which sends $n \in \overline{\mathbb{N}}$ to $1/n$ and $\infty$ to
$0$.  This makes the set $\overline{\mathbb{N}}$ into a compact metric
space. We use the notation $\textnormal{dist}(x,y) =
|\phi(y)-\phi(x)|$ for this metric.

For each $k \in \mathbb{N}$, define independent Poisson process
$\{N^k_t : t \geq 0\}$ with jump rate given by $u_k$.  Denote by
$\sigma^k_i$, $i \geq 1$, the time of the $i$-th jump performed by the
process $N^k_t$.  Independently from the Poisson processes, let
$\{T_0, T^k_i; k \in \mathbb{N}, i \geq 1\}$ be a collection of mean
one independent exponential random variables.

Let $Z_\infty=0$ and for $y \in \overline{\mathbb{N}}$ consider the
process
\begin{equation*}
\Gamma^y(t) = Z_y T_0 + \sum_{k \in \mathbb{N}} Z_k \sum_{i = 1}^{N^k_t} T^k_i.
\end{equation*}
Define the $K$-process with parameter $(Z_k, u_k)$, starting from $y$
as follows
\begin{equation}
\label{e:Xyt}
X^y(t) =
\begin{cases}
y & \text{ if $0 \leq t < Z_y T_0$,}\\
k & \text{ if $\Gamma^y(\sigma^k_i -) \leq t 
< \Gamma^y(\sigma^k_i)$ for some $i \geq 1$ and}\\
\infty & \text{ otherwise}.
\end{cases}
\end{equation}
Note that $X^y(0)=y$ almost surely if $y \in \mathbb N$, and even in
the case $y = \infty$ if \eqref{e:sumu} holds. We summarize in the
next result the main properties of the process $X^y_t$. Its proof can
be found in \cite{p1} or adapted from \cite{fm1} where the case in
which $u_k=1$ for all $k\ge 1$ is examined. Recall that we denote by
$H_A$ the hitting time of a set $A$ and that $Z_\infty =0$.

\begin{theorem}
\label{s02}
For any $y\in \overline{\mathbb{N}}$, the process $\{X^y(t) : t\ge
0\}$ is a strong Markov process on $\overline{\mathbb{N}}$ with
right-continuous paths with left limits. Being at $k\in\bb N$, the
process waits a mean $Z_k$ exponential time at the end of which it
jumps to $\infty$. For any finite subset $A$ of $\bb N$, $H_A$ is
a.s. finite and
\begin{equation*}
\mb P \big[X^y (H_A) = j \big] \;=\; \frac {u_j}{\sum_{i\in A}
  u_i}\;, \quad j\in A\;.
\end{equation*}
\end{theorem}

We investigate in this section the convergence of a sequence of Markov
processes in finite state spaces towards the process $X^y(t)$.  Let
$\{M_N:N\ge 1\}$ be a sequence of integers such that
$M_N\uparrow\infty$, and consider the sequences of positive real
numbers
\begin{equation}
\label{e:uM}
u^N_k \;, \;\; Z^N_k \;, \quad 1\le k\le M_N \;, \quad
N \geq 1\;.
\end{equation}
In analogy with \eqref{e:Xyt}, we define processes $X^y_N(t)$ with
`entrance measure' given by $u^N_k$ and `hopping times' given by
$Z^N_k$. For $N \geq 1$, let $T^N_0$, $T^{N,k}_i$, $N^{N,k}_t$ and
$\sigma^{N,k}_i$, $1\le k\le M_N$, $i\ge 1$, be defined as above and
write
\begin{equation*}
\Gamma^y_N(t) = Z^N_y T^N_0 + \sum_{k=1}^{M_N} 
Z^N_k \sum_{i = 1}^{N^{N,k}_t} T^{N,k}_i\;, \text{ for $1\le y\le M_N$}
\end{equation*}
and
\begin{equation}
\label{e:Xynt}
X^y_N(t) =
\begin{cases}
y & \text{ if $0 \leq t < Z^N_y T^N_0$,}\\
k & \text{ if $\Gamma^y_N(\sigma^{N,k}_i -) \leq t <
  \Gamma^y_N(\sigma^{N,k}_i)$ for some $i\ge 1$}.
\end{cases}
\end{equation}

One can easily see that the process $X^y_N$ is a continuous-time
c\`adl\`ag, Markov chain over $\{1, \dots, M_N\}$.  The order in which
the points $\{1,\dots, M_N\}$ are visited by $X^y_N$, after the
starting position, is given by the order of the times
$\sigma^{N,k}_i$.  From this fact we can conclude that the law of
$X^y_N$ is characterized by the following properties:
\begin{itemize}
\item The state space is $\{1, \dots, M_N\}$ and the process starts
  from $y$ almost surely,
\item The process $X^y_N$ remains at any site $k$ an exponential time
  with mean $Z^N_k$, after which it jumps to a site $j$ with
  probability $u^N_j/\sum_{1\le i\le M_N} u^N_i$.
\end{itemize}

\begin{remark}
\label{s14}
Note that the dynamics of the process $X^y_N$ does not change if one
replaces the vector $\{u^N_k : 1\le k\le M_N\}$ by the vector
$\{\gamma_N u^N_k : 1\le k\le M_N\}$ for some $\gamma_N>0$. In
particular, when applying the theorem below we may multiply the
sequence $u^N_k$ by a constant $\gamma_N$ to ensure the convergence of
$\gamma_N u^N_k$ to $u_k$.
\end{remark}

The main result of this section is stated below. Recall from
\cite{EK86}, (5.2) the definition of the Skorohod's $J_1$ topology.

\begin{theorem}
\label{l:XntoX}
Assume that for every $k \in \mathbb{N}$
\begin{equation}
\label{e:Zuconv}
\lim_{N \to \infty} (Z^N_k, u^N_k) = (Z_k,u_k)
\end{equation}
and that 
\begin{equation*}
\lim_{m\to\infty}\limsup_{N\to\infty} \sum_{k=m}^{M_N} Z^N_k \, u^N_k \;=\;0\;.
\end{equation*}
Then, for any given $y \in \mathbb{N}$, $X^y_N$ converges weakly, as
$N\uparrow\infty$, towards $X^y$ in the Skorohod's $J_1$ topology.
\end{theorem}

\begin{proof}
The proof is a modification of the one of Lemma~3.11 in \cite{fm1}.
We first couple the Poisson point processes used to define
$\Gamma^y_N$ and $\Gamma^y$.  In some probability space $(\Omega,
\mathcal{A}, \mb Q)$ we construct a collection $\{N^k : k \in
\mathbb{N}\}$ of Poisson point processes in $\mathbb{R}_+ \times
\mathbb{R}_+$ with respect to the Lebesgue measure.  Let $N^k(u,t)$ be
the number of points falling in the rectangle $[0,t] \times [0,u]$.
For fixed $k \in \mathbb{N}$ and $u \geq 0$, $N^k(u,\,\cdot\,)$ is
distributed as a Poisson counting process with rate $u$.  Define
$\Gamma^y$ and $\Gamma^y_N$ as before, but using these coupled arrival
processes, with corresponding intensities $u_k$ and $u^N_k$.
Moreover, we also use the same jump clocks $\{T^k_i : k\in \bb N \,,\,
i \geq 1\}$ in their constructions.

Fix an integer $m \in \mathbb{N}$ and denote by $\{S^m_i : i\ge 1\}$
the arrival times of the process $N^1(u_1, \,\cdot\,) + N^2(u_2,
\,\cdot\,) + \cdots + N^m(u_m, \,\cdot\,)$, with $S^m_0 = 0$.  Fix
$T \geq 0$ and let
\begin{equation*}
L^m_T = \inf \{i \geq 1; \Gamma^y_N(S_i^m) \geq T 
\text{ for every $N \geq 1$} \}.
\end{equation*}
Since $(Z^N_1, u^N_1)$ converges to $(Z_1,u_1)$ and since
$\Gamma^y_N(s) \geq \sum_{1\le i \le N^1 (u^N_1, s)} Z^N_1 T^1_i$, by
the law of large numbers the above infimum is finite.

Since the sequence $\{u_k : k\in\bb N\}$ is not summable, there exists
a random integer $m'$ large enough so that almost surely
\begin{equation}
\label{l01}
\sum_{k = m+1}^{m'} Z_k \sum_{j =
  N^k(u_k, S^m_i)}^{N^k(u_k, S^m_{i+1}-)} 
T^k_j > 0 \;, \quad i=0, \dots, L^m_T\; ,
\end{equation}
where $f(s-)$ stands for the left limit at $s$ of a c\`adl\`ag
function $f$.

Since $u^N_k$ converges to $u_k$, almost surely there exists $N(m)$
such that
\begin{equation}
\label{e:NuN}
N^k(u^N_k, t) \;=\; N^k(u_k, t)
\end{equation}
for all $1\le k\le m$, $0\le t \le S^m_{L^m_T}$ and all $N\ge N(m)$.
By possibly increasing $N(m)$ we can also assume that,
\begin{equation}
\label{l02}
\inf_{N \geq N(m)} \sum_{k = m+1}^{m'} Z^N_k 
\sum_{j = N^k(u^N_k, S^m_j)}^{N^k(u^N_k, S^m_{j+1}-)} 
T^k_j \;>\; 0 \;, \quad i = 0, \dots, L^m_T\;.
\end{equation}

It follows from \eqref{e:NuN} that the arrival times $S^m_i$ are the
same for the process $X^y$ and $X^y_N$. Furthermore, by \eqref{l01},
\eqref{l02}, on each interval $(S^m_i, S^m_{i+1})$ there is at least
one arrival of a Poisson process $N^k(u_k, \cdot)$ for some $k>m$ and
one arrival for a Poisson process $N^k(u^N_k, \cdot)$ for some
$k>m$. In particular, in the time interval $[\Gamma^y(S^m_i),
\Gamma^y(S^m_{i+1}-))$ (resp. $[\Gamma^y_N(S^m_i),
\Gamma^y_N(S^m_{i+1}-))$), $0\le i < L^m_T$, the process $X^y$
(resp. $X^y_N$) performs an excursion in the set $\{1, \dots, m\}^c$,
while on each time interval $[\Gamma^y(S_i-), \Gamma^y(S_{i}))$
(resp. $[\Gamma^y_N(S_i-), \Gamma^y_N(S_{i}))$), $1\le i \le L^m_T$,
the processes $X^y$ and $X^y_N$ sit on the same site of $\{1, \dots,
m\}$.

For $N \geq N(m)$, define the time changes $\lambda^m_N:
[0,\Gamma^y_N(S^m_{L^m_T})] \to \mathbb{R}_+$ by
\begin{equation*}
\lambda^m_N(t) \;=\; 
\frac{Z_y}{Z^N_y} \, t \quad\text{for}\quad 0 \leq t < Z^N_yT_0\;.
\end{equation*}
For $0 \leq i \leq L^m_T - 1$, let
\begin{equation*}
\lambda^m_N(t) \;=\;
\Gamma^y(S^m_i) \;+\; \frac{\Gamma^y(S^m_{i+1}-) - \Gamma^y(S^m_i)}
{\Gamma^y_N(S^m_{i+1}-) - \Gamma^y_N(S^m_i)} \, [t-\Gamma^y_N(S^m_{i})]
\end{equation*}
if $\Gamma^y_N(S^m_i) \leq t \leq \Gamma^y_N(S^m_{i+1}-)$ and let
\begin{equation*}
\lambda^m_N(t) \;=\; \Gamma^y(S^m_{i+1}-) \;+\; 
\frac{\Gamma^y(S^m_{i+1}) - \Gamma^y(S^m_{i+1}-)}
{\Gamma^y_N(S^m_{i+1}) - \Gamma^y_N(S^m_{i+1}-)} \,
[t-\Gamma^y_N(S^m_{i+1}-)]
\end{equation*}
if $\Gamma^y_N(S^m_{i+1}-) \leq t\leq \Gamma^y_N(S^m_{i+1})$.

In view of our previous discussion, 
\begin{equation}
\label{e:diam}
\textnormal{dist}(X^y(\lambda^m_N(t)), X^y_N(t)) 
\leq (1+m)^{-1}, \text{ for every $t \leq T$}.
\end{equation}
Indeed, whenever $X^y(\lambda^m_N(t))$ differs from $X^y_N(t)$, they
are both above $m$, and the diameter of the set $\{m+1, m+2, \dots\}$
under dist$(\cdot,\cdot)$ is given by $(m+1)^{-1}$.

We claim that $\lambda^m_N$ is close to the identity: for any $\delta
> 0$,
\begin{equation}
\label{e:mbQtozero}
\lim_m \; \limsup_{N} \; \mb Q \Big[ \sup_{0\leq t \leq T} 
|\lambda^m_N(t)-t| > \delta \Big] = 0.
\end{equation}
To prove this claim, fix $m \geq 1$ and note that
\begin{equation*}
\sup_{0\leq t \leq T} |\lambda^m_N(t) - t| \;\leq\; 
\max_{0 \leq i \leq L^m_T} \Big\{|\Gamma^y(S^m_i) -
\Gamma^y_N(S^m_i)| \vee |\Gamma^y(S^m_i-) - \Gamma^y_N(S^m_i-)|\Big\}
\;. 
\end{equation*}
By construction, the right hand side is bounded above by
\begin{equation}
\label{l03}
\begin{split}
& |Z^N_y-Z_y|\, T_0 \;+\; \sum_{k = 1}^m |Z^N_k - Z_k| 
\sum_{j = 1}^{N^k(u^N_k, S^m_{L^m_T})} T^k_j \\
& \qquad + \sum_{k = m+1}^{\infty} Z_k 
\sum_{j = 1}^{N^k(u_k, S^m_{L^m_T})} T^k_j  \;+\; 
\sum_{k = m+1}^{M_N} Z^N_k \sum_{j = 1}^{N^k(u^N_k, S^m_{L^m_T})} T^k_j\;.
\end{split}
\end{equation}
For each fixed $m$, the first two terms vanish almost surely as $N$
goes to infinity. To estimate the other two terms note that $L^m_T \ge
L^{m+1}_T$, that $S^{m+1}_{L^{m+1}_T} \le S^m_{L^m_T}$ and that $N^k$,
$\{T^k_j : j\ge 1\}$ are independent of $S^m_{L^m_T}$ for $k>m$. In
particular, for $k>m$ and $u>0$,
\begin{equation*}
E_{\mb Q} \Big[ \sum_{j = 1}^{N^k(u, S^m_{L^m_T})} T^k_j\Big] \;=\; 
u \, E_{\mb Q} \big[ S^m_{L^m_T} \big] \;\le \; 
u \, E_{\mb Q} \big[ S^1_{L^1_T} \big]\;.
\end{equation*}
Last expectation is bounded because $S^1_{L^1_T}$ is defined through a
Poisson process. Therefore, as $Z_k u_k$ is summable in $k$, the third
term in \eqref{l03}, which does not depend on $N$, has finite
expectation and converges to zero almost surely and in $L^1(\mb Q)$ as
$m$ tends to infinity. Similarly, 
\begin{equation*}
E_{\mb Q} \Big[ \sum_{k = m+1}^{M_N} Z^N_k \sum_{j = 1}^{N^k(u^N_k,
  S^m_{L^m_T})} T^k_j \Big] \;\le\; \sum_{k \ge m+1} Z^N_k \, u^N_k\,
E_{\mb Q} \big[ S^1_{L^1_T} \big]\;.
\end{equation*}
By assumption, this expression vanishes as $N\uparrow\infty$ and then
$m\uparrow\infty$. This proves that \eqref{e:mbQtozero} holds in fact
in $L^1(\mb Q)$.

As a consequence of \eqref{e:mbQtozero}, one can extract a sequence
$m_N$ growing slowly enough such that
\begin{equation*}
\sup_{0 \leq t \leq T} |\lambda_N^{m_N} -t| 
\text{ converges to zero in probability as $N\uparrow\infty$}\;.
\end{equation*}
This, together with \eqref{e:diam} provides the two conditions of
Proposition~5.3 (c) in \cite{EK86}. Hence, $X^y_N$ converges in
probability to $X^y$ in the Skorohod's $J_1$ topology as $N$ tends to
infinity.
\end{proof}

\section{Scaling limit of trap models}
\label{sec8}

In this section we join the results of the last three sections to establish
the asymptotic behaviour of random walks on vertex-weighted graphs.

Throughout this section, we restrict our attention to weights given
by an i.i.d. sequence of random variables in the basin of attraction
of an $\alpha$-stable distribution, as in \eqref{13}. Let us first collect
some consequences of this choice of random variables. In particular we
obtain the convergence of the environment to a limiting distribution.

Recall that $\alpha\in(0,1)$ is the parameter of the stable
distribution. Let $\lambda$ be the measure on $\bb R \times
(0,\infty)$ given by $\lambda = \alpha w^{-(1+\alpha)} dx\, dw$.
Denote by $\{(z_i,\hat w_i) \in \bb R \times (0,\infty): i\ge 1\}$ the
marks of a Poisson point process of intensity $\lambda$ independent of
the sequence of graphs $\{G_N :N\ge 1\}$ and defined on a probability
space $(\Omega', \mc F', P)$. Define the random measure $\zeta$ on
$\bb R$ by
\begin{equation}
\label{09}
\zeta \;=\; \sum_{i\ge 1} \hat w_i \, \delta_{z_i}\;,
\end{equation}
and let $\zeta_t =\zeta((0,t])$, $t\ge 0$, be the $\zeta$-measure of the
interval $(0,t]$.  Let $F:[0,\infty) \to [0,\infty)$ be defined by
\begin{equation*}
P[\zeta_1 > F(t) ] \;=\; \bb P[W^N_x >t]\;, \quad t\ge 0\;.
\end{equation*}
The function $F$ is non-decreasing and right-continuous. Denote its
right-continuous generalized inverse by $F^{-1}$ and let
\begin{equation}
\label{52}
\hat{\tau}^N_i \;=\; F^{-1} \big (\bb V^{1/\alpha}[\zeta_{i/\bb V} 
- \zeta_{(i-1)/\bb V}]\big) \;, \quad 1\le i\le \bb V\;.
\end{equation}
Denote by $\tau^N_i$, $1\le i\le \bb V$, the sequence $\hat{\tau}^N_i$ in
decreasing order: $\hat \tau^N_i ={\tau}^N_{\sigma(i)}$ for some
permutation $\sigma$ of $\{1, \dots, \bb V\}$ and $\tau^N_i \ge
\tau^N_{i+1}$.

By \cite[Proposition 3.1]{fin}, $\{\hat\tau^N_i : 1\le i\le \bb V\}$
has the same distribution as $\{W^N_{x} : x\in V_N\}$. Therefore,
$(\tau^N_1, \dots, \tau^N_\bb V)$ has the same distribution as
$(W^N_{x^N_1}, \dots, W^N_{x^N_\bb V})$. Moreover, since $\bb V_N = |V_N|
\to \infty$ $\bb P$-almost surely, the same result implies that
$(\bb P \times P)$-almost surely,
\begin{equation}
\label{08}
\lim_{N\to\infty} \sum_{j\ge 1} |\, c_{\bb V} \tau^N_j - w_j\,| \;=\; 0 \;, 
\end{equation}
where $\mb W = \{w_i : i\ge 1\}$ represents the weights in decreasing
order of the measure $\zeta$ restricted to $[0,1]$:
\begin{equation}
\label{10}
\begin{split}
& w_1 = \max\{ \hat w_i : z_i \in [0,1]\}\;, \quad \\
& \quad w_{j+1} = \max\{ \hat w_i : z_i \in [0,1] \,,\, \hat w_i \not\in
\{w_1, \dots, w_j\}\}\;, \quad j\ge 1\;,  
\end{split}
\end{equation}
and $\{c_k : k\ge 1\}$ is the sequence defined by \eqref{aa1}.

Recall the definition of the function $\Psi_N$ introduced just before
the statement of Theorem \ref{t:trans}.

\begin{theorem}
\label{s08}
Let $G_N=(V_N,E_N)$ be a sequence of finite vertex-weighted graphs
fulfilling assumptions {\rm ({\bf A0})--({\bf A2})} for some sequences
$M_N$, $\ell_N$. Assume, furthermore, that there exist sequences
$L_N\uparrow\infty$, $\{\beta_N : N\ge 1\}$ and $\{\gamma_N : N\ge 1\}$
such that 
\begin{equation}
\label{11}
\lim_{N\to\infty} \kappa(L_N, M_N, \ell_N) \;=\; 0\;,
\end{equation}
and such that 
\begin{equation}
\label{04}
\begin{split}
& \lim_{N \to \infty} \Big( \frac{W^N_{x_j}}{\beta_N v_\ell(x_j)}
\,,\, \gamma_N \, v_\ell(x_j)\, \deg(x_j) \Big) = (Z_j,u_j)\;, \quad 
\text{for all } j\ge 1\;, \\
&\quad 
\lim_{m\to\infty}\limsup_{N\to\infty} \sum_{j= m}^{M_N} 
\frac{W^N_{x_j}}{\beta_N v_\ell(x_j)} \, \gamma_N \, v_\ell(x_j)\, 
\deg(x_j) \;=\;0\;.
\end{split}
\end{equation}
Suppose, finally, that $\Psi_N(X^N_{0})$ converges weakly to $k\in\bb
N$.  Then, for every $T>0$, the Markov chain $\{\Psi_N(X^N_{\beta_N
  t}) : 0\le t\le T\}$ converges to the $K$-process with parameters
$(Z_j,u_j)$ starting from $k$, in the topology introduced in Section
\ref{sec4}.
\end{theorem}

\begin{proof}
  Repeating the arguments presented below \eqref{06}, we obtain a new
  sequence $M'_N$ for which ({\bf A3}) holds, as well as \eqref{04}
  with $M'_N$ instead of $M_N$. Denote this new sequence by $M_N$.
  Under assumptions ({\bf A0})--({\bf A3}), Theorem \ref{dXYf}
  furnishes a coupling between the random walk $X^N_{\beta_N t}$ and a
  Markov process $Y^N_{\beta_N t}$ on $\{1, \dots, M_N\}$ whose
  $d_T$-distance converges to $0$ in probability. In view of Remark
  \ref{s14} and by Theorem \ref{l:XntoX}, under conditions \eqref{04},
  the Markov process $Y^N_{\beta_N t}$ converges to the $K$-process
  with parameters $(Z_j,u_j)$ in the Skorohod's $J_1$ topology. By
  Skorohod's representation theorem, there exists a probability space
  in which this convergence take place almost surely. It remains to
  apply Corollary \ref{c1}.
\end{proof}

In view of Remark \ref{s12}, we may replace condition ({\bf A0}) by
assumptions \eqref{e:c2} and \eqref{14}. 

\begin{theorem}
\label{s13}
Let $G_N=(V_N,E_N)$ be a sequence of finite vertex-weighted graphs
fulfilling assumptions {\rm ({\bf A1})--({\bf A3})} for some sequences
$M_N$, $\ell_N$, $L_N$. Assume that there exists a sequence of subsets
$B_N = \{x^N_1, \dots, x^N_{I_N}\}$, $I_N\le M_N$,
$I_N\uparrow\infty$, satisfying \eqref{e:c2}, \eqref{14}. Suppose,
furthermore, that condition \eqref{04} is in force and that
$\Psi_N(X^N_{0})$ converges weakly to $k\in\bb N$.  Then, for every
$T>0$, the Markov chain $\{\Psi_N(X^N_{\beta_N t}) : 0\le t\le T\}$
converges to the $K$-process with parameters $(Z_j,u_j)$ starting from
$k$, in the topology introduced in Section \ref{sec4}.
\end{theorem}

\begin{proof}
By Remark \ref{s12}, there exists a coupling between the random walk
$X^N_{\beta_N t}$ and a Markov process $Y^N_{\beta_N t}$ on $\{1,
\dots, M_N\}$ whose $d_T$-distance converges to $0$ in probability. By
Theorem \ref{l:XntoX}, under conditions \eqref{04}, the Markov process
$Y^N_{\beta_N t}$ converges to the $K$-process with parameters
$(Z_j,u_j)$ in the Skorohod's $J_1$ topology. By Skorohod
representation theorem, there exists a probability space in which this
convergence take place almost surely. It remains to apply Corollary
\ref{c1}.  \end{proof}

\section{Pseudo-transitive graphs}
\label{sec7}

We prove in this section Theorem \ref{t:trans}, inspired by Theorem
\ref{s13}, and we apply this result to some pseudo-transitive
graphs. The assumptions ({\bf A1})--({\bf A3}), \eqref{e:c2},
\eqref{14}, \eqref{04} simplify in this context because the degree and
the escape probability from the deep traps do not depend on the
specific vertex.

\begin{proof}[Proof of Theorem \ref{t:trans}]
Fix an increasing sequence $\ell_N$ and a sequence of
pseudo-transitive graphs $G_N$ with respect to the sequence $\ell_N$.
We first derive some consequences of assumptions ({\bf B0})--({\bf
  B2}) and \eqref{pseudo}.

It follows from these hypotheses that there exists an increasing
sequence $M_N\uparrow\infty$ such that
\begin{equation*}
\begin{split}
& \lim_{N\to\infty} M_N^2 \, \bb E \Big[ \frac{|B(\mf x,
  2\ell_N)|}{\bb V_N} \Big] \;=\; 0\;, \quad
\lim_{N\to\infty} M_N \, \bb P \big[ ( \mf x, B(\mf x, \ell_N)) \not \equiv
(\mf y, B(\mf y, \ell_N))  \big] \;=\; 0\;, \\
& \quad \lim_{N\to\infty} M_N^5 \, \bb E \Big[ 
\sup_{\substack{y \not \in B(\mf x, \ell_N)}} 
\mb P_y \big[\bb H_{\mf x} \leq L_Nt_{\text{mix}} \big] \, \Big] 
\;=\; 0\;, \quad \lim_{N\to\infty} M_N^3 \, 2^{-L_N} \;=\; 0\;.
\end{split}
\end{equation*}
where we replaced $(\mf x, B(\mf x,\ell_N)) \not \equiv (\mf y, B(\mf
y,\ell_N))$ by $B(\mf x,\ell_N) \not \equiv B(\mf y,\ell_N)$.

Let $\Sigma_N^j$, $1\le j\le 3$ be the events 
\begin{equation*}
\begin{split}
& \Sigma_N^1 \;=\; \bigcap_{1\le i\not = j\le M_N} 
\big\{ B(x^N_i, \ell_N) \cap B(x^N_j, \ell_N) = \varnothing
\big\}\;, \\
&\quad \Sigma_N^2 \;=\; \bigcap_{j=1}^{M_N} \big\{ (x^N_1, B(x^N_1, \ell_N))
\equiv (x^N_j, B(x^N_j, \ell_N)) \big\}\;, \\
&\qquad \Sigma_N^3 \;=\;  \Big\{ M^3_N \max_{1\le j\le M_N}
\sup_{\substack{y \not \in B(x^N_j, \ell_N)}} 
\mb P_y \big[\bb H_{x^N_j} \leq L_N t_{\text{mix}} \big] 
\le M^{-1}_N \Big\}\;. 
\end{split}
\end{equation*}
In the places where the vertices of the graph appear, as in the
definition of the set $\Sigma^1_N$, the sequence $M_N$ obtained above
has to be replaced by $\min\{M_N, \bb V_N\}$, where $\bb V_N$ stands
for the number of vertices of the random graph $G_N$.  It is easy to
see that all three events have probability asymptotically equal to
one. We prove this assertion for $\Sigma_N^1$ and leave to the reader
the proof for the other two. By definition, $\bb P[(\Sigma_N^1)^c]$ is
bounded above by
\begin{equation*}
\sum_{1\le i\not = j\le M_N} \bb P \big[ B(x^N_i, \ell_N) 
\cap B(x^N_j, \ell_N) \not = \varnothing \big] \;\le\;
M_N^2\, \bb P \big[ B(x^N_1, \ell_N) 
\cap B(x^N_2, \ell_N) \not = \varnothing \big] 
\end{equation*}
because $x^N_1, \dots, x^N_{\bb V}$ is uniformly distributed. By this
same reason, conditioning on $x^N_1$, we obtain that the right-hand
side is equal to
\begin{equation*}
M_N^2 \, \bb E \Big[ \frac{|B(x^N_1, 2\ell_N)|-1}{\bb V_N-1} \Big]\;,
\end{equation*}
which vanishes as $N\uparrow \infty$ in view of the definition of the
sequence $M_N$.

Let $A_N = \{x^N_1, \dots, x^N_{M_N}\}$. By hypothesis ({\bf B0}),
$\nu_N(A^c_N)$ converges to $0$ in $\bb P$-probability. In particular,
there exists a deterministic sequence $I_N \uparrow\infty$, $I_N\le
M_N$, such that $I_N \nu_N(A^c_N)$ converges to $0$ in $\bb
P$-probability. Let $B_N = \{x^N_1, \dots x^N_{I_N}\}$. Since $I_N
\uparrow\infty$, by hypothesis ({\bf B0}), $\nu_N(B^c_N)$ converges to
$0$ in $\bb P$-probability.  Therefore, there exists a sequence
$\epsilon_N \downarrow 0$ for which
\begin{equation*}
\lim_{N\to\infty} \bb P \Big[ \nu_N(B^c_N) + I_N \nu_N(A^c_N) \ge 
\epsilon_N \, \Big] \;=\; 0\;.
\end{equation*}
Let $\Sigma_N^4 = \{ \nu_N(B^c_N) + I_N \nu_N(A^c_N) < \epsilon_N
\}$.

We turn now into the proof of the theorem which relies on Theorem
\ref{s13}. Recall the definition of the random weights $\hat\tau^N_j$,
$1\le j\le \bb V_N$, introduced at the beginning of Section
\ref{sec8}. Since $\{\hat\tau^N_j : 1\le j\le \bb V_N\}$ has the same
distribution as $\{W^N_{j} : 1\le j\le N\}$, we may replace the latter
random weights by the former and assume that the random walk $X^N_t$
evolves among random traps with depth $\tau^N_j$ instead of
$W^N_{x_j}$.

To show that the pair $(c_{\bb V} \tau^N, \Psi_N(X^N_{t \beta_N}))$
converges weakly to $(\mb w, K_t)$, it is enough to show that any
subsequence $\{N_j : j\ge 1\}$ possesses a sub-subsequence $\mf n$
such that $(c_{\mf n} \tau^{\mf n},\Psi_{\mf n}(X^{\mf n}_{t
  \beta_{\mf n}}) )$ converges to $(\mb w, K_t)$. Fix, therefore, a
subsequence $N_j$.

By \eqref{08}, the ordered sequence $(c_{N_j}\tau^{N_j}_1, \dots,
c_{N_j}\tau^{N_j}_{\bb V})$ converges almost surely in $L^1(\bb N)$ to
$\mb w =(w_1, w_2, \dots)$. This proves the weak convergence of the
first coordinate.  Let $\Sigma_{N_j} = \cap_{1\le k\le 4}
\Sigma^k_{N_j}$. There exists a sub-subsequence, denoted by $\mf n$,
for which
\begin{equation*}
\bb P \Big[ \bigcup_{\mf n_0\ge 1} \bigcap _{\mf n \ge \mf n_0}
\Sigma_{\mf n} \Big] 
\;=\; 1\;.
\end{equation*}

We affirm that all assumptions of Theorem \ref{s13} hold on the set
$\cup_{\mf n_0\ge 1} \cap _{\mf n \ge \mf n_0} \Sigma_{\mf
  n}$. Indeed, recall that $\beta^{-1}_{\mf n} = c_{\mf n} v^{\mf
  n}_{\ell_{\mf n}} (x^{\mf n}_1)$.  Condition ({\bf A1}) follows from
the definition of the set $\Sigma^1_{\mf n}$. On the set
$\Sigma^2_{\mf n}$, the escape probabilities $v_\ell(x^{\mf n}_j)$ and
the degrees $\deg(x^{\mf n}_j)$ are all the same for $1\le j\le M_{\mf
  n}$. In particular, by definition of the sequence $\beta_{\mf n}$,
condition ({\bf A2}) becomes
\begin{equation*}
\limsup_{\mf n \to \infty}  \frac{\sum_{j=1}^{J_{\mf n}} c_{\mf n} (\tau^{\mf
    n}_j)^2 } {\sum_{j=1}^{J_{\mf n}} \tau^{\mf n}_j } \;<\; \infty\;,
\quad
\limsup_{\mf n \to \infty} \frac 1
{\sum_{j=1}^{J_{\mf n}} c_{\mf n} \tau^{\mf n}_j } \;<\; \infty\;.
\end{equation*}
for all sequences $J_{\mf n}$ such that $J_{\mf n}\le M_{\mf n}$,
$J_{\mf n}\uparrow\infty$. Since the sequence $\tau_j^{\mf n}$ is
decreasing, the first ratio is bounded by $c_{\mf n} \tau^{\mf n}_1$,
and these bounds are a consequence of \eqref{08}. Condition ({\bf A3})
follows from the definition of the sequence $M_N$ and from the
definition of the set $\Sigma^3_{\mf n}$. Conditions \eqref{e:c2},
\eqref{14} follow from the definition of the set $\Sigma^4_{\mf
  n}$. Finally, on the set $\Sigma^2_{\mf n}$, $v_\ell (x^{\mf n}_j)
\deg(x^{\mf n}_j)$, $1\le j\le M_{\mf n}$, is constant and the
hypotheses \eqref{04} with $\gamma_{\mf n} = [v_\ell (x^{\mf n}_1)
\deg(x^{\mf n}_1)]^{-1}$ and $(Z_j,u_j) = (w_j, 1)$ follow from
\eqref{08}. This proves the affirmation.

We may now apply Theorem \ref{s13} to conclude that the Markov chain
$\Psi_{\mf n}(X^{\mf n}_{\beta_{\mf n} t})$ converges to the
$K$-process with parameters $(w_j,1)$ starting from $k$, in the
topology introduced in Section \ref{sec4}. This concludes the proof of
Theorem \ref{t:trans}.
\end{proof}

We conclude this section with some examples of graphs satisfying the
assumptions of Theorem~\ref{t:trans}.

\subsection{Hypercube}

We prove in this subsection the convergence of the trap model on the
$n$-dimensional hypercube towards the $K$-process associated to
constant entrance measure. This result has been established in
\cite{fli1} under the stronger Skorohod's $J_1$ topology with a
different approach. Here we give a proof as an application of
Theorem~\ref{t:trans}.

Let $N = 2^n$, $n \geq 1$, and let $G_N$ be the $n$-dimensional
hypercube $\{0,1\}^n$ with edges connecting any two points that differ
by only one coordinate. By estimate (6.15) in \cite{LPW09},
$t^N_{\textnormal{mix}} \ll n^2$.

\begin{proposition}
\label{p:hyper}
The assumptions of Theorem~\ref{t:trans} are in force for the
hypercube $G_N$ with $\ell_N = \log_2(N)/10 = n/10$.
\end{proposition}

\begin{proof}
Since the graph is transitive, condition \eqref{pseudo} is satisfied
and ({\bf B0}) follows from \eqref{08}. To estimate the ratio in
\eqref{BB0} note that $|B(0,2\ell_N)|/\bb V_N$ is equal to the
probability that the sum of $n$ Bernoulli($1/2$) independent random
variables is less than or equal to $2\ell_N = n/5$. By the law of
large numbers, this probability vanishes as $n\uparrow\infty$.

To show that \eqref{BB1} is in force, we could compare the distance
$d(0, X_t)$ with an Ehrenfest's urn, see \cite[Section~2.3]{LPW09},
and proceed with a calculation based on a birth and death chain. For
simplicity, we give instead a reference implying the result. By
Lemmas~3.6~(i) and 3.2 (i) of \cite{CG08}, with $m(N) = N^2$ and $a =
1$, there exists a finite constant $C_0$ independent of $n$ such that
\begin{equation*}
\begin{split}
& \sup_{y \not \in B(0, \ell_N)} \mb P_y \big[ \bb H_0 \leq n^2 \big] 
\;\leq\; C_0 \Big(n^2/N + \binom{n}{{n}/{10}}^{-1} n^{1/2}
\log(n)\Big) \\
& \qquad \leq\; C_0 \Big(n^2/N + (10)^{-n/10} n^{1/2} \log(n)\Big)\;,
\end{split}
\end{equation*}
which vanishes as $n\uparrow\infty$, proving \eqref{BB1}.
\end{proof}

To complete the description of the asymptotic behavior of the trap
model on the hypercube, it remains to determine the time scale
$\beta_N$. By a computation based on a birth-and-death chain, the
escape probability converges to $1$ as $N\uparrow\infty$, and
therefore $\lim_{N} \beta_N \, c_N =1$.

\subsection{Discrete torus for $d \geq 2$}

In this subsection the graph $G_N$ stands for the $d$-dimensional
discrete torus $\mathbb{T}^d_N = (\mathbb{Z}/ N\mathbb{Z})^d$, $d \geq
2$, endowed with nearest neighbors edges. By
\cite[Theorem~5.5]{LPW09},
\begin{equation}
\label{e:mixtor}
t_{\textnormal{mix}}^N \leq C_0 N^2
\end{equation}
for some $C_0 = C_0(d)$. This constant may change from line to line,
but will only depend on $d$.

We proved in \cite{jlt1} that in this context the trap model converges
to the $K$-process. The next proposition shows that this result
follows from Theorem \ref{t:trans}.

\begin{proposition}
\label{s11}
The assumptions of Theorem~\ref{t:trans} are in force for the
$d$-dimen\-sion\-al torus $G_N$ with 
\begin{equation*}
\ell_N =
\begin{cases}
N^{1/2} & \text{$d \geq 3$\,,} \\
\frac{N}{\log^{1/4} N} & \text{$d = 2$\,,}
\end{cases} 
\qquad
L_N =
\begin{cases}
\log^2 N & \text{$d \geq 3$\,,} \\
\log^{1/4} N & \text{$d = 2$\,.}
\end{cases}
\end{equation*}
\end{proposition}

\begin{proof}
Since the graph is transitive, condition \eqref{pseudo} is satisfied
and ({\bf B0}) follows from \eqref{08}. On the other hand, assumption
({\bf B1}) is clearly in force by definition of $\ell_N$.  It remains
to check hypothesis ({\bf B2}). Recall the definition of the sequence
$L_N$.  The case $d \geq 3$ follows directly from Lemma~3.1 of
\cite{TW10}, and we focus on the case $d = 2$. Fix $x\in \bb T^d_N$
and $z \not\in B(x, \ell_N)$. If $\Pi$ stands for the canonical
projection from $\mathbb{Z}^2$ to $\mathbb{T}^2_N$ and $\mc P_z$ for
the probability corresponding to the symmetric nearest neighbor
discrete time random walk on $\bb Z^2$,
\begin{equation*}
\mb P_z[\mathbb{H}_{x} < L_N t_{\textnormal{mix}}^N ] \;=\;
\mc P_{z}[\mathbb{H}_{\Pi^{-1}(x)} < L_N t_{\textnormal{mix}}^N ]\;.
\end{equation*}
We may bound the previous probability by
\begin{equation}
\label{e:AMNcopies}
\mc P_{z}[\mathbb{H}_{B(z,N \log^{1/4} N)^c} < L_N t_{\textnormal{mix}}^N]
\;+\; \sum_i \mc P_{z}[ \mathbb{H}_{x_i} < L_N t_{\textnormal{mix}}^N ]\;,
\end{equation}
where the sum is performed over all sites $x_i$ in the pre-image of
$x$ which belong to the ball $B(z,N \log^{1/4} N)$. 

The first term can be bounded using the estimate \eqref{e:mixtor} for
the mixing time and an exponential Doob inequality since each
component of the random walk is a martingale. This argument shows that
the first term is bounded by $4 \exp\{- a \log^{1/4} N\}$ for some
$a>0$. Since there are no more than $C_0\sqrt{\log N}$ terms in the
sum, the second expression in the previous decomposition is bounded
above
\begin{equation*}
C_0\sqrt{\log N} \, \mc P_{0}[ \mathbb{H}_{x} < L_N
t_{\textnormal{mix}}^N ]\;, 
\end{equation*}
where $x$ is a site at distance $\ell_N$ from the origin. Decomposing
this probability according to whether the random walk reached the ball
with radius $N \log^{1/4} N$ before time $C_0 N^2 \log^{1/4} N$ or
not, and recalling the argument employed to bound the first term in
\eqref{e:AMNcopies}, we conclude that the previous expression is
bounded by
\begin{equation*}
C_0\sqrt{\log N} e^{- a \log^{1/4} N} \, \;+\; C_0\sqrt{\log N} \, 
\mc P_{0}[ \mathbb{H}_{x} < \mathbb{H}_{B(0,N \log^{1/4} N)^c} ]
\end{equation*}
for some finite constant $C_0$ and some positive $a$. By
\cite[Proposition 1.6.7]{Law91} and the reversibility of the random walk,
the second term is less than or equal to
\begin{equation*}
\begin{split}
C_0 \sqrt{\log N} \, \Big(1 - \frac{\log \ell_N}{\log( N \log^{1/4}
  N)} + \frac{C_0}{\log^2 N}\Big) \;\le\; C_0 \log^{-1/4} N\;,
\end{split}
\end{equation*}
which proves condition ({\bf B2}). 
\end{proof}

To complete the description of the asymptotic behavior of the trap
model on the discrete torus $\bb T^d_N$, it remains to determine the
time scale $\beta_N$.  Let $v_d$, $d \geq 3$, be the escape
probability of a simple random walk on $\mathbb{Z}^d$, and let
\begin{equation*}
\beta'_N =
\begin{cases}
c^{-1}_{|\bb T^d_N|}\, (2/\pi) \log(N) &\text{$d = 2\,$, }\\
c^{-1}_{|\bb T^d_N|}\, v^{-1}_d &\text{$d \geq 3\,$.}
\end{cases}
\end{equation*}
In view of the definition of $\beta_N$ and of
\cite[Theorem~1.6.6]{Law91}, $\lim_{N\to\infty} \beta_N/\beta'_N =1$.

\subsection{\bf Random $d$-regular graphs.}

In this subsection we consider a sequence of graphs $G_N$ with $N$
vertices satisfying the following three assumptions.

\renewcommand{\theenumi}{{\bf G\arabic{enumi}}}
\renewcommand{\labelenumi}{(\theenumi)}

\begin{enumerate}
\item $G_N$ is $d$-regular for some $d\ge 3$;
\item There is a constant $\alpha > 0$ such that for any vertex $x$ of
  $V_N$, the ball $B(x,\alpha \log N)$ contains at most one cycle;
\item The spectral gap $\lambda_{N}$ of the continuous time random walk
  on $G_N$ is bounded below by some positive constant: $\lambda_N \ge
  \gamma>0$ for all $N\ge 1$.
\end{enumerate}

It follows from \cite[Remark~1.4]{CTW10} that these three hypotheses
hold, with probability approaching $1$ as $N\uparrow\infty$, for a
sequence of random $d$-regular graphs on $N$ vertices. They are also
satisfied by the so-called Lubotzky-Phillips-Sarnak graphs
\cite{LPS88}.

By \cite{SC97} p. 328, under conditions ({\bf G1}) and ({\bf G3}), the
mixing time $t^N_{\rm mix}$ is bounded above by $C_0 \log N$ for some
finite constant $C_0$.

\begin{proposition}
\label{p:dreg}
Let $\{G_N : N \geq 1\}$ be a sequence of random graphs defined on
some probability space $(\Omega, \mc F, \bb P)$ satisfying the
assumptions {\rm ({\bf G1})--({\bf G3})} with a $\bb P$-probability
converging to $1$ as $N\uparrow\infty$. Then, the conditions of
Theorem \ref{t:trans} are fulfilled with $L_N = \log N$ and $\ell_N =
\alpha' \log N$ for some $\alpha'< \min\{\alpha, [2\log
(d-1)]^{-1}\}$, where $\alpha$ is the constant appearing in condition
{\rm ({\bf G2})}.
\end{proposition}

\begin{proof}
Condition ({\bf B0}) follows from assumption ({\bf G1}) and
\eqref{08}. The rest of the proof is based on estimates obtained in
\cite{CTW10}.  

By \cite[Lemma~6.1]{CTW10} with $\Delta = \ell_N$, the probability
that a ball $B(x, \ell_N)$ is not a tree is bounded by
$(d-1)^{-(\alpha - \alpha') \log N}$.  Let $\Sigma_N$ be the event
\begin{equation}
\label{e:alltree}
\Sigma_N \;=\; \big\{
\text{$B(x^N_1,\ell_N)$ and $B(x^N_2,\ell_N)$ are disjoint trees} \big\}\;.
\end{equation}
We claim that $\bb P[\Sigma_N]$ converges to $1$ as
$N\uparrow\infty$. Indeed, if $\tilde \Sigma_N$ stands for the event
that $B(x^N_1,\ell_N)$, $B(x^N_2,\ell_N)$ are trees, in view of the
estimate of the previous paragraph, $\bb P[\tilde \Sigma_N^c]$ is
bounded by $2 (d-1)^{-(\alpha - \alpha') \log N}$ which vanishes as
$N\uparrow\infty$. On the other hand, since $|B(x_1,r)| \le 4
(d-1)^{r}$ for any ball in a $d$-regular graph and since $x^N_1$,
$x^N_2$ are uniformly distributed,
\begin{equation*}
\begin{split}
& \bb P \big[ B(x_1,\ell_N) \cap B(x_2,\ell_N)
\not = \varnothing \big] \; \le\; 4 \, 
\frac{(d-1)^{2\ell_N}}{N}\;\cdot
\end{split}
\end{equation*}
As $\alpha'< [2\log (d-1)]^{-1}$, this expression vanishes as
$N\uparrow\infty$. This proves the claim and assumption
\eqref{pseudo}, which clearly follows from the claim.  Condition
({\bf B1}) is also in force because $|B(x_1,2\ell_N)| \le 4 (d-1)^{2
  \ell_N}$.

It remains to examine the escape probability appearing in condition
({\bf B2}).  It follows from the bound for the mixing time presented
just before the statement of the proposition and from our choice of
the sequence $L_N$ that
\begin{equation*}
\mb P_z \big[ \mathbb{H}_{x} < L_N t_{\textnormal{mix}}^N \big ]
\;\le\; 
\mb P_z \big[ \mathbb{H}_{x} < C_0 (\log N)^2 \big ]\;.
\end{equation*}
By \cite[Lemma~3.4]{CTW10} with $r = 0$ and $s = \alpha' \log N$, the
previous expression for $z \not\in B(x, \ell_N)$, is bounded by $C_0
N^{-a}$ for some finite constant $C_0$ and some positive $a>0$.
This concludes the proof of the proposition.
\end{proof}

We conclude this section computing the scaling factor $\beta_N$ in the
context of graphs satisfying assumptions ({\bf G1})--({\bf G3}).  On
the event \eqref{e:alltree}, which has asymptotic probability equal to
one, $B(x_1, \ell_N)$ is a $d$-regular tree so that
\begin{equation*}
v_{\ell_N}(x_1) \;=\; \frac{d-2}{d-1} \Big(\frac 1{1- (d-1)^{-\ell_N}} \Big)\;.
\end{equation*}
In particular, $\lim_{N\to\infty} \beta_N c_N =  (d-1)/(d-2)$. 

\section{Graphs with asymptotically random conductances} 
\label{sec11}

We prove in this section Theorem \ref{t:random}. The proof follows the
one of Theorem~\ref{t:trans}. However, the absence of regularity of
the graph requires some extra effort in establishing ({\bf A2}).

Recall the coupling $\mc Q_N$ defined in \eqref{e:c4} between the
random graph $G_N$ and the sequence of i.i.d. random vectors
$\{(D_j,E_j) : j\ge 1\}$. We extend this coupling $\mc Q_N$ to a
coupling $\mc Q$ between all random graphs $G_N$ and the sequence of
i.i.d. random vectors $\{(D_j,E_j) : j\ge 1\}$ using $\mc Q_N$ as
the conditional probability:
\begin{equation*}
\mc Q \big[G_N = G \big| \{ (D_j, E_j) : j \geq 1\} \big] \;=\;
\mc Q_N \big[G_N = G \big| \{ (D_j, E_j) : j \geq 1\}\big]\;,
\end{equation*}
with the further condition that the graphs $G_N$, $N\ge 1$, are
conditionally independent, given $\{ (D_j, E_j) : j \geq 1\}$.
Include in the probability space just defined the random measure
$\zeta$ introduced in \eqref{09} which is associated to the marks of a
Poisson point process independent from the variables $(D_j,E_j)$ and
from the random graphs $G_N$. The probability measure on this new
space is still denoted by $\mc Q$.

Recall the definition of the random weights $\hat\tau^N_j$, $1\le j\le
\bb V_N$, introduced in Section \ref{sec1}. Since $\{\hat\tau^N_j :
1\le j\le \bb V_N\}$ has the same distribution as $\{W^N_{j} : 1\le j
\le |V_N|\}$, we may replace the latter random weights by the former
and assume that the random walk $X^N_t$ evolves among random traps
with depth $\tau^N_j$ instead of $W^N_{x_j}$.

Since $w_j$ is a.s. summable, since by \eqref{e:c4} $D_1/E_1$ has
finite $\mc Q$-expectation and since the sequences $\{w_j\}$ and
$\{(D_j,E_j)\}$ are independent,
\begin{equation}
\label{e:wDE}
\sum_{j \geq 1} w_j \, \frac{D_j}{E_j} \quad \text{is $\mc Q$-almost 
surely finite}\;.
\end{equation}
By the strong law of large numbers, almost surely
\begin{equation}
\label{lln}
\frac 1n \sum_{j = 1}^n D_j / E_j \;\leq\; C_1 
\end{equation}
for all large enough $n$, where $C_1 = 2 E_{\mc Q} [D_1 / E_1]$.

By hypotheses \eqref{BB0}--\eqref{e:c4}, there exists an increasing
sequence $M_N\uparrow\infty$ such that
\begin{equation*}
\begin{split}
& \lim_{N\to\infty} M_N^2 \, E_{\mc Q} \Big[ \frac{|B(\mf x,
  2\ell_N)|}{\bb V_N} \Big] \;=\; 0\;, \quad 
\lim_{N\to\infty} M_N^3 \, 2^{-L_N} \;=\; 0\;,\\ 
&\quad \lim_{N\to\infty} \mc Q \Big [\, 
\max_{1\le j\le M_N} \big \vert \, [v_{\ell}(\mf x_j)]^{-1} 
- E_j^{-1} \big\vert > M_N^{-2} \Big] \;=\; 0 \;, \\
&\qquad \lim_{N\to\infty} \mc Q \Big [\, 
\bigcup_{j=1}^{M_N}  \{ \deg(\mf x_j) \not = D_j \} \Big] \;=\; 0 \;, \\
& \qquad\quad 
\lim_{N\to\infty} M_N^5 \, E_{\mc Q} \Big[ 
\sup_{\substack{y \not \in B(\mf x, \ell_N)}} 
\mb P_y \big[\bb H_{\mf x} \leq L_Nt_{\text{mix}} \big] \, \Big] 
\;=\; 0\;\;.
\end{split}
\end{equation*}
As before, in the places where the vertices of the graph appear, as in
the definition of the set $\Sigma^1_N$, the sequence $M_N$ obtained
above has to be replaced by $\min\{M_N, \bb V_N\}$, where $\bb V_N$
stands for the number of vertices of the random graph $G_N$.

Let $\Sigma_N^j$, $1\le j\le 4$ be the events 
\begin{equation*}
\begin{split}
&\Sigma_N^1 \;=\; \bigcap_{1\le i\not = j\le M_N} 
\big\{ B(x^N_i, \ell_N) \cap B(x^N_j, \ell_N) = \varnothing
\big\}\;, \\
&\quad \Sigma_N^2 \;=\; \Big \{\, 
\max_{1\le j\le M_N} \big \vert \, [v_{\ell}(x^N_j)]^{-1} 
- E_j^{-1} \big\vert \le M_N^{-2} \Big\}, \\
&\qquad \Sigma_N^3 \;=\; \bigcap_{j=1}^{M_N}
\{ \deg(x^N_j) = D_j \} \;,  \\
&\quad\qquad \Sigma_N^4 \;=\;  \Big\{ M^3_N \max_{1\le j\le M_N}
\sup_{\substack{y \not \in B(x^N_j, \ell_N)}} 
\mb P_y \big[\bb H_{x^N_j} \leq L_N t_{\text{mix}} \big] 
\le M^{-1}_N \Big\}\;. 
\end{split}
\end{equation*}
Similarly to what was done in the proof of Theorem~\ref{t:trans}, we
can show that these events have probability asymptotically equal to
one. In the places where the vertices of the graph appear, as in the
definition of the set $\Sigma^1_N$, the sequence $M_N$ obtained above
has to be replaced by $\min\{M_N, \bb V_N\}$, where $\bb V_N$ stands
for the number of vertices of the random graph $G_N$.

By \eqref{08}, we may replace the sequence $M_N$ by a possibly random
increasing sequence $M'_N \le \min\{M_N,\bb V_N\}$,
$M'_N\uparrow\infty$ $\mc Q$-a.s., still denoted by $M_N$, for which
all the previous estimates hold and such that for all $N\ge 1$,
\begin{equation}
\label{e:ctau}
\sum_{j\ge 1} |c_{\bb V} \tau_j^N - w_j| \;\leq\;  M_N^{-2} \;.
\end{equation}
By hypothesis ({\bf B0}), even though the sequence $M_N$ is random,
the expectation $\bb E[\nu_N(\{x^N_1, \dots, x^N_{M_N}\}^c)]$ vanishes
as $N\uparrow\infty$.  Let $A_N = \{x^N_1, \dots, x^N_{M_N}\}$. As in
the proof of Theorem~\ref{t:trans}, presented in the previous
sections, using again hypothesis ({\bf B0}) we construct a set
$B_N=\{x^N_1, \dots, x^N_{I_N}\}$, $|B_N| =I_N$, and a sequence
$\epsilon_N \downarrow 0$ for which
\begin{equation*}
\lim_{N\to\infty} \mc Q \Big[ \nu_N(B^c_N) + I_N \nu_N(A^c_N) \ge 
\epsilon_N \, \Big] \;=\; 0\;.
\end{equation*}
Let $\Sigma_N^5 = \{ \nu_N(B^c_N) + I_N \nu_N(A^c_N) \le \epsilon_N
\}$.

To show that the pair $(c_{\bb V} \tau^N, \Psi_N(X^N_{t \beta_N}))$
converges weakly to $(\mb w, K_t)$, it is enough to show that any
subsequence $\{N_j : j\ge 1\}$ possesses a sub-subsequence $\mf n$
such that $(c_{\mf n} \tau^{\mf n},\Psi_{\mf n}(X^{\mf n}_{t
  \beta_{\mf n}}) )$ converges to $(\mb w, K_t)$. Fix, therefore, a
subsequence $N_j$. By \eqref{08}, the ordered sequence
$(c_{N_j}\tau^{N_j}_1, \dots, c_{N_j}\tau^{N_j}_{\bb V})$ converges
almost surely in $L^1(\bb N)$ to $\mb w =(w_1, w_2, \dots)$. This
proves the weak convergence of the first coordinate.  Let
$\Sigma_{N_j} = \cap_{1\le k\le 5} \Sigma^k_{N_j}$. There exists a
sub-subsequence, denoted by $\mf n$, for which
\begin{equation*}
\mc Q \Big[ \bigcup_{\mf n_0\ge 1} \bigcap _{\mf n \ge \mf n_0}
\Sigma_{\mf n} \Big] 
\;=\; 1\;.
\end{equation*}

We affirm that all assumptions of Theorem \ref{s13} hold on the event
$\cup_{\mf n_0\ge 1} \cap _{\mf n \ge \mf n_0} \Sigma_{\mf n}$
intersected with the ones in \eqref{e:wDE}, \eqref{lln} and
\eqref{e:ctau}. Indeed, condition ({\bf A1}) follows from the
definition of the set $\Sigma^1_{\mf n}$.  Similarly to the proof of
Theorem~\ref{t:trans}, condition ({\bf A3}) follows from the
definitions of the sequence $M_{\mf n}$ and the set $\Sigma^4_{\mf
  n}$.  Conditions \eqref{e:c2}, \eqref{14} follow from the definition
of the set $\Sigma^5_{\mf n}$.

We turn to condition ({\bf A2}). Recall that $\beta_{\mf n} = c_{\mf
  n}^{-1}$.  Fix a sequence $J_{\mf n} \uparrow\infty$ such that
$J_{\mf n} \le M_{\mf n}$, and let $B_{\mf n} = \{x^{\mf n}_1, \dots,
x^{\mf n}_{J_{\mf n}}\}$.  Since we replaced the weights $W^{\mf
  n}_{x^{\mf n}_j}$ by $\tau^{\mf n}_j$, the first expectation
appearing in this hypothesis can be rewritten as
\begin{equation}
\label{e:EnuB}
\frac{\sum_{1\le j\le J_{\mf n}} [c_{{\mf n}} \tau^{\mf n}_j]^2
\frac{\deg(x^{\mf n}_j)}{v_\ell(x^{\mf n}_j)}}
{\sum_{1\le j\le J_{\mf n}} c_{{\mf n}} \tau^{\mf n}_j \deg(x^{\mf n}_j)}\;\cdot
\end{equation}
By definition of the set $\Sigma^3_{\mf n}$ we may replace $\deg(x^{\mf
  n}_j)$ by $D_j$. Since $\tau^{\mf n}_j$ is decreasing, by definition
of the set $\Sigma^2_N$ the numerator is bounded by
\begin{equation*}
c_{{\mf n}} \tau^{\mf n}_1 \sum_{j=1}^{J_{\mf n}} c_{{\mf n}}
\tau^{\mf n}_j \, \frac{D_j}{E_j} \;+\; 
\frac{c_{{\mf n}} \tau^{\mf n}_1}{M^2_{\mf n}} 
\sum_{j=1}^{J_{\mf n}} c_{{\mf n}} \tau^{\mf n}_j \, D_j\;.
\end{equation*}
The second term divided by the denominator in \eqref{e:EnuB} is less
than or equal to $c_{{\mf n}} \tau^{\mf n}_1M^{-2}_{\mf n}$ which goes
to 0 as $\mf n \to \infty$ in view of \eqref{e:ctau}. Also, by
\eqref{e:ctau}, the first term is bounded by
\begin{equation*}
c_{{\mf n}} \tau^{\mf n}_1 \sum_{j=1}^{J_{\mf n}} w_j \,
\frac{D_j}{E_j} \;+\; c_{{\mf n}} \tau^{\mf n}_1 \frac 1{M_{\mf n}}
\max_{1\le j\le J_{\mf n}} \frac{D_j}{E_j}\;.
\end{equation*}
Since the denominator in \eqref{e:EnuB} is bounded below by $c_{{\mf
    n}} \tau^{\mf n}_1 \, D_1 \ge c_{{\mf n}} \tau^{\mf n}_1$, the
first condition in ({\bf A2}) follows from \eqref{e:wDE}, \eqref{lln}.

The second condition of assumption ({\bf A2}) can be written as
\begin{equation*}
\frac{1}{J_{\mf n}} \frac{\sum_{1\le j\le J_{\mf n}} 
v_\ell(x^{\mf n}_j) \deg(x^{\mf n}_j)}{\sum_{1\le j\le J_{\mf n}}
c_{\mf n} \tau^{\mf n}_j \deg(x^{\mf n}_j)} \;\cdot
\end{equation*}
By definition of the set $\Sigma^3_{\mf n}$ we may replace
$\deg(x^{\mf n}_j)$ by $D_j$.  The sum in the denominator is bounded
below by $c_{\mf n} \tau^{\mf n}_1 D_1 \ge c_{\mf n} \tau^{\mf n}_1$,
which is uniformly bounded.  Since the escape probability is bounded
by one and since by \eqref{e:c4} $E_j$ is bounded by one, the
numerator is less than or equal to $\sum_{1\le j\le J_{\mf n}}
(D_j/E_j)$, whose average by \eqref{lln} is bounded.

It remains to establish \eqref{04} with $\gamma_N = 1$, $Z_j =
w_j/E_j$ and $u_j = E_j D_j$.  The convergence of the first term
follows from \eqref{e:ctau}, the definition of $\Sigma^2_{\mf n}$ and
$\Sigma^3_{\mf n}$ and the fact that the variables $E_j$ are bounded
by one. The second part of \eqref{04} amounts to estimate
\begin{equation*}
\sum_{j = m}^{M_{\mf n}} c_{\mf n} \tau_j^{\mf n} \deg (x^{\mf n}_j)
\;=\; \sum_{j = m}^{M_{\mf n}} c_{\mf n} \tau_j^{\mf n} D_j
\;\le\; \sum_{j = m}^{M_{\mf n}} w_j\, (D_j/E_j) \;+\; 
\frac 1{M^2_{\mf n}}\max_{1\le j \le M_{\mf n}}  (D_j/E_j) \;,
\end{equation*}
where the identity follows from the definition of $\Sigma^3_{\mf n}$
and the inequality from \eqref{e:ctau} and the boundedness of
$E_j$. The first term on the right hand side vanishes in view of
\eqref{e:wDE} and the second one by \eqref{lln}. This concludes the
proof of the Theorem.

\section{Supercritical Erd\"os-R\'enyi random graphs}
\label{sec10}

We show in this section that super-critical Erd\"os-R\'eny random
graphs satisfy the assumptions of Theorem \ref{t:random}. Let $\ms
V_N$ be the set of vertices $\ms V_N = \{1, \dots, N\}$. For $\lambda
> 1$ fixed, let $\{\xi_{x,y} : x,y \in \ms V_N\}$ be
i.i.d.~Bernoulli($\lambda/N$) random variables constructed in a
probability space $(\Omega, \mathcal{A}, \mathbb{P})$. The
Erd\"os-R\'enyi random graph is defined as $\ms G_N = (\ms V_N, \ms
E_N)$, where $\ms E_N$ is the random set of edges given by $\{\{x,y\};
\xi_{x,y} = 1\}$.  Throughout this section, $c_j$, $C_j$, $j\ge 0$,
represent positive constants depending on $\lambda$ and sometimes on
further parameters, the first ones being tipically small and the last
ones large. Next result can be found in \cite[Theorem~2.3.2]{Dur10}.

\begin{theorem}
\label{e:unique}
There is a constant $c_{0}$ such that with $\mathbb{P}$-probability
converging to one as $N$ tends to infinity, there is a unique
component $\mathcal{C}_{\text{max}}$ in $(\ms V_N, \ms E_N)$ with
$|\mathcal{C}_{\text{max}}| > c_{0} \log N$. Moreover, there exists
$0<\mf v_\lambda<1$ such that
\begin{equation*}
\lim_{N\to\infty} \mathbb{P} \Big[ \, \Big| 
\frac{|\mathcal{C}_{\text{max}}|}N 
- \mf v_\lambda \,\Big| \,>\, \epsilon \Big] \;=\; 0\;.
\end{equation*}
for all $\epsilon >0$.
\end{theorem}

We will be interested in analyzing the trap model in
$\mathcal{C}_{\text{max}}$, providing another interesting example for
which our machinery can be applied. For the sake of simplicity we
shall assume that the common distribution of the traps $\{W_j^N : j\ge
1\}$ is $\alpha$-stable. More precisely, recall the definition of the
variables $\hat \tau^N_i$, $1\le i\le \bb V$, introduced in \eqref{52}
with $\bb V = N$ and $F(t)=t$. We assume in this section that
$W^N_i=\hat \tau^N_i$, $1\le i\le N$.

Let $V_N = \mathcal{C}_{\text{max}}$ be the random set of vertices and
let $E_N = \{ \{x,y\}\subset V_N : \{x,y\}\in \ms E_N\}$ be the random
set of edges of the random graph $G_N$.  In contrast with the previous
examples presented in Section \ref{sec7}, the number of vertices of
the random graph $G_N$ is also random. The weights are distributed as
follows. Given $V_N$, re-enumerate the weights $W^N_{j}$, $1\le j \le
|V_N|$, in decreasing order and denote by $\hat W^N_{j}$ the new
sequence, so that $\hat W^N_{j} \ge \hat W^N_{j+1}$, $1\le j < |V_N|$,
$\hat W^N_{\sigma(j)}=W^N_j$ for some permutation $\sigma$ of $V_N$.
Randomly enumerate the vertices of $V_N$, obtaining a vector $(x^N_1,
\dots, x^N_{|V_N|})$, and set $W^N_{x^N_j} = \hat W^N_j$. Given this
random vertex-weighted graph, we examine the continuous-time random
walk $X^N_t$ on $G_N$ with generator given by \eqref{f05}.

Note that to define the random weights $W^N_j = \hat\tau^N_j$ we
divided the interval $[0,1]$ in $N$ sub-intervals instead of dividing
it in $|V_N|$ intervals. In particular, in contrast with the examples
of Section \ref{sec7}, $N^{-1/\alpha} W^N_{x^N_1}$ does not converge
to a Fr\'echet distribution, but so does $\mf v_\lambda^{-1/\alpha}
N^{-1/\alpha} W^N_{x^N_1}$, where $\mf v_\lambda$ is given by Theorem
\ref{e:unique}.

%conv +infty em probabilidade somente.

%construcao dos W

In the rest of this section, we prove that the assumptions of Theorem
\ref{t:random} are fulfilled. By Theorem \ref{e:unique}, the number of
vertices converges in probability to $+\infty$. To establish
\eqref{B0}, fix a sequence $J_N\uparrow\infty$ and denote by $\ms
W^N_1, \dots, \ms W^N_N$ the sequence $W^N_1, \dots, W^N_{N}$
enumerated in decreasing order. Note that $\ms W^N_j \ge \hat W^N_j$,
$1\le j\le |V_N|$.  By \eqref{08} and \eqref{aa1}, for every $\epsilon
>0$, 
\begin{equation*}
\lim_{N\to\infty} \bb P \Big[ \sum_{j\ge 1} 
| N^{-1/\alpha} \ms W^N_j - w_j| \ge \epsilon \Big] \;=\; 0\;.
\end{equation*}
Since $\sum_{j\ge J_N} w_j$ vanishes almost surely as
$N\uparrow\infty$, if $\Sigma^0_N$ stands for the event $\sum_{j\ge
  J_N} N^{-1/\alpha} \ms W^N_j \le 1$,
\begin{equation*}
\lim_{N\to\infty} \bb P \big[ \Sigma^0_N \big] \;=\; 1\;.
\end{equation*}
Denote by $\Sigma^1_N$ the event $\{|V_N - \mf v_\lambda N |\le
\epsilon N\}$ for some $0<\epsilon < \min\{\mf v_\lambda , 1 - \mf
v_\lambda\}$. By Theorem \ref{e:unique}, $\bb P[\Sigma^1_N ] \to 1$.
In conclusion, to prove \eqref{B0} we need to show that
\begin{equation*}
\lim_{N\to\infty} \bb E \big[ \nu_N \big(\{x_1, \dots, x_{\min\{J_N,
  |V_N|\}}\}^c \big) \, \mb 1\{\Sigma^0_N \cap \Sigma^1_N\} \big] 
\;=\; 0\;.
\end{equation*}
By definition of $\nu_N$, and since all vertices in $V_N$ have degree
at least equal to one,
\begin{equation*}
\nu_N \big(\{x_1, \dots, x_{\min\{J_N, |V_N|\}}\}^c\big)
\;\le\; \frac{\sum_{j= J_N+1}^{|V_N|} W_{x_j}^N \deg
  (x_j)}{W_{x_1}^N}\; \cdot
\end{equation*}
Since $\ms W^N_j \ge \hat W^N_j$, $1\le j\le |V_N|$, 
\begin{equation*}
\sum_{j= J_N+1}^{|V_N|} W_{x_j}^N \deg (x_j) \;\le\;
\sum_{j= J_N+1}^{|V_N|} \ms W_{j}^N \deg (x_j) \;\le\;
\sum_{j= J_N+1}^{N} \ms W_{j}^N \deg (x_j)\;,
\end{equation*}
if $x_{|V_N|+1}, \dots, x_{N}$ represents a random enumeration of the
vertices of $\ms V_N$ which do not belong to the largest component.
On the set $\Sigma^1_N$, $W_{x_1}^N \ge \max_{1\le k\le c_\lambda N}
W^N_k$, where $c_\lambda = \mf v_\lambda - \epsilon$. This latter
variable as well as the variables $\ms W_{j}^N$ depend only on the
Poisson point process defined at the beginning of Section 8. Hence if
we denote by $\mf W$ the $\sigma$-algebra generated by this process
and let $\Sigma^{0,1}_N = \Sigma^0_N\cap \Sigma^1_N$, we obtain that
\begin{equation*}
\begin{split}
& \bb E \Big[ \frac{\sum_{j= J_N+1}^{N} \ms W_{j}^N \deg
  (x_j)}{W_{x_1}^N} \, \mb 1\{\Sigma^{0,1}_N \} \Big] \;\le\; 
\bb E \Big[ \frac{\sum_{j= J_N+1}^{N} \ms W_{j}^N \deg
  (x_j)}{\max_{1\le k\le c_\lambda N } W_{k}^N}
\, \mb 1\{\Sigma^{0,1}_N \}\Big] \\
&\qquad
\le\; \bb E \Big[ \frac {\mb 1\{\Sigma^0_N \}}
{\max_{1\le k\le c_\lambda N } W_{k}^N}
\sum_{j= J_N+1}^{N} \ms W_{j}^N \,
\bb E \big[\deg  (x_j) \, \mb 1\{\Sigma^1_N \} \,
\big|\, \mf W \big]\, \Big] \; . 
\end{split}
\end{equation*}

We first estimate the conditional expectation and then the remaining
expression. Since the law of the graph $\ms G_N$ is independent of the
$\sigma$-algebra $\mf W$, the previous conditional expectation is
equal to $\bb E [\deg (x_j) \, \mb 1\{\Sigma^1_N \}]$. By construction
if $j\le |V_N|$, $\deg (x_j)$ has the same distribution as $\deg(x_k)$
for $1\le k\le |V_N|$, with a similar fact if $j> |V_N|$. Therefore,
for a fixed $j$, the previous expectation is bounded by
\begin{equation*}
\sum_{\ell\le j-1} \bb E \Big[\mb 1\{|V_N|=\ell \} \, 
\frac 1{N-\ell} \sum_{y\not \in V_N} \deg  (y) \, \Big] \;+\;
\sum_{\ell\ge j} \bb E \Big[\mb 1\{|V_N|=\ell \} \, 
\frac 1{\ell} \sum_{y\in V_N} \deg  (y) \, \Big]\;,
\end{equation*}
where the sum is carried over all $\ell$ such that $|\ell - \mf
v_\lambda N|\le \epsilon N$. Estimating the denominators by the worst
case, we get that the sum is less than or equal to 
\begin{equation*}
\frac 1{\min\{\mf v_\lambda - \epsilon, 1 - \epsilon - \mf v_\lambda\}}
\, \bb E \Big[ \frac 1{N} \sum_{y=1}^N \deg  (y) \, \Big] \;.
\end{equation*}
This expectation is equal to $\lambda$.

It remains to estimate the expectation involving the weights. On the
set $\Sigma^0_N$, $\sum_{J_N+1\le j\le N} \ms W_{j}^N \le
N^{1/\alpha}$. On the other hand, using the notation introduced in
\eqref{09}, $\max_{1\le k\le c_\lambda N } N^{-1/\alpha} W_{k}^N \ge
w(\lambda)$, where $w(\lambda) = \max_i \hat w_i$, and where the
maximum is carried over all indices $i$ such that $z_i\le
c_\lambda$. Hence,
\begin{equation*}
\bb E \Big[ \frac {\mb 1\{\Sigma^0_N \}}
{\max_{1\le k\le c_\lambda N } W_{k}^N}
\sum_{j= J_N+1}^{N} \ms W_{j}^N \, \Big] \;\le\;
\bb E \Big[ \frac {1} {w(\lambda)} \Big]\;.
\end{equation*}
Since $w' = w(\lambda)/c_\lambda^{1/\alpha}$ has a Fr\'echet
distribution, $P(w' \leq t) = \exp\{-1/t^\alpha\}$, this expectation
is finite, which proves condition \eqref{B0}.

The results of this section should still hold if we require the
variables $W^N_j$ to belong to the domain of attraction of an
$\alpha$-stable law and to satisfy the bound  
\[
\limsup_{N \to \infty} \bb E\Big[\big(c_N \sup_{1 \leq i \leq N} 
W^N_i\big)^{-1}\Big] <+\infty\;,
\]
where $c_N$ has been introduced in \eqref{aa1}.

To understand the asymptotic law of the escape probabilities, we need
to introduce a related branching process.  Let $\mathcal{T}$ be the
random tree obtained by the Galton-Watson process with offspring
distribution Poisson($\lambda$) and denote its law by
$\mathcal{P}$. Since $\lambda$ is assumed to be greater than one, the
event that $\mathcal{T}$ is infinite has positive
$\mathcal{P}$-probability, \cite[Theorem~2.1.4]{Dur10}. We denote by
$\varnothing$ the root of $\mathcal{T}$.

We first show that the neighborhood of a random point in the
Erd\"os-R\'enyi graph looks like the neighborhood of $\varnothing$ in
$\mathcal{T}$. This is made precise as follows. We write $(x,G)$ for a
graph with a marked vertex $x$. We say that $(x,G)$ is isometric to
$(x',G')$ if there exists an isometry between $G$ and $G'$, sending
$x$ to $x'$. As an abuse of notation, we consider $A \subseteq \ms
V_N$ both as a set of vertices and as the corresponding induced
subgraph of $\ms G_N$.

\begin{proposition}
\label{l:isom}
Let $0<\gamma < (3\log \lambda)^{-1}$.  There exist constants $C_1$
and $N_0=N_0(\lambda, \gamma)$ such that given a random point $z \in
\ms V_N$, we can find a coupling $Q_N$ between the random graph $\ms
G_N$ under $\mathbb{P}$ and the Galton-Watson tree $\mathcal{T}$ under
$\mathcal{P}$ such that for all $N\ge N_0$,
\begin{equation*}
Q_N \Big[ \big(z,B(z, \gamma \log N) \big) 
\text{ is isometric to } \big(\varnothing, B(\varnothing, 
\gamma \log N) \big) \Big] \geq 1 - C_1 N^{3\gamma\log \lambda-1}\;.
\end{equation*}
\end{proposition}

\begin{proof}
We follow an argument similar to the one in
\cite[Section~2.2]{Dur10}. Assume, without loss of generality, that $z
= 1$ and define an exploration of the cluster $\mathcal{C}_1$
containing $1$ in the following way. Let $S_0 = \{2,3,\dots,N\}$, $I_0
= \{1\}$ and $R_0 = \varnothing$. These sets represent respectively
the `susceptible', the `infected' and the `removed' sites. Define 
a discrete time evolution by
\begin{equation*}
\begin{split}
& R_{t+1} = R_t \cup I_t,\\
& I_{t+1} = \{y \in S_t; \xi_{x,y} = 1 \text{ for some $x \in I_t$}\},\\
& S_{t+1} = S_t \setminus I_{t+1}.
\end{split}
\end{equation*}
Note that the cluster $\mathcal{C}_1$ is given by $\cup_{t=1}^\infty
I_t$ and that $B(1,r) = \cup_{t=1}^r I_t$.

In order to couple the above exploration process with a Galton-Watson
branching process, we introduce a new set of independent
Bernoulli($\lambda/N$) random variables $\zeta_{x,y}^t$, $t \geq 1$,
$x\ge 1$, $1 \leq y \leq N$. Let $Z_0 = 1$ and
\begin{equation}
\label{e:Zt}
Z_{t+1} = \sum_{\substack{x \in I_t\\ y \in S_t}} \xi_{x,y} + 
\sum_{\substack{x \in I_t\\ y \in \ms V_N \setminus S_t}} \zeta_{x,y}^t + 
\sum_{x = N + 1}^{N+Z_t-|I_t|} \sum_{y=1}^N \zeta_{x,y}^t\;.
\end{equation}
The first term in the above sum can be written as $|I_{t+1}| +
C_{t+1}$, where $C_{t+1}$ represents the number of `collisions'
occurring in the exploration process, that is, individuals in
$I_{t+1}$ connected to more than one individual in $I_t$.  The second
term stands for the `immigrants' introduced to compensate the fact
that $|S_t| < N$, and the third term for children of individuals that
are not in $I_t$.

It is easy to check that the process $\{Z_t : t \geq 0\}$ is a
branching process with offspring distribution Binomial($N,
\lambda/N$). Let $\mathcal{T}'$ be the random tree associated with
$Z_t$. More precisely, if $x$ is the $i$-th individual in the $t$-th
generation of $\mathcal{T}'$, the number of offsprings of $x$ will be
given by
\begin{equation*}
\begin{cases}
\sum_{y \in S_t} \xi_{x,y} + \sum_{y \in \ms V_N \setminus S_t}
\zeta_{x,y}^t & \text{if $i \leq |I_t|$},\\
\sum_{y = 1}^N \zeta_{x,y}^t & \text{otherwise}.
\end{cases}
\end{equation*}
It is immediate to check that $Z_t$ is the size of the $t$-th
generation of $\mathcal{T}'$ and that $Z_t \geq |I_t|$.

On the event $Z_s = |I_s|$, $1\le s\le t$, there were no collisions
and no immigrants. Therefore, in this event the subgraph $(1,B(1,t))$
of $\ms G_N$ is isometric to the subgraph
$(\varnothing,B(\varnothing,t))$ of $\mathcal{T}'$.  Hence, by
\cite[Theorem~2.2.2]{Dur10} with $t = \gamma \log N$, there exist a
constant $C_1<\infty$ and a coupling $Q'$ between $\ms G_N$ and
$\mathcal{T}'$ such that with probability at least $1 - C_1
N^{2\gamma \log \lambda -1}$, $(1,B(1,t))$ is isometric to
$(\varnothing,B(\varnothing,t))$.

\smallskip\noindent{\bf Claim A:} Let $0<\gamma < (3\log
\lambda)^{-1}$. There exist $n_0$ and a coupling $Q''$ between the
tree $\mathcal{T}'$ with Binomial($N, \lambda/N$) offsprings and the
tree $\mathcal{T}$ with Poisson($\lambda$) offsprings, such that, with
probability at least $1 - C_1 N^{3\gamma \log \lambda -1}$,
$(\varnothing,B(\varnothing, \gamma \log N))$ (in $\mathcal{T}'$) is
isometric to $(\varnothing, B(\varnothing, \gamma \log N))$ (in
$\mathcal{T}$) for $N\ge n_0$.

It is well known that a Poisson($\lambda$) random variable $Y$ can be
coupled with a Binomial($N$, $\lambda / N$) random variable $Y'$, in a
way that
\begin{equation}
\label{e:YYp}
P[Y=Y'] \;\ge\;  1 - 2\lambda^2 N^{-1}\;, 
\end{equation}
see for instance \cite[Chapter~2.6]{du} or \cite[Theorem~1]{LC60} for
a bound on the total variation distance and \cite[Chapter 4]{LPW09}
for a connection between total variation distance and coupling.  On
the other hand, by \cite[Theorem~4]{Ath94}, there exist $\theta =
\theta(\lambda) > 0$ and $C_3$ such that for and any $t$, $A \geq
0$,
\begin{equation*}
\mathcal{P} \big[ Z_t \geq A \lambda^{t} \big] \;=\;
\mathcal{P}\big[ e^{ \theta \, (Z_t / \lambda^t) } 
\geq e^{\theta A} \big] \;\leq\; e^{- \theta A}\,
\mathcal{E} \big[ e^{ \theta \,(Z_t / \lambda^t) } \big]
\;\leq\; C_3 \, e^{-\theta A } \; .
\end{equation*}
This bound permits to estimate the volume of the subgraph
$B(\varnothing, \gamma \log N)$ of $\mc T$. Fix $\gamma \in
(0,1)$. Since $|B(\varnothing, \gamma \log N)| = \sum_{0\le t\le
  \gamma \log N} Z_t$, we have that
\begin{equation*}
\begin{split}
& \mathcal{P} \big[|B(\varnothing, \gamma \log N)| 
\geq N^{3 \gamma \log \lambda } \big]  \\
&\quad \leq\; 
\sum_{t = 0}^{\gamma \log N} \mathcal{P}[Z_t \geq 
N^{2 \gamma \log \lambda }]
\;\leq\; \sum_{t = 0}^{\gamma \log N} \mathcal{P}[Z_t \geq 
N^{\gamma \log \lambda} \lambda^t]
\end{split}
\end{equation*}
for all $N \ge N_0(\lambda,\gamma)$. Therefore, applying the previous
estimate, we conclude that for every $0<\gamma<1$, there exist
$C_3 <\infty$ and $N_0(\lambda,\gamma) < \infty$ such that 
\begin{equation}
\label{e:volume}
\mathcal{P} \big[|B(\varnothing, \gamma \log N)| 
\geq N^{3 \gamma \log \lambda } \big]
\; \leq\;  C_3 \exp\{- \theta \, N^{\gamma \log \lambda}\}.
\end{equation}
for all $N\ge N_0$.

Claim A follows from \eqref{e:YYp} and \eqref{e:volume}, which
concludes the proof of Proposition~\ref{l:isom}.
\end{proof}

In the proof of the previous lemma we also obtained a bound on the
size of a ball $B(z, \gamma \log N)$ around a typical point $z$. 

\begin{corollary}
\label{c:volume}
For any $0 < \gamma < (3\log \lambda)^{-1}$, there exist a
finite constant $C_2$ and an integer $N_0$, depending
only on $\lambda$ and $\gamma$, such that for any random point $z \in
\{1,\dots,N\}$,
\begin{equation*}
\mathbb{P} \big[ |B(z,\gamma \log N)| \geq  
N^{3 \gamma \log \lambda } \big] \;\leq \; C_2 \, N^{3 \gamma \log
  \lambda - 1} 
\end{equation*}
for all $N\ge N_0$.
\end{corollary}

As required in \eqref{e:c4}, we extend the local isometry obtained in
Proposition~\ref{l:isom} to various balls in the random graph $\ms G_N$.

\begin{corollary}
\label{c:manyballs}
Fix positive numbers $b$ and $\gamma$ such that $0 < 2b + 6 \gamma
\log \lambda<1$.  There exist constants $C_0$, $N_0$, depending only
on $\lambda$ and $\gamma$, and a coupling $Q'=Q'_N$ between the random
graph $\ms G_N$ and $N^{b}$ independent Galton-Watson trees
$\mathcal{T}_i$, $1\le i \le N^{b}$, such that for all $N\ge N_0$,
\begin{equation*}
Q' [\ms B^c ] \;\leq \; C_0 \, N^{2b + 6 \gamma \log \lambda -1} \;,
\end{equation*}
where $\ms B$ is the event ``The balls $(z_i,B(z_i, \gamma \log N))$,
$1\le i \le N^{b}$, are disjoint and isometric to $(\varnothing_i,
B(\varnothing_i, \gamma \log N))$'', and $z_1, \dots, z_{N^{b}}$ are
sites randomly chosen in $\{1,\dots, N\}$.
\end{corollary}

\begin{proof}
Choose randomly $N^{b}$ sites on $\{1, \dots, N\}$, denoted by $z_1,
\dots, z_{N^{b}}$.  By Proposition \ref{l:isom}, for $N$ large, there is a
coupling $Q'$ between independent Erd\"os-R\'enyi random graphs $\ms
G^i_N$, $1\le i \le N^{b}$, and independent Galton-Watson trees
$\mathcal{T}_i$ in a way that with probability at least $1 - C_1
N^{b} N^{3\gamma \log \lambda -1}$ each ball $\big(z_i,B(z_i, \gamma
\log N) \big)$ in $\ms G^i_N$ is isomorphic to $\big(\varnothing_i,
B(\varnothing_i, \gamma \log N) \big)$ in $\mathcal{T}_i$.

We construct an Erd\"os-R\'enyi-distributed graph $\ms G_N$ which is
partially determined by the above $\ms G^i_N$'s. We first explore the
ball $B(z_1, \gamma \log N)$ in $\ms G^1_N$. Every edge $\{x,y\}$
revealed during this exploration is open in $\ms G_N$ if and only if it
is open in $\ms G^1_N$. Then we proceed by exploring $B(z_2, \gamma
\log N)$ in $\ms G_N$ observing only that we do not reassign values to
edges in $\ms G_N$ that were already established in the previous
step. After proceeding with this exploration for $i = 1,\dots,
N^{b}$, we assign the remaining edges of $\ms G_N$ independently.

It is clear from the above exploration procedure that the graph $\ms
G_N$ is distributed as an Erd\"os-R\'enyi random graph. Moreover, on the
event $\ms A$ defined as ``the balls $B(z_i,\gamma \log N)$, $i =
1,\dots, N^{b}$, are pairwise disjoint in $\{1,\dots, N\}$'', we
have that $(z_i, B(z_i, \gamma \log N))$ in $\ms G_N$ is isomorphic to
the corresponding pair in $\ms G^i_N$. Consequently they will be
isomorphic to $(\varnothing_i, B(\varnothing_i, \gamma \log N))$ in
$\mathcal{T}_i$. Therefore, to conclude the proof of the corollary, it
remains to estimate $Q'[\ms A^c]$.

Since all the vertices are indistinguishable, $Q'[\ms A^c]$ is bounded
by
\begin{equation*}
N^{2b} Q' \big[ B(z_1, \gamma \log N) 
\cap B(z_2, \gamma \log N) \not =\varnothing \big] 
\;=\; N^{2b} Q' \big[ z_1 \in  B(z_2, 2 \gamma \log N) 
\big]\;.
\end{equation*}
Since $z_2$ is independent of $z_1$, this latter probability is
bounded by
\begin{equation*}
Q' \big[ | B(z_2, 2 \gamma \log N) |  \ge N^{6 \gamma \log \lambda }
\big] \;+\; \frac 1N\, N^{6 \gamma \log \lambda}\;.
\end{equation*}
By Corollary \ref{c:volume}, for $N$ large, the first term is bounded
above by $C_2 N^{6 \gamma \log \lambda -1}$ for some finite
constant $C_2$. Hence,
\begin{equation*}
Q'[\ms A^c] \;\le\; C_2 N^{2b + 6 \gamma \log \lambda -1} \;,
\end{equation*}
which proves the corollary.
\end{proof}

It is a well known fact that
\begin{equation}
\label{e:trans}
\text{conditioned on being infinite, $\mathcal{T}$ 
is $\mathcal{P}$-a.s. transient,}
\end{equation}
see Theorem 3.5 and Corollary~5.10 in \cite{LP11}. We denote by
$v_\varnothing$ the probability that a simple random walk starting at
$\varnothing$ never returns to this site, the so called escape
probability. As we will show, the distribution of $v_\varnothing$
under $\mathcal{P}$ is close to that of the probability that a random
walk on the giant component $\mathcal{C}_{\text{max}}$ of the random
graph $\ms G_N$ escapes from a certain neighborhood of a random
vertex.

Since the isometry obtained in Corollary~\ref{c:manyballs} is local,
we need a tool to show that looking at a neighborhood of $\varnothing
\in \mathcal{T}$ we can obtain precise estimates on the escape
probability $v_\varnothing$. The next result plays a central role in
this respect. Denote by $\Delta_l$, $l \ge 0$, the points of the
$l$-th generation of a tree: $\Delta_l = B(\varnothing,l) \setminus
B(\varnothing,l-1)$.

For a fixed tree $\mc T$, we denote by $\mb P_y$, $y\in \mc T$ the
probability induced by the discrete-time simple random walk on $\mc T$
starting from $y$.

\begin{proposition}
\label{l:nodive}
There exist constants $c_1$, $c_2$, depending only on $\lambda$, such
that, for every $l \geq 1$,
\begin{equation*}
\mathcal{P} \Big[
\sup_{y\in \Delta_l} \mb P_y [H_\varnothing < \infty ] 
\geq \exp \{-c_1 l\} \Big] \;\leq\; \exp \{-c_2 l\}\;.
\end{equation*}
\end{proposition}

\begin{proof}
Throughout the proof of this lemma, given a rooted tree $\mathcal{T}$
and a vertex $y \in \mathcal{T}$, we denote by $\mathcal{T}_y$ the
subtree formed by the root $y$ together with the descendants of $y$ in
$\mathcal{T}$.

The idea is to show that in the path between $y$ and $\varnothing$
there are many tunnels from which the random walk can escape to
infinity. In order to properly define these tunnels, we need to
introduce some extra notation. For an arbitrary tree $\mathcal{T}$
rooted at $\varnothing$, we define the tree
$\mathcal{T}^{\text{tail}}$, obtained by adding a vertex
$\varnothing'$ which is connected to $\varnothing$ by an edge. This
extra element should be regarded as the ancestor of $\varnothing$. In
the proof, we use the notation $\mb P_x^{\mathcal{T}}$ to specify on
which tree the random walk is defined.

For a given $\delta > 0$ and a tree $\mathcal{T}$ with root
$\varnothing$, we say that $\mathcal{T}$ satisfies the property
$\mathcal{Q}^\delta$ if
\begin{equation*}
\mb P^{\mathcal{T}^{\text{tail}}}_\varnothing 
[ H_{\varnothing'} = \infty] \geq \delta\; .
\end{equation*}
In other words, the property $\mathcal{Q}^\delta$ is saying that a
random walk on $\mathcal{T}^{\text{tail}}$ has probability at least
$\delta$ of never hitting the ancestor $\varnothing'$ of the root
$\varnothing$.

It is clear from \eqref{e:trans} that for every $\epsilon > 0$, there
exists a $\delta = \delta(\epsilon, \lambda) > 0$ such that
\begin{equation}
\label{e:deltaescape}
\mathcal{P} [ \text{$\mathcal{T}$ does not satisfy 
$\mathcal{Q}^\delta$} ] \leq q + \epsilon,
\end{equation}
where $q$ is the extinction probability: $q=\mathcal{P}[\mathcal{T}
\text{ is finite}]$.

If $y$ is in the $l$'th generation of $\mathcal{T}$, we write
$\varnothing = y_0$, $y_1$, $\dots$, $y_l = y$ to denote the unique
simple path connecting $\varnothing$ to $y$. Moreover, we denote by
$\Gamma(y)$ the number of elements $y_k$, $0\le k <l$, having at least
one descendant $y_k' \not = y_{k+1}$ such that $\mathcal{T}_{y_k'}$
satisfies $\mathcal{Q}^\delta$.

We can now use \eqref{e:deltaescape} together with in \cite[
Lemma~1]{GK84} to conclude that there exist constants $c_3$ and $c_4$
such that
\begin{equation*}
\mathcal{P} [ \exists \, y \in \Delta_l
\text{ such that $\Gamma(y) < c_3 l$} ] \leq \exp \{-c_4 l\}\;.
\end{equation*}
To conclude the proof of the lemma it remains to show that there
exists $c_1>0$ for which the event ``$\exists \, y\in \Delta_l$ such
that $\mb P_y [H_\varnothing < \infty ] \geq \exp \{-c_1 l\}$'' is
contained in the event ``$\exists \, y \in \Delta_l$ such that
$\Gamma(y) < c_3 l$''.

Assume that all points $z$ in generation $l$ of $\mathcal{T}$ are such
that $\Gamma(z) \ge c_3 l$ and fix a point $y\in \Delta_l$. Recall the
definition of $y_0, \dots, y_l$ given above and consider a subsequence
$k_j$, $1\le j \le c_3 l$, for which $y_{k_j}$ has a descendant
$y'_{k_j} \not = y_{k_j + 1}$ such that $\mathcal{T}_{y'_{k_j}}$
satisfies $\mathcal{Q}^\delta$. These points are the entrance to the
tunnels $\mathcal{T}_{y_{k_j}'}$ that we have referred to in the
beginning of the proof.

Let $\mathcal{T}_-$ be the subtree of $\mathcal{T}$ with all the
descendants of $y_{k_j}$ removed, $1\le j \le c_3 l$, with the
exception of $y_{k_j + 1}$ and $y_{k_j}'$. An argument based on flows
or capacities shows that $\mb P_y^{\mathcal{T}}[H_\varnothing <
\infty] \le \mb P_y^{\mathcal{T}_-}[H_\varnothing < \infty] \le \mb
P_{y_{k_m}}^{\mathcal{T}_-} [ H_\varnothing < \infty]$ where $m=c_3
l$.  By the strong Markov property,
\begin{equation*}
\mb P_{y_{k_m}}^{\mathcal{T}_-} [ H_\varnothing < \infty]
\;\leq \; \mb P_{y_{k_m}}^{\mathcal{T}_-} 
[H_{y_{k_{m-1}}} < \infty] \,
\mb P_{y_{k_{m-1}}}^{\mathcal{T}_-} [H_\varnothing < \infty]\;.
\end{equation*}
Since $\mathcal{T}_{y'_{k_j}}$ satisfies $\mathcal{Q}^\delta$ and
since we removed all descendants of $y_{k_j}$ with the exception of
$y'_{k_j}$ and $y_{k_j +1}$, $\mb P_{y_{k_m}}^{\mathcal{T}_-}
[H_{y_{k_{m-1}}} =\infty] \ge (1/3) \mb P_{y_{k_m}'}^{\mathcal{T}_-}
[H_{y_{k_m}} =\infty] \ge \delta/3$. Hence, the previous expression is
bounded by
\begin{equation*}
[1-(\delta/3)]\, \mb P_{y_{k_{m-1}}}^{\mathcal{T}_-} 
[H_\varnothing < \infty]\;.
\end{equation*}
Iterating this argument $m-1$ times we finally get that $\mb
P_y^{\mathcal{T}}[H_\varnothing < \infty]$ is bounded by
$[1-(\delta/3)]^{c_3 l-1}$, which concludes the proof of the lemma.
\end{proof}

Proposition~\ref{l:nodive} permits to approximate the inverse of the escape
probability $v_\varnothing$ by a local quantity.
Fix a infinite tree $\mc T$ and $m\ge 1$. Let $v^{(m)}_\varnothing$ be
the probability to escape from $B(\varnothing, m)$,
$v^{(m)}_\varnothing = \mb P_\varnothing [{H}^+_\varnothing >
H_{B(\varnothing, m)^c}]$.  Recall from \cite[Chapter 9]{LPW09} the
notion of flow and energy of a flow.  Since $|\mathcal{T}| = \infty$,
we can define a trivial unit flow from $\varnothing$ to
$B(\varnothing, m)^c$ which has energy equal to $m$. Hence, by
Proposition~9.5 and Theorem~9.10 of \cite{LPW09}, 
\begin{equation}
\label{31}
v_{\varnothing}^{(m)} \;\geq\;  (d_\varnothing m)^{-1}\;,
\end{equation}
where $d_\varnothing$ is the degree of the root.

\begin{corollary}
\label{c:bigB}
There exist positive constants $c_1$ and $c_2$, depending only on
$\lambda$, such that
\begin{equation*}
\mathcal{P}\Big[ |\Delta_l| 
< \exp\{c_1 l\} \, \big|\, \Delta_l \not = 
\emptyset\Big] \leq \exp\{-c_2 l\}
\end{equation*}
for every $l \geq 1$.
\end{corollary}

\begin{proof}
For a tree with at least $l$ generations, let $\mathcal{G}_l$ be the
graph obtained by identifying all points in $\Delta_l$, naming this
vertex $z_l$. All other sites are left untouched, and the number of
vertices of this new graph is $|B(\varnothing,l)| - |\Delta_l| +1$. We
know that
\begin{equation*}
|\Delta_l|/d_\varnothing = \pi(z_l)/\pi(\varnothing),
\end{equation*}
where $\pi$ stands for the stationary measure of a simple random walk
on $\mathcal{G}_l$.  The ratio in the right hand side of the above
equation can be estimated using the escape probabilities from these
two points. If $\mb P^{\mc G}_x$, $x\in \mc G_l$, stands for the
probability on the path space induced by a discrete-time random walk
on $\mc G_l$ starting from $x$, 
\begin{equation*}
\frac{\pi(z_l)}{\pi(\varnothing)} = 
\frac{\mb P^{\mc G}_\varnothing[H_{z_l} < H^+_\varnothing]}
{\mb P^{\mc G}_{z_l}[H_\varnothing < H^+_{z_l}]} \;\cdot
\end{equation*}
We may couple the random walk on $\mc G_l$ with a random walk on the
tree in such a way that $\mb P^{\mc G}_\varnothing[H_{z_l} <
H^+_\varnothing] = \mb P_\varnothing[H_{\Delta_l} < H^+_\varnothing] $
and that $\mb P^{\mc G}_{z_l}[H_\varnothing < H^+_{z_l}] \le
\max_{y\in \Delta_i} \mb P_{y}[H_\varnothing < H^+_{\Delta_i}]$.  By
\eqref{31}, $\mb P_\varnothing[H_{\Delta_l} < H^+_\varnothing] \ge
(d_\varnothing l)^{-1}$. Putting together all previous estimates, we
get that on the set $\Delta_l \not = \emptyset$, 
\begin{equation}
\label{51}
|\Delta_l|^{-1} \;\le\; l \, \max_{y\in \Delta_i} 
\mb P_{y}[H_\varnothing < H^+_{\Delta_i}]
\;\le\; l \, \max_{y\in \Delta_i} 
\mb P_{y}[H_\varnothing < \infty]\;.
\end{equation}

Since there is a positive probability that a super-critical tree
survives, the probability appearing in the statement of the lemma is
bounded by $C_0 \mathcal{P} [ \, |\Delta_l| < \exp\{c_1 l\}$,
$\Delta_l \not = \emptyset]$. By \eqref{51}, this probability is
bounded by $C_0 \mathcal{P} [ \, l \, \max_{y\in \Delta_i} 
\mb P_{y}[H_\varnothing < \infty] \ge \exp\{- c_1 l\}]$, which
is bounded by $\exp\{-c_2 l\}$ by Proposition \ref{l:nodive}
\end{proof}

\begin{corollary}
\label{l:approxesc1}
For any $0<\gamma <1$, there exist positive constants $c_0$ and
$N_0\ge 1$, depending only on $\gamma$ and $\lambda$, such that for
all $N\ge N_0$,
\begin{equation*}
\mathcal{P} \Big[ \big|\tfrac{1}{v_\varnothing} 
- \tfrac{1}{v'_\varnothing} \big| \geq {d_\varnothing} N^{-c_0} \Big| 
|\mathcal{T}| = \infty \Big] \;\leq\; N^{-c_0} \;,
\end{equation*}
where $d_\varnothing$ represents the degree of $\varnothing$ and
$v'_\varnothing = \mb P_\varnothing [{H}^+_\varnothing >
H_{B(\varnothing,\gamma \log N)^c}]$.
\end{corollary}

\begin{proof}
Fix $0<\gamma < 1$ and an infinite tree $\mc T$. To keep notation
simple, let $B = B(\varnothing,\gamma \log N)$ and let $\partial B$ be
the set of points in $B^c$ which have a neighbor in $B$. By the strong
Markov property, 
\begin{equation*}
\mb P_\varnothing [{H}^+_\varnothing > H_{B^c}]
\;\geq\; \mb P_\varnothing [{H}^+_\varnothing = \infty]
\;\geq\; \mb P_\varnothing [{H}^+_\varnothing > H_{B^c}] 
\inf_{x \in \partial B} \mb P_x [{H}^+_\varnothing = \infty]\;.
\end{equation*}
Inverting these terms, we obtain
\begin{equation*}
0\;\le\; \frac{1}{v_{\varnothing}'} - \frac{1}{v_\varnothing}
\;\leq \; \frac{1}{v_\varnothing'} 
\Big(\frac{1}{\inf_{x \in \partial B} \mb P_x[H_\varnothing = \infty]} 
-1\Big)\;.
\end{equation*}
By Proposition~\ref{l:nodive} with $l = \gamma \log N$, there exists
constants $c_1$, $c_2>0$, depending on $\lambda$, such that on a set
with probability at least $1- N^{-\gamma c_2}$ the previous infimum is
bounded below $1- N^{- \gamma c_1}$. Since $(1-x)^{-1} \leq 1 + 2x$
for $x \in (0,1/2)$, there exists $N_0 = N_0(\gamma, \lambda)$ such
that for $N\ge N_0$,
\begin{equation*}
\Big | \frac{1}{v_{\varnothing}'} - \frac{1}{v_\varnothing} \Big| 
\;\leq\; \frac 2{N^{\gamma c_1}}\, \frac 1{v'_{\varnothing}}\;\cdot 
\end{equation*}
Estimate \eqref{31} permits to conclude the proof of the corollary,
changing the values of the exponents if necessary.
\end{proof}

\begin{corollary}
\label{c:esc}
Let $\mathcal{T}$ be a Galton-Watson tree with Poisson{\rm
  ($\lambda$)} offsprings, $\lambda>1$. Then, there exist finite
constants $c_0$, $C_0$ and $s_0<\infty$, depending only on $\lambda$,
such that
\begin{equation*}
\mathcal{P} \big[ (v_\varnothing)^{-1} \geq s \,\big|\, 
|\mathcal{T}| = \infty \big] \;\leq\; C_0 \, \exp \{-c_0 \sqrt{s}\}
\end{equation*}
for all $s \geq s_0$.
\end{corollary}

\begin{proof}
Since $\mathcal{T}$ is super-critical, the probability appearing in the
statement of the lemma is bounded by $C_3 \, \mathcal{P} [
(v_\varnothing)^{-1} \geq s \,,\, |\mathcal{T}| = \infty ]$ for some
finite constant $C_3$ depending only on $\lambda$. Fix an integer $n
\geq 1$. By the strong Markov property, $v_\varnothing$ is bounded
below by $\mb P_\varnothing[H_{B^c} \leq H_\varnothing^+] \inf_{y\in
  B^c} \mb P_y[ H_\varnothing = \infty]$, where
$B=B(\varnothing,n)$. Therefore, $\mathcal{P} [ (v_\varnothing)^{-1}
\geq s \,,\, |\mathcal{T}| = \infty ]$ is less than or equal to
\begin{equation*}
\mathcal{P} \big[\mb P_\varnothing[H_{B^c} 
\leq H_\varnothing^+]^{-1} \geq s/2, |\mathcal{T}| = \infty \big] 
\;+ \; \mathcal{P} \big[ 
\inf_{y\in B^c} P_y[ H_\varnothing = \infty] \leq 1/2 \big] \;.
\end{equation*}
By \eqref{31}, $\mb P_\varnothing[H_{B^c} \leq H_\varnothing^+] \ge
(d_\varnothing n)^{-1}$. The previous expression is thus bounded by
\begin{equation*}
\mathcal{P}[d_\varnothing n \geq s/2] \;+\; 
\mathcal{P} \big[\sup_{y\in B^c} P_y[H_\varnothing < \infty] 
\geq 1/2 \big] \;.
\end{equation*}
Set $ n=\sqrt{s}$, recall that $d_\varnothing$ has a
Poisson($\lambda$) distribution.  Apply an exponential Tchebychev
inequality to estimate the first term. By Proposition~\ref{l:nodive}
with $l = \sqrt{s}$, the second term is bounded by $\exp\{-c_2
\sqrt{s}\}$ provided $s$ is large enough.
\end{proof}

The following corollary allows us to bound the quantity $\epsilon_N$
appearing in \eqref{e:c1} and \eqref{e:c3}.

\begin{corollary}
\label{l:approxesc2}
Fix an arbitrary vertex $y \in \{1, \dots, N \}$ and $0<\gamma <
(3\log \lambda)^{-1}$. Then, there exists positive constants $c_0$
and $N_0\ge 1$, depending only on $\gamma$ and $\lambda$, such that
for all $N\ge N_0$,
\begin{equation*}
\mathbb{P} \Big[ \sup_{z \in B(y, \gamma \log N)^c} 
\mb P_z[ H_y \leq \log^4 N ] > N^{-c_0} \Big] 
\;\leq\;  N^{-c_0}\;.
\end{equation*}
\end{corollary}

\begin{proof}
Denote by $\partial_i A$ the internal boundary of a set $A$:
$\partial_i A = \{x\in A : d(x,A^c)=1\}$. Fix $0<\gamma < (3\log
\lambda)^{-1}$. By Propositions~\ref{l:isom} and \ref{l:nodive}, there
exist positive constants $c_1$, $c_2$ and $C_1$, depending only on
$\lambda$, such that
\begin{equation*}
\begin{split}
& \mathbb{P} \Big[ \sup_{z \in \partial_i B} 
\mb P_z[ H_y \leq H_{B^c} ] > N^{- \gamma c_1} \Big] \\
& \qquad \leq\; C_1 \, N^{c} \;+\;
\mathcal{P} \Big[ \sup_{z \in \partial_i B} 
\mb P_z[ H_\varnothing \leq H_{B^c} ] > N^{- \gamma c_1} \Big] \;
\leq\; C_1 \, N^{c}  \;+\; N^{- \gamma c_2}\;,
\end{split}
\end{equation*}
where $c = 3\gamma\log \lambda - 1$ and $B= B(y, \gamma \log N)$.

Assume that $\sup_{z \in \partial_i B} \mb P_z[ H_y \leq H_{B^c} ] \le
N^{-\gamma c_1}$. We claim that in this case
\begin{equation}
\label{30}
\sup_{z \in B^c}  \mb P_z[ H_y \leq \log^4 N ] \;\le\;
N^{-\gamma c_1} \;+\; \sup_{z \in B^c}  \mb P_z[ H_y \leq \log^4 N - 1]\;.
\end{equation}
Iterating this estimate $\log^4 N$ times, we conclude the proof of the
corollary. It is enough, therefore, to prove \eqref{30}. By the strong
Markov property, $\mb P_z[ H_y \leq \log^4 N ]$ is bounded by
$\sup_{w\in \in \partial_i B} \mb P_w[ H_y \leq \log^4 N ]$. If $\{H_y
< H_{B^c}\}$, by the initial assumption we may bound the probability
by $N^{-\gamma c_1}$. This gives the first term on the right hand side of
\eqref{30}. On the other hand, on the set $\{H_y > H_{B^c}\}$, $H_y =
H_{B^c} + H_y \circ \theta_{H_{B^c}}$ and $H_y \circ \theta_{H_{B^c}}
\le \log^4N -1$. Hence, by the strong Markov property, for every
$w\in \partial_i B$,
\begin{equation*}
\mb P_w[ H_y \leq \log^4 N \,,\, H_{B^c}<H_y ] \;\le\; 
\mb P_w[ H_{B^c}<H_y ] \sup_{z \in B^c}  \mb P_z[ H_y \leq \log^4 N -
1]\;, 
\end{equation*}
which proves \eqref{30} and the corollary.
\end{proof}

We conclude this section deriving the scaling limit of the random walk
$X^N_t$ on the giant component of the super-critical Erd\"os-R\'enyi
random graph.

\begin{theorem}
\label{t:ER}
Consider the trap model $X^N_t$ on the largest component
$\mathcal{C}_{\textnormal{max}}$ of the Erd\"os-R\'enyi random graph
with traps $W^N_x$, $x\in \mc C_{\max}$, as described in the beginning
of this section.  Assume that $\Psi_N(X^N_0)$ converges in probability
to some $k\in\bb N$. Let $\beta_N= (\mf v_\lambda
N)^{1/\alpha}$. Then,
\begin{equation*}
(\beta^{-1}_N \mb W^N, \Psi_N(X^N_{t \beta_N})) 
\quad\text{converges weakly  to}\quad  (\mb w, K_t)\;,
\end{equation*}
where $\mb w$ is the sequence defined in \eqref{10} and where, for
each fixed $\mb w$, $K_t$ is a $K$-process starting from $k$ with
parameter $(\mb Z,\mb u)$, where $Z_k = w_k/E_k$ and $u_k = D_k
E_k$. Here, $(D_k, E_k)$, $k \geq 1$ is an i.i.d.~sequence,
distributed as $(d_\varnothing, v_\varnothing)$ under $\mathcal{P}
\big[ \cdot \big| |\mathcal{T}| = \infty \big]$. The above convergence
refers to the $L^1$-topology in the first coordinate and
$d_T$-topology in the second.
\end{theorem}

\begin{proof}
We need to establish conditions \eqref{B0}--\eqref{e:c4} for the above
sequence of graphs and to apply Theorem~\ref{t:random}.  Condition
\eqref{B0} has been proven in the beginning of this section.  The main
difficulty in checking the remaining hypotheses comes from the fact
that we are dealing with the giant component $\mathcal{C}_{\max}$,
which has a random size, instead of the whole set $\{1, \dots, N\}$ as
in the above lemmas and propositions.

In order to prove \eqref{BB0}, let $\ell_N = (\gamma/2) \log N$ with
$\gamma$ satisfying the conditions of
Corollary~\ref{c:manyballs}. Since the term inside the expectation in
\eqref{BB0} is bounded by one, the expectation in \eqref{BB0} is
less than or equal to
\begin{equation*}
\begin{split}
& \mathbb{E} \Big[\frac 1{|\mathcal{C}_{\max}|} 
\sum_{x \in \mathcal{C}_{\max}} \frac{|B(x,2\ell_N)|}
{|\mathcal{C}_{\max}|} \Big]\\  
& \qquad \leq \; \mathbb{P} \big[ \, |\mathcal{C}_{\max}| 
< (\mf v_\lambda/2) N \, \big]  \;+\;
\frac{4}{(\mf v_\lambda N)^2} 
\mathbb{E} \Big[ \sum_{x =1}^N |B(x,2\ell_N)| \Big]\;.
\end{split}
\end{equation*}
By Theorem~\ref{e:unique} the first term vanishes as
$N\uparrow\infty$, while by Proposition~\ref{l:isom} and
Corollary~\ref{c:volume} the second term vanishes. This proves that
condition \eqref{BB0} is fulfilled.

By \cite{Benj, FR, FR2}, with high probability the mixing time of a
random walk on $\mathcal{C}_{\textnormal{max}}$ is less than or equal
to $C_0 \log^2 N$ for some finite constant $C_0$. Choosing $L_N =
C_0^{-1} \log^2 N$, the hypothesis \eqref{BB1} becomes a direct
consequence of Corollary~\ref{l:approxesc2}. It is indeed enough to
condition the event appearing in the statement of Corollary
\ref{l:approxesc2} on the set that $y$ belongs to
$\mathcal{C}_{\textnormal{max}}$ and to recall from Theorem
\ref{e:unique} that the giant component has a positive density with
probability converging to $1$.

It remains to check \eqref{e:c4}. Let $Q'_N$ be the coupling between
the random graph $\ms G_N$ and $N^b$ independent Galton-Watson trees
$\mathcal{T}_i$ constructed in Corollary~\ref{c:manyballs}. We assume
that these trees are the first $N^b$ trees of an infinite
i.i.d. sequence of Galton-Watson trees.

Fix $K \geq 1$ and let $\mf x_1, \mf x_2, \dots, \mf x_K$ be the first
$K$ points $z_i$ which belongs to $\mc C_{\max}$: $\mf x_1 = z_j$ if
$z_j\in \mc C_{\max}$ and $z_i\not\in \mc C_{\max}$ for $1\le i<j$,
and so on. It is clear that $\mf x_1, \dots, \mf x_K$ is uniformly
distributed among all possible choices and that the probability of not
finding $K$ points in $\mc C_{\max}$ among $N^b$ points uniformly
distributed in $\ms V_N$ converges to $0$.

Let $\mf y_j$, $1\le j\le K$, be the first $K$ indices of trees
$\mathcal{T}_i$ which are infinite and let $(D_j,E_j)$ be the degree
and the escape probabilities $(d_\varnothing, v_\varnothing)$ in
$\mathcal{T}_{\mf y_j}$. Note that the vectors $(D_j,E_j)$ are
independent and identically distributed and that $Q'_N[ (D_1,E_1) \in
A] = \mathcal{P} \big[ (d_\varnothing, v_\varnothing) \in A \big
||\mathcal{T}| = \infty \big]$. In particular, by Corollary
\ref{c:esc} and Schwarz inequality the last two conditions in
\eqref{e:c4} are fulfilled.

Let $\bb A_N$ be the event ``the graphs $(\mf x_i, B(\mf x_i, \gamma
\log N))$, $1\le i\le K$, are isometric to the graphs $(\mf y_i, B(\mf
x_i, \gamma \log N ))$, $1\le i\le K$''. In view of
Corollary~\ref{l:approxesc1}, on the set $\bb A_N$, the first two
condition in \eqref{e:c4} are fulfilled. To conclude the proof of
condition \eqref{e:c4} it remains to show that
\begin{equation}
\label{e:allisom}
\lim_{N\to\infty} \bb P \big[ \bb A_N^c \big]\;=\; 0\;.
\end{equation}

We define six sets $\Sigma_{N,j}$, $0\le j\le 5$, such that
$\cap_{0\le j\le 5} \Sigma_{N,j} \subset \bb A_N$ and then prove that
each of this set has asymptotic full measure. Recall that $b$ and
$\gamma$ satisfy the assumptions of Corollary \ref{c:manyballs} and
let $\Sigma_{N,0} = \ms {B}$. Since $\{(\mf x_i)_{i = 1}^K = (\mf
y_i)_{i = 1}^K \} \cap \ms B \subset \bb A_N$, we need to impose
further restriction to guarantee that $\mf x_i = \mf y_i$, $1\le i\le
K$.

Let $\Sigma_{N,1} = \{\diam (\mathcal{C}_{\max}) \ge \gamma \log N\}$,
let $\Sigma_{N,2} = \{|\{z_1,\dots, z_{\log N}\} \cap
\mathcal{C}_{\max}| \ge K\}$ and let $\Sigma_{N,3}$ be the event
``every three $\mc T_i$, $1\le i\le \log N$, with diameter greater or
equal to $\gamma \log N$ survives''. On $\Sigma_{N,0} \cap
\Sigma_{N,1} \cap \Sigma_{N,2}\cap \Sigma_{N,3}$, the graphs $(\mf
x_i, B(\mf x_i, \gamma \log N))$, $1\le i\le K$, are coupled to
infinite trees.

It remains to guarantee that there is no infinite tree coupled with a
graph $(z_i, B(z_i, \gamma \log N))$ whose root $z_i$ does not belong
to $\mc C_{\max}$. Let $\Sigma_{N,4}$ be the event ``Every tree $\mc
T_i$, $1\le i\le \log N$, with diameter greater of equal than $\gamma
\log N$ has at least $N^\delta$ elements among the first $\gamma \log
N$ generations'', and let $\Sigma_{N,5}$ be the event ``Every
connected subset of $\ms V_N$ with more than $N^\delta$ elements is
contained in $\mc C_{\max}$''. On $\Sigma_{N,0} \cap \Sigma_{N,4} \cap
\Sigma_{N,5}$, all infinite trees $\mc T_i$, $1\le i\le \log N$, are
coupled with graphs whose root belongs to $\mc C_{\max}$.

Putting together the previous assertions, we get that $\cap_{0\le j\le
  5} \Sigma_{N,j} \subset \bb A_N$, as claimed. We next show that
each event introduced above has asymptotic full probability. By
Corollary \ref{c:manyballs}, $\bb P[\Sigma_{N,0}^c]$ vanishes, by
Theorem \ref{e:unique} and by Corollary \ref{c:volume} $\bb
P[\Sigma_{N,1}^c]$, and by Theorem~\ref{e:unique}, $\bb
P[\Sigma_{N,2}^c]$ vanishes.  By Corollary \ref{c:bigB}, $\bb
P[\Sigma_{N,4}^c]$ vanishes for some $\delta>0$, and by Theorem
\ref{e:unique} $\bb P[\Sigma_{N,5}^c]$ vanishes.  Finally, by
Corollary \ref{c:bigB}, there exists $\delta = \delta (\gamma,
\lambda)>0$ with the following property. A tree which has diameter
$\gamma \log N$ has at least $N^\delta$ elements at generation $\gamma
\log N$ with probability converging to $1$. Since from each element of
the generation $\gamma \log N$ descends an independent super-critical
tree which has positive probability to survive, $\bb
P[\Sigma_{N,3}^c]$ vanishes
\end{proof}

\end{document}